\pdfoutput=1
\ifdefined\AMSVersion
\documentclass[11pt,reqno]{amsart}
\else
\documentclass[11pt,a4paper]{article}
\usepackage[margin=1in]{geometry}
\fi

\usepackage[T1]{fontenc}
\usepackage{lmodern}
\usepackage{microtype}
\usepackage{amsmath,amssymb,mathtools}
\usepackage{amsthm}
\usepackage{aliascnt}
\usepackage{enumitem}
\usepackage{booktabs}
\usepackage{longtable}
\usepackage{array}
\usepackage{tikz}
\usepackage{graphicx}
\usepackage{placeins}
\usepackage{xcolor}
\usepackage{hyperref}
\usepackage[nameinlink,capitalize,noabbrev]{cleveref}

\hypersetup{
  pdftitle={Spanning-Tree Maximization Forces Isoperimetry in the Square Lattice},
  pdfauthor={Jiechen Zhang},
  colorlinks=true,
  linkcolor=blue!45!black,
  citecolor=blue!45!black,
  urlcolor=blue!45!black
}

\setlist[itemize]{leftmargin=2em,itemsep=0.2em,topsep=0.4em}
\setlist[enumerate]{leftmargin=2.2em,itemsep=0.2em,topsep=0.4em}

\newtheorem{theorem}{Theorem}[section]

\newaliascnt{proposition}{theorem}
\newtheorem{proposition}[proposition]{Proposition}
\aliascntresetthe{proposition}
\newaliascnt{lemma}{theorem}
\newtheorem{lemma}[lemma]{Lemma}
\aliascntresetthe{lemma}
\newaliascnt{corollary}{theorem}
\newtheorem{corollary}[corollary]{Corollary}
\aliascntresetthe{corollary}
\newaliascnt{claim}{theorem}

\aliascntresetthe{claim}

\theoremstyle{definition}
\newaliascnt{definition}{theorem}
\newtheorem{definition}[definition]{Definition}
\aliascntresetthe{definition}
\newaliascnt{remark}{theorem}

\aliascntresetthe{remark}
\newaliascnt{example}{theorem}

\aliascntresetthe{example}

\crefname{theorem}{Theorem}{Theorems}
\crefname{proposition}{Proposition}{Propositions}
\crefname{lemma}{Lemma}{Lemmas}
\crefname{corollary}{Corollary}{Corollaries}
\crefname{definition}{Definition}{Definitions}
\crefname{remark}{Remark}{Remarks}
\crefname{example}{Example}{Examples}

\newcommand{\Z}{\mathbb Z}

\newcommand{\Latt}{\mathcal L}
\newcommand{\Cat}{G_{\mathrm{Cat}}}
\newcommand{\asinh}{\operatorname{arsinh}}
\newcommand{\tr}{\operatorname{tr}}

\newcommand{\remop}{\operatorname{rem}}

\newcommand{\symdiff}{\triangle}
\newcommand{\1}{\mathbf 1}
\newcommand{\Rect}[2]{R_{#1,#2}}
\newcommand{\Square}[1]{Q_{#1}}
\newcommand{\Energy}{\mathcal E}
\newcommand{\An}{\mathcal A_n}
\DeclareMathOperator{\Max}{Max}

\numberwithin{equation}{section}

\ifdefined\AMSVersion
\title[Spanning-Tree Maximization and Isoperimetry]{Spanning-Tree Maximization Forces Isoperimetry in the Square Lattice}
\author{Jiechen Zhang}
\thanks{EPFL, Lausanne, Switzerland. Email: \textsf{\{jiechen.zhang\}@epfl.ch}. The results of this paper have been formalized in Lean; see the \href{https://github.com/novou/spanning-tree-isoperimetry-lean}{GitHub repository}.}
\else
\title{Spanning-Tree Maximization Forces\\Isoperimetry in the Square Lattice}
\author{Jiechen Zhang\thanks{EPFL, Lausanne, Switzerland. Email: \textsf{\{jiechen.zhang\}@epfl.ch}. The results of this paper have been formalized in Lean; see the \href{https://github.com/novou/spanning-tree-isoperimetry-lean}{GitHub repository}.}}
\fi
\date{}

\begin{document}
\ifdefined\AMSVersion\else\maketitle\fi

\begin{abstract}
We determine which connected sets of $N$ vertices in the square lattice support the most spanning trees, using all lattice edges between the chosen vertices. Every maximizer minimizes the number of edges joining the set to the rest of the lattice. Its shape is a rectangle with at most two straight strips of vertices removed at opposite corners. For $N\ge18$, the maximizer is unique up to translation, rotation, and reflection: an explicitly determined rectangle with one possibly empty corner strip removed along its longer side. The proof combines comparisons of Laplacian determinants with geometric comparisons of corner deletions. For $N=n^2$, we prove a sharp stability inequality that bounds the symmetric difference from a translated $n\times n$ square in terms of the loss in logarithmic tree count. Thus, nearly optimal tree counts force the vertex set to be close to a square. We also derive an exact product formula and a recurrence for the number of spanning trees in rectangles with a single corner strip removed.
\end{abstract}
\ifdefined\AMSVersion\maketitle\fi

\setcounter{tocdepth}{1}
\hypersetup{bookmarksdepth=2}
\tableofcontents
\newpage

\section{Introduction}
\label{sec:introduction}

Which connected sets of $N$ vertices in the square lattice support the most spanning trees? The edge-isoperimetric problem favors sets with small boundary. The spanning-tree count depends on the cycles throughout a graph: Kirchhoff's theorem expresses it as a Laplacian determinant. The question is whether this global quantity selects the same compact geometry.

Let $\Latt$ be the nearest-neighbor graph on $\Z^2$, and let $\tau(S)$ count the spanning trees of the induced graph $\Latt[S]$. For $N\ge1$, set
\[
M_N=\max\{\tau(S):S\subset\Z^2,\ |S|=N,\ \Latt[S]\text{ connected}\}.
\]
The maximum exists because a connected $N$-vertex set lies, up to translation, in a box with side lengths at most $N$. Adding a missing lattice edge strictly increases the count, so every maximizing graph is induced. The optimization ranges over all connected vertex sets, including those with holes or irregular boundaries.

The suggestion that square shapes should maximize the spanning-tree count appears in a 2010 MathOverflow comment by Swanepoel \cite{Swanepoel2010MO}. Procaccia and Tucker-Foltz \cite[Section 1.2]{ProcacciaTuckerFoltz2022} record the conjecture that the square maximizes the count among grid subgraphs with the same number of vertices. Earlier versions of this work prove rectangular balancing \cite{ZhangRect2026} and square maximality \cite{ZhangSquare2026}; the latter has a Lean formalization registered in Palomar. Here the classification covers every cardinality.

\subsection{Geometry of the maximizers}
For positive integers $r,s$, write $\Rect r s=\{0,\ldots,r-1\}\times\{0,\ldots,s-1\}$. The minimal rectangular hull of a domain is the smallest translate of such a rectangle containing it; side lengths count vertices.

\begin{theorem}[Isoperimetry of spanning-tree maximizers]
\label{thm:intro-isoperimetric}
Every $N$-vertex spanning-tree maximizer $S$ is edge-isoperimetric:
\[
|E(\Latt[S])|=2N-\lceil2\sqrt N\rceil.
\]
\end{theorem}

The maximizers share a geometric description for all $N$: each is its minimal rectangular hull, with $r+s=\lceil2\sqrt N\rceil$, minus at most two straight strips at opposite corners (\cref{thm:all-N-classification}). The proof first forces interval rows and columns with staircase-shaped corner deletions, called Ferrers defects (\cref{thm:isoperimetric-rigidity}), as in minimum-perimeter polyominoes \cite{Kurz2008}. The tree count then distinguishes among these shapes.

The canonical candidate uses the most elongated rectangle with this minimum side sum and enough room for $N$ vertices. With $L=\lceil2\sqrt N\rceil$, its shorter side $r$ is the smallest positive integer satisfying $r(L-r)\ge N$. Thus
\[
r=\left\lceil\frac{L-\sqrt{L^2-4N}}2\right\rceil,\qquad s=L-r,\qquad k=rs-N.
\]
Then $r\le s$ and $0\le k<r$. Removing a possibly empty corner strip of length $k$ along the longer side gives
\[
C_N=\Rect r s\setminus\{(0,j):0\le j<k\}.
\]

\begin{theorem}
\label{thm:intro-classification}
For every $N\ge1$, $M_N=\tau(C_N)$ except at $N=7,14,17$, where
\[
M_7=16,\qquad M_{14}=8960,\qquad M_{17}=112456.
\]
Up to translation, rotation, and reflection, nonuniqueness occurs precisely at $N=3,8,10,13$. In particular, for every $N\ge18$, the unique maximizer up to lattice isometry is $C_N$.
\end{theorem}

Two nonempty corner strips occur only at $N=7,10,13,14,17$: they beat $C_N$ at $7,14,17$ and tie it at $10,13$. \Cref{tab:exceptions} lists all exceptional maximizers, and \cref{fig:domains} illustrates the seven-vertex case. Uniform estimates and polarization give the single-strip form for $N\ge18$; the graph recurrences in Appendix~\ref{app:corner-certificates} determine the small exceptions.

\begin{figure}[htbp]
\centering
\begin{tikzpicture}[scale=.42,baseline=(current bounding box.south)]
\draw[gray] (0,2) -- (1,2);
\draw[gray] (0,2) -- (0,3);
\draw[gray] (0,3) -- (1,3);
\draw[gray] (0,3) -- (0,4);
\draw[gray] (0,4) -- (1,4);
\draw[gray] (0,4) -- (0,5);
\draw[gray] (0,5) -- (1,5);
\draw[gray] (1,0) -- (2,0);
\draw[gray] (1,0) -- (1,1);
\draw[gray] (1,1) -- (2,1);
\draw[gray] (1,1) -- (1,2);
\draw[gray] (1,2) -- (2,2);
\draw[gray] (1,2) -- (1,3);
\draw[gray] (1,3) -- (2,3);
\draw[gray] (1,3) -- (1,4);
\draw[gray] (1,4) -- (2,4);
\draw[gray] (1,4) -- (1,5);
\draw[gray] (1,5) -- (2,5);
\draw[gray] (2,0) -- (3,0);
\draw[gray] (2,0) -- (2,1);
\draw[gray] (2,1) -- (3,1);
\draw[gray] (2,1) -- (2,2);
\draw[gray] (2,2) -- (3,2);
\draw[gray] (2,2) -- (2,3);
\draw[gray] (2,3) -- (3,3);
\draw[gray] (2,3) -- (2,4);
\draw[gray] (2,4) -- (3,4);
\draw[gray] (2,4) -- (2,5);
\draw[gray] (2,5) -- (3,5);
\draw[gray] (3,0) -- (3,1);
\draw[gray] (3,1) -- (3,2);
\draw[gray] (3,2) -- (3,3);
\draw[gray] (3,3) -- (3,4);
\draw[gray] (3,4) -- (3,5);
\fill (0,2) circle (2pt);
\fill (0,3) circle (2pt);
\fill (0,4) circle (2pt);
\fill (0,5) circle (2pt);
\fill (1,0) circle (2pt);
\fill (1,1) circle (2pt);
\fill (1,2) circle (2pt);
\fill (1,3) circle (2pt);
\fill (1,4) circle (2pt);
\fill (1,5) circle (2pt);
\fill (2,0) circle (2pt);
\fill (2,1) circle (2pt);
\fill (2,2) circle (2pt);
\fill (2,3) circle (2pt);
\fill (2,4) circle (2pt);
\fill (2,5) circle (2pt);
\fill (3,0) circle (2pt);
\fill (3,1) circle (2pt);
\fill (3,2) circle (2pt);
\fill (3,3) circle (2pt);
\fill (3,4) circle (2pt);
\fill (3,5) circle (2pt);
\draw[gray] (0,0) circle (2pt);
\draw[gray] (0,1) circle (2pt);
\node[anchor=north] at (1.5,-0.65) {(a) $C_{22}$};
\end{tikzpicture}
\qquad
\begin{tikzpicture}[scale=.42,baseline=(current bounding box.south)]
\draw[gray] (0,1) -- (1,1);
\draw[gray] (0,1) -- (0,2);
\draw[gray] (0,2) -- (1,2);
\draw[gray] (1,0) -- (2,0);
\draw[gray] (1,0) -- (1,1);
\draw[gray] (1,1) -- (2,1);
\draw[gray] (1,1) -- (1,2);
\draw[gray] (2,0) -- (2,1);
\fill (0,1) circle (2pt);
\fill (0,2) circle (2pt);
\fill (1,0) circle (2pt);
\fill (1,1) circle (2pt);
\fill (1,2) circle (2pt);
\fill (2,0) circle (2pt);
\fill (2,1) circle (2pt);
\draw[gray] (0,0) circle (2pt);
\draw[gray] (2,2) circle (2pt);
\node[anchor=north] at (1.0,-0.65) {(b) $N=7$};
\end{tikzpicture}
\caption{A canonical single-strip domain and a two-corner maximizer. In (a), $C_{22}=R_{4,6}\setminus\{(0,0),(0,1)\}$. In (b), the seven-vertex maximizer is $R_{3,3}\setminus\{(0,0),(2,2)\}$. Open circles indicate deleted vertices; all edges between occupied nearest neighbors are retained.}
\label{fig:domains}
\end{figure}

\subsection{Quantitative stability}
Near-maximal tree counts also force quantitative closeness to a square. Earlier lattice stability results control symmetric difference by an edge-boundary deficit \cite[Theorem 8]{EllisEtAl2016}, with extensions to general lattice generating sets \cite{BarberEtAl2023}. Here the deficit is a loss in logarithmic tree count. For $n\ge1$, write $\Square n=\Rect n n$ and, when $|S|=n^2$, define the normalized translational asymmetry
\begin{equation}
\An(S)=\frac1{n^2}\min_{z\in\Z^2}|S\symdiff(z+\Square n)|.
\label{eq:intro-asymmetry}
\end{equation}

\begin{theorem}
\label{thm:intro-stability}
For every connected $S\subset\Z^2$ with $|S|=n^2$,
\[
\log\frac{\tau(\Square n)}{\tau(S)}
\ge\frac{\asinh(1)}4\,n\,\An(S)^2.
\]
The constant $\asinh(1)/4$ is optimal uniformly in $n$.
\end{theorem}

Nearly square rectangles show that the constant is sharp (\cref{sec:stability}). This family also gives the sharp boundary coefficient in the rectangular preliminary version \cite[Proposition 6.1]{ZhangRect2026}.

\subsection{Exact enumeration}
The corner-strip geometry also leads to an exact counting formula. For $r,s\ge2$ and $0\le k\le s-1$, define
\[
C(r,s,k)=\Rect r s\setminus\{(0,j):0\le j<k\}.
\]
Planar duality expresses the count as a rectangular factor times a leading principal minor of the boundary Green matrix. For a fixed rectangle, one scalar rational function supplies the moments of Hankel matrices with the same determinants, for every strip length $k$. A signed Euclidean chain then gives the exact product formula and recurrence in \cref{thm:scalar-product} for the entire family $C(r,s,k)$.

\subsection{The determinant comparison}
The geometric proof extends the hull comparison of \cite{ZhangSquare2026}. Compare a domain $S$ with its minimal rectangular hull $R$. After filling the interior eight-neighbor vacancy components, let $g$ count the extra occupied intervals in rows and columns, and let $c$ count those vacancy components. The hull and rectangle estimates use the same bulk coefficient $h$, with positive boundary and vacancy penalties $\alpha,\delta$:
\begin{align*}
\log\frac{\tau(R)}{\tau(S)}
&\ge h(|R|-|S|)+\alpha g+\delta c,
\\
\log\frac{\tau(\Rect r s)}{\tau(\Square n)}
&\le h(rs-n^2)-\alpha(r+s-2n),\qquad rs\ge n^2.
\end{align*}
Their exact values are fixed in \cref{sec:preliminaries}. The area terms cancel, leaving positive costs for excess hull side sum, gaps, and interior vacancies. An explicit $N$-vertex competitor controls the cost of starting a new row or column. These costs exclude every nonzero integer defect from a maximizer.

The remaining choice concerns the corner defects. Their limiting energies have nonnegative closed-walk coefficients: coefficientwise comparisons favor straight strips, while concavity favors concentrating vacancies at one corner. In a finite rectangle, walks that reach the opposite sides add a correction. Bounds for this correction, together with polarization, give the classification. The same polarization compares strip orientations and neighboring hulls; its equality cases account for the small nonuniqueness exceptions.

\subsection{Proof ingredients and related work}
Discrete compression and grid inequalities provide the edge-isoperimetric background \cite{BollobasLeader1991Grid,BollobasLeader1991Compression}. In the 1990s, Burton and Pemantle \cite{BurtonPemantle1993} used transfer-impedances to calculate local statistics and entropy, and Chung \cite{Chung1999} bounded tree counts in terms of domain size and boundary. Chung's estimates and Tapp's later bounds \cite{Tapp2021} quantify the influence of the boundary on the tree count. Kenyon's asymptotic determinant expansion \cite{Kenyon2000} identifies finer terms for a fixed rectilinear shape. The comparisons here track these effects as the shape varies and determine the exact finite maximizers. They use transfer-current identities for local vacancy estimates and Tapp's insertion bound for comparison domains.

\Cref{sec:preliminaries} fixes shared conventions. The geometric classification occupies \cref{sec:isoperimetry,sec:corner,sec:polarization,sec:classification}, followed by stability in \cref{sec:stability} and exact counting formulas in \cref{sec:counting}. The appendices give the rectangular estimates and corner comparisons.

\section{Preliminaries}
\label{sec:preliminaries}

The geometric and counting arguments use the same lattice conventions, dual determinant, and closed-walk expansion. Logarithms are natural, and empty products and determinants equal one.

\subsection{Square-lattice domains and rectangular hulls}
A \emph{lattice domain} is a finite set \(S\subset\Z^2\), identified with its induced graph \(\Latt[S]\); domains are connected unless stated otherwise. We retain the notation \(\tau(S)\), \(\Rect r s\), and the minimal rectangular hull from the introduction. Every row and column of a connected domain's hull meets the domain. We call \(S\) \emph{row convex} if its intersection with each horizontal lattice line is empty or an interval, and \emph{column convex} analogously.

A \emph{Ferrers defect} is a finite lower ideal \(M\subset\Z_{\ge0}^2\): whenever \((i,j)\in M\), every \((i',j')\) with \(0\le i'\le i\) and \(0\le j'\le j\) also belongs to \(M\). Translating and reflecting the ideal anchors it at a chosen corner of a rectangle. We describe such diagrams by their nonincreasing row lengths, as in \((3,2,1)\). A \emph{hook} consists of one row and one column meeting at the corner.

\subsection{Dual determinants and closed walks}
Kirchhoff's Matrix--Tree theorem \cite{Kirchhoff1847} gives \(\tau(G)=\det L_G^{(v)}\), where \(L_G^{(v)}\) is obtained by deleting row and column \(v\) from the graph Laplacian. For a connected planar graph, the edges complementary to a primal spanning tree form a spanning tree of the geometric dual, whose unbounded face is the root. Parallel dual edges are retained. For a rectangle, every bounded face is a unit square, so this correspondence gives a fixed-diagonal Dirichlet determinant.

For a finite set \(U\subset\Z^2\) of unit-square faces, let \(A_U\) be its nearest-neighbor adjacency matrix and wire every missing neighbor to the exterior root. The face operator is then \(4I-A_U\): each missing neighbor contributes a root edge, keeping the diagonal four even at the boundary. This representation applies to the unit-square face domains used below; a hole in a primal domain requires the separate dual treatment in \cref{sec:isoperimetry}.

For a finite connected multigraph \(D\), fix a root and orient its edges. Let \(L_D\) be the Laplacian with the root row and column removed, and let \(b_e\) be the incidence vector of edge \(e\) with its root coordinate removed. The transfer-current matrix is
\[
K_D(e,f)=b_e^{\mathsf T}L_D^{-1}b_f.
\]
It is positive semidefinite. For a uniform spanning tree \(T\) of \(D\) and a set of edges \(F\), the transfer-current identities \cite{BurtonPemantle1993} give
\[
\Pr(F\subseteq T)=\det K_D[F,F],
\qquad
\Pr(F\cap T=\varnothing)=\det\bigl(I-K_D[F,F]\bigr).
\]

For every finite \(U\subset\Z^2\), the spectral radius of \(A_U\) is strictly less than four. Hence, for \(z\ge4\),
\begin{equation}
\log\det(zI-A_U)
=|U|\log z-\sum_{\ell\ge1}\frac{\tr(A_U^\ell)}{\ell z^\ell}.
\label{eq:closed-walk-expansion}
\end{equation}
Here \(\tr(A_U^\ell)\) counts closed walks of length \(\ell\) with a specified starting vertex. A \emph{rooted closed word} is their sequence of steps in \(\{(\pm1,0),(0,\pm1)\}\), with total displacement zero. Translating a word places its starting vertex in the lattice. Thus determinant comparison reduces to counting admissible placements of these words.

\subsection{Rectangular counts}
We fix
\begin{equation*}
 h=\frac{4\Cat}{\pi},\qquad
 \alpha=\asinh(1)=\log(1+\sqrt2),\qquad
 \delta=\frac35.
\end{equation*}
Here \(\Cat\) is Catalan's constant, \(h\) is the square-lattice spanning-tree entropy, and \(\alpha\) is the boundary coefficient in the rectangular comparison. The constant \(\delta\) is a lower bound for the extra cost of an interior vacancy component in the local electrical estimate of \cref{sec:isoperimetry}. The strict inequalities used later are
\begin{equation}
2(h-\alpha)<\delta<\alpha.
\label{eq:strict-constants}
\end{equation}
Appendix~\ref{app:corner-certificates} verifies these inequalities and the choice of \(\delta\) by elementary rational bounds.

The standard rectangular product formula is
\begin{equation}
\tau(\Rect r s)=
\prod_{i=1}^{r-1}\prod_{j=1}^{s-1}
\left(4-2\cos\frac{\pi i}{r}-2\cos\frac{\pi j}{s}\right).
\label{eq:rectangle-product}
\end{equation}
The normalized form needed in the structural theorem is proved in Appendix~\ref{app:rectangles}.

\section{Spanning-tree maximality forces isoperimetry}
\label{sec:isoperimetry}

Comparing a domain with its rectangular hull charges separately for missing vertices, row--column gaps, and interior vacancy components. A second comparison, with a nearly square domain of the same cardinality, cancels the bulk cost and leaves positive charges for the geometric defects. Their integrality then forces every maximizer to be row and column convex on a hull of minimum side sum.

\subsection{Rectangular comparison}

\begin{proposition}
\label{prop:rectangular-comparison}
For positive integers \(r,s,n\) with \(rs\ge n^2\),
\begin{equation*}
\log\tau(\Rect r s)-\log\tau(\Square n)
\le h(rs-n^2)-\alpha(r+s-2n).
\end{equation*}
Equality holds if and only if \(r=s=n\).
\end{proposition}

Appendix~\ref{app:rectangles} derives this comparison from the hyperbolic form of \eqref{eq:rectangle-product}.

\subsection{Vacancy decomposition}
Let \(S\) be connected with minimal hull \(R=\Rect r s\), after translation. Decompose \(R\setminus S\) into eight-neighbor components: let \(V\) collect those meeting the perimeter, and let \(J_1,\ldots,J_c\) be the interior components. Filling the latter gives \(S^+=R\setminus V\).

Let \(g_x\) and \(g_y\) count the internal vacancy runs of \(V\) in rows and columns, respectively, and put \(g=g_x+g_y\). Thus \(S^+\) has \(s+g_x\) occupied row intervals and \(r+g_y\) occupied column intervals.

Figure~\ref{fig:hull-defects} illustrates the decomposition and the gaps counted in \(S^+\).
\begin{figure}[htbp]
\centering
\definecolor{hull-interior}{RGB}{155,78,30}
\definecolor{hull-gap}{RGB}{45,87,126}
\begin{tikzpicture}[scale=.58,baseline=(current bounding box.south)]
\draw[gray,dashed] (-.45,-.45) rectangle (5.45,4.45);
\node[font=\small] at (2.5,4.85) {$R_{6,5}$};
\draw[gray] (0,0) -- (1,0);
\draw[gray] (0,0) -- (0,1);
\draw[gray] (0,1) -- (1,1);
\draw[gray] (0,1) -- (0,2);
\draw[gray] (0,2) -- (1,2);
\draw[gray] (0,2) -- (0,3);
\draw[gray] (0,3) -- (1,3);
\draw[gray] (0,3) -- (0,4);
\draw[gray] (0,4) -- (1,4);
\draw[gray] (1,0) -- (1,1);
\draw[gray] (1,1) -- (1,2);
\draw[gray] (1,2) -- (1,3);
\draw[gray] (1,3) -- (2,3);
\draw[gray] (1,3) -- (1,4);
\draw[gray] (1,4) -- (2,4);
\draw[gray] (2,3) -- (3,3);
\draw[gray] (2,3) -- (2,4);
\draw[gray] (2,4) -- (3,4);
\draw[gray] (3,0) -- (4,0);
\draw[gray] (3,0) -- (3,1);
\draw[gray] (3,1) -- (4,1);
\draw[gray] (3,1) -- (3,2);
\draw[gray] (3,2) -- (4,2);
\draw[gray] (3,2) -- (3,3);
\draw[gray] (3,3) -- (3,4);
\draw[gray] (3,4) -- (4,4);
\draw[gray] (4,0) -- (5,0);
\draw[gray] (4,0) -- (4,1);
\draw[gray] (4,1) -- (5,1);
\draw[gray] (4,1) -- (4,2);
\draw[gray] (4,2) -- (5,2);
\draw[gray] (4,4) -- (5,4);
\draw[gray] (5,0) -- (5,1);
\draw[gray] (5,1) -- (5,2);
\draw[gray] (5,2) -- (5,3);
\draw[gray] (5,3) -- (5,4);
\fill (0,0) circle (2pt);
\fill (0,1) circle (2pt);
\fill (0,2) circle (2pt);
\fill (0,3) circle (2pt);
\fill (0,4) circle (2pt);
\fill (1,0) circle (2pt);
\fill (1,1) circle (2pt);
\fill (1,2) circle (2pt);
\fill (1,3) circle (2pt);
\fill (1,4) circle (2pt);
\fill (2,3) circle (2pt);
\fill (2,4) circle (2pt);
\fill (3,0) circle (2pt);
\fill (3,1) circle (2pt);
\fill (3,2) circle (2pt);
\fill (3,3) circle (2pt);
\fill (3,4) circle (2pt);
\fill (4,0) circle (2pt);
\fill (4,1) circle (2pt);
\fill (4,2) circle (2pt);
\fill (4,4) circle (2pt);
\fill (5,0) circle (2pt);
\fill (5,1) circle (2pt);
\fill (5,2) circle (2pt);
\fill (5,3) circle (2pt);
\fill (5,4) circle (2pt);
\draw[gray,fill=white] (2,0) circle (2pt);
\draw[gray,fill=white] (2,1) circle (2pt);
\draw[gray,fill=white] (2,2) circle (2pt);
\draw[hull-interior,line width=.7pt,fill=white] (4,3) circle (2.3pt);
\draw[hull-interior,thin] (4.15,3.12) -- (5.75,3.70);
\node[anchor=west,font=\small,text=hull-interior] at (5.80,3.70) {$J_1$};
\draw[gray,thin] (2,-.13) -- (2,-.83);
\node[anchor=north,font=\small,text=gray] at (2,-.88) {$V$};
\node[anchor=north,font=\small] at (2.5,-1.55) {(a) $S$, \quad $|S|=26$};
\end{tikzpicture}
\qquad
\begin{tikzpicture}[scale=.58,baseline=(current bounding box.south)]
\foreach \yy in {0,1,2}
  \fill[hull-gap!8] (-.20,\yy-.16) rectangle (5.20,\yy+.16);
\draw[gray,dashed] (-.45,-.45) rectangle (5.45,4.45);
\node[font=\small] at (2.5,4.85) {$R_{6,5}$};
\draw[gray] (0,0) -- (1,0);
\draw[gray] (0,0) -- (0,1);
\draw[gray] (0,1) -- (1,1);
\draw[gray] (0,1) -- (0,2);
\draw[gray] (0,2) -- (1,2);
\draw[gray] (0,2) -- (0,3);
\draw[gray] (0,3) -- (1,3);
\draw[gray] (0,3) -- (0,4);
\draw[gray] (0,4) -- (1,4);
\draw[gray] (1,0) -- (1,1);
\draw[gray] (1,1) -- (1,2);
\draw[gray] (1,2) -- (1,3);
\draw[gray] (1,3) -- (2,3);
\draw[gray] (1,3) -- (1,4);
\draw[gray] (1,4) -- (2,4);
\draw[gray] (2,3) -- (3,3);
\draw[gray] (2,3) -- (2,4);
\draw[gray] (2,4) -- (3,4);
\draw[gray] (3,0) -- (4,0);
\draw[gray] (3,0) -- (3,1);
\draw[gray] (3,1) -- (4,1);
\draw[gray] (3,1) -- (3,2);
\draw[gray] (3,2) -- (4,2);
\draw[gray] (3,2) -- (3,3);
\draw[gray] (3,3) -- (4,3);
\draw[gray] (3,3) -- (3,4);
\draw[gray] (3,4) -- (4,4);
\draw[gray] (4,0) -- (5,0);
\draw[gray] (4,0) -- (4,1);
\draw[gray] (4,1) -- (5,1);
\draw[gray] (4,1) -- (4,2);
\draw[gray] (4,2) -- (5,2);
\draw[gray] (4,2) -- (4,3);
\draw[gray] (4,3) -- (5,3);
\draw[gray] (4,3) -- (4,4);
\draw[gray] (4,4) -- (5,4);
\draw[gray] (5,0) -- (5,1);
\draw[gray] (5,1) -- (5,2);
\draw[gray] (5,2) -- (5,3);
\draw[gray] (5,3) -- (5,4);
\fill (0,0) circle (2pt);
\fill (0,1) circle (2pt);
\fill (0,2) circle (2pt);
\fill (0,3) circle (2pt);
\fill (0,4) circle (2pt);
\fill (1,0) circle (2pt);
\fill (1,1) circle (2pt);
\fill (1,2) circle (2pt);
\fill (1,3) circle (2pt);
\fill (1,4) circle (2pt);
\fill (2,3) circle (2pt);
\fill (2,4) circle (2pt);
\fill (3,0) circle (2pt);
\fill (3,1) circle (2pt);
\fill (3,2) circle (2pt);
\fill (3,3) circle (2pt);
\fill (3,4) circle (2pt);
\fill (4,0) circle (2pt);
\fill (4,1) circle (2pt);
\fill (4,2) circle (2pt);
\fill[hull-interior] (4,3) circle (2.3pt);
\fill (4,4) circle (2pt);
\fill (5,0) circle (2pt);
\fill (5,1) circle (2pt);
\fill (5,2) circle (2pt);
\fill (5,3) circle (2pt);
\fill (5,4) circle (2pt);
\draw[gray,fill=white] (2,0) circle (2pt);
\draw[gray,fill=white] (2,1) circle (2pt);
\draw[gray,fill=white] (2,2) circle (2pt);
\draw[hull-interior,thin] (4.15,3.12) -- (5.75,3.70);
\node[anchor=west,font=\small,text=hull-interior] at (5.80,3.70) {$J_1$ filled};
\foreach \yy in {0,1,2}
  \draw[hull-gap,thin] (5.48,\yy) -- (5.88,\yy);
\draw[hull-gap,thin] (5.88,0) -- (5.88,2);
\node[anchor=west,font=\small,text=hull-gap] at (6.05,1) {$g_x=3$};
\node[anchor=north,font=\small] at (2.5,-1.55) {(b) $S^+$, \quad $|S^+|=27$};
\end{tikzpicture}
\caption{Vacancy decomposition for the enhanced hull inequality. The connected set $S\subset R_{6,5}$ has $|S|=26$, boundary component $V=\{(2,0),(2,1),(2,2)\}$, and interior eight-neighbor component $J_1=\{(4,3)\}$. Open circles mark vacancies; all occupied nearest-neighbor edges are retained. Filling $J_1$ gives the auxiliary set $S^+$ with $|S^+|=27$. Its three shaded rows give $g_x=3$, $g_y=0$, hence $g=3$. Since $S$ has one interior component, $c=1$; \cref{cor:enhanced-hull} therefore charges $4h+3\alpha+\delta$.}
\label{fig:hull-defects}
\end{figure}

The boundary and interior costs add in the following hull estimate.

\begin{corollary}
\label{cor:enhanced-hull}
For every connected \(S\subseteq R_{r,s}\) with minimal hull \(R_{r,s}\),
\begin{equation*}
\log\frac{\tau(\Rect r s)}{\tau(S)}
\ge h(rs-|S|)+\alpha g+\delta c.
\end{equation*}
\end{corollary}

To separate these contributions, factor the deletion ratio in the rooted dual.

\begin{lemma}
\label{lem:dual-factorization}
With the notation above,
\begin{equation}
\frac{\tau(S)}{\tau(R)}
\le
\frac{\tau(S^+)}{\tau(R)}
\prod_{i=1}^{c}\frac{\tau(R\setminus J_i)}{\tau(R)}.
\label{eq:dual-factorization}
\end{equation}
\end{lemma}

\begin{proof}
If the hull has a side of length one, connectedness forces \(S=R\), and the inequality is equality. Otherwise let \(D\) be the rooted geometric dual of \(R\). For a vacancy set \(W\) with \(R\setminus W\) connected, let \(H_W\) consist of the dual edges crossing primal edges incident with \(W\), and choose a spanning forest \(F_W\) of \(H_W\). Contracting \(F_W\) identifies the same dual vertices as contracting all edges of \(H_W\). Planar deletion--contraction duality therefore gives, for a uniform spanning tree \(T\) of \(D\),
\[
\Pr(F_W\subseteq T)=\frac{\tau(R\setminus W)}{\tau(R)}.
\]
Each nearest-neighbor component of \(R\setminus S\) has an edge to \(S\). Filling whole eight-neighbor vacancy components therefore preserves connectivity, so \(S^+\) and every \(R\setminus J_i\) are connected. Apply the identity to \(W=V,J_1,\ldots,J_c\). Distinct eight-neighbor vacancy components affect disjoint bounded-face sets, and only \(H_V\) meets the exterior root. Thus their forests are disjoint, and their union \(F\) is a spanning forest for the full deletion. The same identity and the transfer-current formula give
\[
\frac{\tau(S)}{\tau(R)}=\det K_D[F,F]
\le\det K_D[F_V,F_V]\prod_{i=1}^c\det K_D[F_{J_i},F_{J_i}],
\]
where the inequality is Hadamard--Fischer. Identifying each factor by the displayed deletion ratio proves \eqref{eq:dual-factorization}.
\end{proof}

The interior eight-neighbor components need not be nearest-neighbor holes; row and column convexity excludes the latter.

\subsection{Boundary vacancies and row--column gaps}
A closed walk affected by boundary deletion determines a vertex window meeting \(V\). The number of such windows is controlled by vacant volume and row--column gaps.

\begin{lemma}
\label{lem:window-charging}
For every pair of nonnegative integers \(p,q\), let \(W_{p+1,q+1}(V)\) denote the number of placements of a \((p+1)\)-by-\((q+1)\) vertex window in \(R\) that meets \(V\).  Then
\begin{equation}
W_{p+1,q+1}(V)
\le |V|+p\,g_x+q\,g_y.
\label{eq:window-charging}
\end{equation}
\end{lemma}

\begin{proof}
Fix a horizontal window of \(p+1\) columns with left endpoint \(x\). Let \(Y_x\) be the rows in which it contains a vacancy, and let \(i(Y_x)\) count the interval components of \(Y_x\) avoiding the top and bottom rows of \(R\). We prove the vertical bound
\begin{equation}
\sum_x i(Y_x)\le g_y.
\label{eq:cross-window-injection}
\end{equation}

Form a graph whose vertices are the maximal vertical vacancy runs of \(V\). Join runs in adjacent columns when their row intervals intersect or have distance one; this includes the diagonal contacts in the eight-neighbor decomposition. Each graph component therefore reaches the perimeter. Choose a spanning tree rooted at a run meeting the top or bottom boundary when one exists, and otherwise at a side-boundary run. Mark the runs meeting the top or bottom boundary.

Restrict the run graph to a window. Each restricted component has an interval row union, so \(i(Y_x)\) is at most the number of components containing no marked run. In each such component choose a run of minimum tree depth, breaking ties in a fixed way.

For a chosen nonroot run, its parent lies outside the window in an adjacent column: a parent inside would belong to the same component and have smaller depth. The parent edge identifies the end of the window containing the chosen run, and hence the window's start. For an unmarked side root, the boundary column identifies that start. Thus each chosen run determines both its window and its component. This injects the unmarked window components into the \(g_y\) unmarked vertical runs and proves \eqref{eq:cross-window-injection}.

The horizontal bound is symmetric. For a vacancy set \(B\) in a fixed integer interval, the number of contained windows of \(p+1\) vertices meeting \(B\) is at most \(|B|+p\,i(B)\), where \(i(B)\) counts the components of \(B\) avoiding the endpoints. Successive horizontal and vertical dilation, together with \eqref{eq:cross-window-injection}, gives \eqref{eq:window-charging}.
\end{proof}

In the dual determinant, the window bound converts each row or column gap into a loss with coefficient \(\alpha\).

\begin{proposition}
\label{prop:gap-cost}
With \(S^+=R\setminus V\),
\begin{equation}
\log\frac{\tau(R)}{\tau(S^+)}
\ge h|V|+\alpha g.
\label{eq:boundary-gap-cost}
\end{equation}
\end{proposition}

\begin{proof}
Let \(\Phi=[0,r-2]\times[0,s-2]\) be the bounded face set of \(R\).  Mark a face when one of its corners lies in \(V\), let \(M\) be the marked set, and put \(U=\Phi\setminus M\).  Rooted planar duality gives
\[
\tau(R)=\det(4I-A_\Phi),
\qquad
\tau(S^+)=\det(4I-A_U).
\]
The occupied-run edge count and Euler's formula give the exact identity
\begin{equation}
|M|=|V|+g_x+g_y.
\label{eq:marked-face-count}
\end{equation}

For a rooted closed nearest-neighbor word \(w\), write \(R_x(w)\) and \(R_y(w)\) for its coordinate spans.  A translated walk meeting a marked face determines a primal vertex window of sizes \(R_x(w)+2\) and \(R_y(w)+2\).  By \cref{lem:window-charging}, the number of affected translations is at most
\[
|V|+g_x(R_x(w)+1)+g_y(R_y(w)+1).
\]
Inserting this into the trace expansion \eqref{eq:closed-walk-expansion} reduces the calculation to the two identities
\begin{align}
\sum_{\ell\ge1}\frac{\#\mathcal C_\ell}{\ell4^\ell}
&=\log4-h,
\label{eq:bulk-walk-sum}\\
\sum_{\ell\ge1}\frac{\sum_{w\in\mathcal C_\ell}(R_x(w)+1)}{\ell4^\ell}
&=\log4-\alpha,
\label{eq:span-walk-sum}
\end{align}
where \(\mathcal C_\ell\) is the set of rooted closed words of length \(\ell\). The first is the square-lattice entropy identity.  For the second, reflect a one-dimensional bridge at its first visit to a prescribed nonzero level; after interleaving horizontal and vertical steps, the resulting generating series is the even part of
\[
\sum_{\ell\ge1}\binom{2\ell}{\ell}\frac{t^\ell}{\ell}
=2\log\frac{2}{1+\sqrt{1-4t}}
\]
at \(t=1/4\).  Combining \eqref{eq:marked-face-count}--\eqref{eq:span-walk-sum} gives \eqref{eq:boundary-gap-cost}.
\end{proof}

\subsection{Interior vacancy components}
An interior vacancy component contracts a connected block in the dual. Its spanning-tree count supplies the bulk charge, and two separated edge pairs give a fixed additional loss.

\begin{lemma}
\label{lem:local-estimates}
\begin{enumerate}
\item If \(H\) is a connected union of complete square-lattice unit-square boundaries containing \(f\ge1\) unit squares, then
\begin{equation}
\tau(H)\ge4e^{h(f-1)}.
\label{eq:square-union-bound}
\end{equation}
\item Let \(D\) be the wired geometric dual of a rectangle and let \(e_1,e_2\) be two perpendicular bounded-to-bounded edges incident with the same bounded vertex.  If \(T\) is a uniform spanning tree of \(D\), then
\begin{equation}
\Pr(e_1,e_2\notin T)>\frac3{16}.
\label{eq:two-edge-avoidance}
\end{equation}
\end{enumerate}
\end{lemma}

\begin{proof}
For \eqref{eq:square-union-bound}, use the grid insertion estimate of \cite[Proposition 10]{Tapp2021}: in the lexicographic sweep, an insertion closing a complete unit square has multiplicative gain at least \(e^h\). For the leftmost square in the highest row of squares, the first square-closing insertion has a unique pre-existing path between its two old neighbors, of length two, and hence multiplier four. Use the product of component tree counts when an intermediate prefix is disconnected.

For \eqref{eq:two-edge-avoidance}, compare the two terminal currents with those in a free \(5\times5\) vertex patch centered at the common endpoint. Its transfer-current matrix \(K_5\) satisfies \(\det(I-K_5)>3/16\), as the rational calculation in Appendix~\ref{app:corner-certificates} shows. Contract the part of the patch lying outside the bounded dual rectangle to the exterior root, then add the remaining network. Rayleigh monotonicity gives \(K_D\preceq K_5\) in Loewner order for the two incidence currents. Therefore
\[
\Pr(e_1,e_2\notin T)=\det(I-K_D)
\ge\det(I-K_5)>\frac3{16}.
\qedhere
\]
\end{proof}

\begin{proposition}
\label{prop:interior-surcharge}
For every interior eight-neighbor vacancy component \(J\),
\begin{equation}
\log\frac{\tau(R)}{\tau(R\setminus J)}
\ge h|J|+\delta.
\label{eq:interior-surcharge}
\end{equation}
\end{proposition}

\begin{proof}
Let \(H_J\) be the dual graph formed by the edges dual to primal edges incident with \(J\). It is a connected union of complete unit-square boundaries containing at least \(|J|\) complete squares, so \cref{lem:local-estimates} gives \(\tau(H_J)\ge4e^{h(|J|-1)}\).

The incident edges of \(H_J\) are north and east at a vertex minimizing \(x+y\), and south and west at a vertex maximizing it. These pairs are disjoint, and every spanning tree of \(H_J\) must meet both. A uniform spanning tree \(T\) of the wired dual meets each pair with probability at most \(13/16\), by \eqref{eq:two-edge-avoidance}. Hadamard--Fischer bounds the probability of avoiding both pairs by the product of their avoidance probabilities; inclusion--exclusion gives the corresponding product bound for meeting both. Hence
\[
\Pr(T\cap E(H_J)\text{ spans }H_J)\le(13/16)^2.
\]
For each spanning tree \(F\) of \(H_J\), planar duality and contraction give \(\Pr(F\subseteq T)=\tau(R\setminus J)/\tau(R)\). These events are pairwise disjoint as \(F\) varies, so
\[
\tau(H_J)\frac{\tau(R\setminus J)}{\tau(R)}
\le(13/16)^2.
\]
Since \(\log6-h>\delta\) by Appendix~\ref{app:corner-certificates}, the lower bound for \(\tau(H_J)\) gives
\[
\frac{\tau(R\setminus J)}{\tau(R)}
\le\frac{(13/16)^2}{4e^{h(|J|-1)}}
<\frac{e^{-h(|J|-1)}}6
<e^{-h|J|-\delta}.
\qedhere
\]
\end{proof}

\begin{proof}[Proof of \cref{cor:enhanced-hull}]
Combine \cref{lem:dual-factorization,prop:gap-cost,prop:interior-surcharge} and use
\[
|V|+\sum_i|J_i|=rs-|S|.
\qedhere
\]
\end{proof}

\subsection{Eliminating integer defects}
An explicit comparison domain nearly attains the bulk term in the hull bound. For \(N\ge1\), write
\[
N=n^2+q,
\qquad n=\lfloor\sqrt N\rfloor,
\qquad 0\le q\le2n,
\]
and set \(j=0,1,2\) according as \(q=0\), \(1\le q\le n\), or \(n<q\le2n\). Then \(L_N:=\lceil2\sqrt N\rceil=2n+j\).

\begin{lemma}
\label{lem:comparison-domain}
There is an explicit connected \(N\)-vertex domain \(B_N\) satisfying
\begin{equation}
\log\tau(B_N)
\ge\log\tau(\Square n)+h(q-j).
\label{eq:comparison-domain}
\end{equation}
\end{lemma}

\begin{proof}
Starting from \(\Square n\), add \(\min(q,n)\) consecutive vertices along a new boundary row beginning at a corner. If \(q>n\), then add \(q-n\) consecutive vertices along a new boundary column. The first insertion in each nonempty stage is a leaf and has multiplier one. Every subsequent insertion closes a unit square in a rotated or reflected lexicographic order, so its multiplier is at least \(e^h\) by \cite{Tapp2021}. There are exactly \(j\) leaf insertions.
\end{proof}

\begin{proposition}
\label{prop:integer-rigidity}
Let \(S\) be a connected \(N\)-vertex domain with minimal hull \(\Rect r s\), and put \(\beta=r+s-L_N\ge0\). Then
\begin{equation}
\log\frac{\tau(B_N)}{\tau(S)}
\ge
\alpha(\beta+g)+\delta c-j(h-\alpha).
\label{eq:integer-rigidity}
\end{equation}
If \(S\) maximizes \(\tau\), then
\[
\beta=g=c=0.
\]
\end{proposition}

\begin{proof}
Apply \cref{cor:enhanced-hull} and then \cref{prop:rectangular-comparison}:
\[
\log\tau(S)
\le
\log\tau(\Square n)+hq
-\alpha(r+s-2n)-\alpha g-\delta c.
\]
Subtract this bound from \eqref{eq:comparison-domain}, noting \(r+s-2n=\beta+j\), to obtain \eqref{eq:integer-rigidity}.

The defect counts \(\beta,g,c\) are nonnegative integers. If any is nonzero, \eqref{eq:strict-constants} gives \(\alpha(\beta+g)+\delta c\ge\delta>2(h-\alpha)\). Since \(j\le2\), the right side of \eqref{eq:integer-rigidity} is then positive, contradicting maximality.
\end{proof}

\subsection{Ferrers decomposition and edge isoperimetry}
Interval rows and columns force the vacancies into four corners and determine the edge count from the hull side sum. The corresponding description of minimum-perimeter polyominoes appears in \cite{Kurz2008}; here it follows directly from the hull equality.
\begin{lemma}
\label{lem:four-corner}
Let \(S\) be connected, row convex, and column convex with minimal hull \(R=\Rect r s\).  Then \(R\setminus S\) is the disjoint union of four corner-anchored Ferrers defects.  If moreover \(r+s=L_N\) and \(|S|=N\), then
\begin{equation}
rs-N<\min(r,s).
\label{eq:deficit-short-side}
\end{equation}
\end{lemma}

\begin{proof}
Write the occupied interval in row \(y\) as \([\ell_y,u_y]\).  The rows on which \(\ell_y=0\) form a nonempty interval because the leftmost column is occupied and column convex.  Above this interval, the sequence \(\ell_y\) is monotone as one moves toward the top; below it, it is monotone toward the bottom.  Otherwise a column would be occupied, then absent, then occupied again.  The same argument applies to the right deficits \(r-1-u_y\).  The four monotone boundary profiles are precisely four Ferrers defects.

For \eqref{eq:deficit-short-side}, suppose \(k=rs-N\ge r\), with \(r\le s\). Then \(r(s-1)\ge N\), so an \(r\)-by-\((s-1)\) rectangle contains at least \(N\) vertices and has side sum \(r+s-1<L_N\), contradicting \(L_N=\lceil2\sqrt N\rceil\). The case \(r=s=1\) has \(k=0\) directly.
\end{proof}

\begin{lemma}
\label{lem:projection-bound}
Every \(N\)-vertex lattice domain satisfies
\begin{equation*}
|E(S)|\le2N-\lceil2\sqrt N\rceil.
\end{equation*}
For a connected domain, equality holds if and only if every nonempty row and column is an interval and the minimal hull side sum is \(\lceil2\sqrt N\rceil\).
\end{lemma}

\begin{proof}
Let \(a,b\) count the nonempty rows and columns. A row with \(m\) vertices in \(t\) intervals contributes \(m-t\) horizontal edges. Summing gives at most \(N-a\) horizontal and \(N-b\) vertical edges, with simultaneous equality exactly when every occupied row and column is an interval. Hence
\[
|E(S)|\le2N-a-b.
\]
Since \(ab\ge N\), we have \(a+b\ge\lceil2\sqrt N\rceil\). For connected \(S\), \(a,b\) are the hull side lengths, giving the stated equality conditions.
\end{proof}

\begin{theorem}
\label{thm:isoperimetric-rigidity}
Let \(S\subset\Z^2\) be connected, with \(|S|=N\) and \(\tau(S)=M_N\). Its minimal rectangular hull \(\Rect r s\) satisfies \(r+s=\lceil2\sqrt N\rceil\), and every nonempty row and column of \(S\) is an interval. Thus \(S\) has no finite nearest-neighbor component in its complement, and its vacancies in the hull form four disjoint corner Ferrers defects. Moreover,
\begin{equation*}
|E(S)|=2N-\lceil2\sqrt N\rceil,
\end{equation*}
so every spanning-tree maximizer is edge-isoperimetric.
\end{theorem}

\begin{proof}
By \cref{prop:integer-rigidity}, a maximizer has \(\beta=g=c=0\).  Since \(c=0\), we have \(S=S^+\); since \(g=0\), every row and column is an interval. A finite nearest-neighbor hole would contain a point with occupied vertices on both sides in its row, contradicting row convexity.  The Ferrers decomposition is \cref{lem:four-corner}, and the edge statement is \cref{lem:projection-bound}.
\end{proof}

\section{Rigidity of corner defects}
\label{sec:corner}

The vacancies of an edge-isoperimetric domain form four Ferrers defects. Subtracting their bulk cost leaves an energy whose nonnegative closed-walk expansion favors a single straight strip. In a finite hull, the correction also records depth: vacancies farther from a corner cost more. These two comparisons, combined with the folds of \cref{sec:polarization}, leave one strip for \(N\ge18\) and five two-corner maximizers below 18.

\subsection{Corner energy}
The same Ferrers diagram records deleted vertices and missing faces. To see this, let \(r,s\ge2\), place a defect \(M\) of size \(k<\min\{r,s\}\) at the bottom-left corner of \(\Rect r s\), and label each unit-square face by its lower-left vertex. A square loses a vertex exactly when its label belongs to \(M\), by the lower-ideal property. Thus \(S=\Rect r s\setminus M\) has bounded-face set \(U\setminus M\), where
\[
U=\{0,\ldots,r-2\}\times\{0,\ldots,s-2\}.
\]
The bound on \(k\) keeps the domain connected with its original hull and ensures that \(M\subseteq U\). At the other corners, the same correspondence holds in coordinates measured inward from the respective vertex and face corners.

Put \(G_U=(4I-A_U)^{-1}\). Rooted planar duality and Jacobi's complementary-minor identity give
\[
\frac{\tau(S)}{\tau(\Rect r s)}
=\frac{\det(4I-A_{U\setminus M})}{\det(4I-A_U)}
=\det G_U[M,M].
\]
For any \(M\subseteq U\), define its finite defect energy by
\[
\Energy_{r,s}(M)=-\log\det G_U[M,M]-h|M|.
\]
For the corner deletion above, this is \(\log(\tau(\Rect r s)/\tau(S))-h|M|\). At fixed hull and deficit, minimizing this energy therefore maximizes the spanning-tree count.

The quadrant limit has a nonnegative expansion in the rooted closed words of \cref{sec:preliminaries}. Let \(\mathcal Q=\Z_{\ge0}^2\). For a finite lower ideal \(M\subset\mathcal Q\) and a rooted closed word \(w\), translate the visited set of \(w\) so that its coordinatewise minimum is \((0,0)\), and denote the resulting set by \(P_w\). Define
\begin{equation}
D_M(w)=|M|-\bigl|\{t\in\mathcal Q:(t+P_w)\cap M\ne\varnothing\}\bigr|.
\label{eq:defect-count}
\end{equation}
Because \(M\) is a lower ideal, every counted translation lies in \(M\), and thus \(D_M(w)\ge0\).

\begin{definition}
\label{def:corner-energy}
For a finite Ferrers defect \(M\), define
\begin{equation}
\Energy(M)=
\sum_{\ell\ge1}\frac1{\ell4^\ell}
\sum_{w:|w|=\ell}D_M(w).
\label{eq:corner-energy}
\end{equation}
For the horizontal strip \(I_k=\{(0,0),\ldots,(k-1,0)\}\), write \(E_k=\Energy(I_k)\), with \(E_0=0\).
\end{definition}

Replacing \(4^\ell\) by \(z^\ell\) gives an absolutely convergent series for \(z>4\); monotone convergence recovers \eqref{eq:corner-energy} as \(z\downarrow4\). The determinant expansion \eqref{eq:closed-walk-expansion} identifies \(\Energy(M)=\lim_{r,s\to\infty}\Energy_{r,s}(M)\), with \(M\) fixed as both hull side lengths tend to infinity. The finite comparison below quantifies the correction to this limit.

\subsection{Exposed cells and compression}
A walk avoiding the origin excludes every maximal cell as a translation start. This gives a comparison for each word in \eqref{eq:corner-energy}.

\begin{lemma}
\label{lem:maximal-cells}
If \((0,0)\notin P_w\), then
\begin{equation*}
D_M(w)
\ge |\Max(M)|,
\end{equation*}
where \(\Max(M)\) is the set of coordinatewise maximal cells of \(M\).
\end{lemma}

\begin{proof}
If \(m\in\Max(M)\), then \(m+p\notin M\) for every nonzero \(p\in\mathcal Q\).  Thus \(m\) cannot occur as a translation start counted in the second term of \eqref{eq:defect-count}.  Distinct maximal cells give distinct excluded starts.
\end{proof}

\begin{lemma}
\label{lem:three-cell-obstruction}
Let \(L=\{(0,0),(1,0),(0,1)\}\). If \(M\) is a nonempty nonrectangular Ferrers defect with \(a=|\Max(M)|\), then for every closed word \(w\),
\begin{equation}
D_M(w)\ge D_L(w)+(a-2)D_{I_1}(w).
\label{eq:L-domination}
\end{equation}
Consequently \(\Energy(M)\ge\Energy(L)+(a-2)E_1\).
\end{lemma}

\begin{proof}
A nonempty finite lower ideal is a rectangle precisely when it has a unique maximal cell, so \(a\ge2\). If the normalized walk visits the origin, every term in \eqref{eq:L-domination} vanishes. Otherwise \(D_{I_1}(w)=1\). If the walk visits at least one of \((1,0),(0,1)\), then \(D_L(w)=2\), and \cref{lem:maximal-cells} gives the required \(a\) excluded starts. In the remaining case \(D_L(w)=3\). Choose a nonmaximal cell \(t\in M\) of largest coordinate sum. Every point of \(P_w\) has coordinate sum at least two, so no translate \(t+p\) lies in \(M\); otherwise a monotone path from \(t\) to \(t+p\) contains a nonmaximal cell of larger coordinate sum. This supplies one further excluded start, in addition to the \(a\) maximal cells.
\end{proof}

Repeating a row or column also cannot lower the energy.

\begin{lemma}
\label{lem:repetition}
If \(M'\) is obtained from a Ferrers defect \(M\) by duplicating one of its rows, then
\[
D_{M'}(w)\ge D_M(w)
\]
for every closed word \(w\).  The same holds for duplicating a column. Consequently \(\Energy(M')\ge\Energy(M)\).  In particular, the energy of a rectangular corner defect is nondecreasing in each side length.
\end{lemma}

\begin{proof}
Let \(\ell_y\) be the row lengths of \(M\), extended by zero, and let
\[
v_y=\max_{(a,b)\in P_w}(\ell_{y+b}-a)_+
\]
be the row lengths of the union of translated copies that appears in \eqref{eq:defect-count}.  Duplicate row \(j\), of length \(u\).  If \(v'_y\) denotes the new union widths, then
\[
v'_y\ge v_y,
\qquad
v'_y\le v_{y-1}\ (y\ge1),
\qquad
v'_y=v_{y-1}\ (y>j).
\]
For \(y\le j\), every added interval has upper endpoint at most \(u\). If \(y\ge1\), it lies in \((v_y,\min\{u,v_{y-1}\}]\); at \(y=0\), it lies in \((v_0,u]\). The sequence \(v_y\) is nonincreasing, so these intervals are disjoint and lie in \((v_j,u]\). Hence
\[
\sum_{y=0}^j(v'_y-v_y)\le u-v_j.
\]
The shifted tail contributes exactly \(v_j\), so the union gains at most \(u\) starts. Since \(|M|\) increases by \(u\), the defect count cannot decrease. Reflection across the diagonal gives the column statement.
\end{proof}

\subsection{Straight strips minimize the limiting energy}
For the normalized visited set \(P_w\), let \(a(w)=\min\{x:(x,0)\in P_w\}\). A translation meeting the horizontal strip \(I_k\) starts on its row and lies at least \(a(w)\) cells from its right endpoint. Thus
\begin{equation*}
D_{I_k}(w)=\min\{k,a(w)\}.
\end{equation*}
The sequence \(E_k\) is increasing and discretely concave, strictly so because, for every \(j\ge1\), some closed word has \(a(w)=j\).

The strip energy remains bounded as the strip grows. The boundary Green matrix gives a Fourier expression for the limit
\[
C_{\mathrm{strip}}=\lim_{k\to\infty}E_k=\sup_kE_k<\infty.
\]
Appendix~\ref{app:corner-certificates} bounds this limit and the energies of the smallest nonstrip defects. Together with the coefficientwise comparisons above, these bounds handle every Ferrers shape.

\begin{proposition}
\label{prop:strip-rigidity}
For every nonempty Ferrers defect \(M\) with \(|M|=k\),
\begin{equation*}
\Energy(M)\ge E_k.
\end{equation*}
Equality holds exactly when \(M\) is a single row or a single column. For every other defect, the gap is bounded below by an absolute positive constant.
\end{proposition}

\begin{proof}
A nonrectangular Ferrers defect dominates the three-cell \(L\) coefficientwise by \cref{lem:three-cell-obstruction}. Apart from the \(2\times2\) square, every nonstrip rectangle contains a \(2\times3\) rectangle up to transpose and has no smaller energy by \cref{lem:repetition}. Appendix~\ref{app:corner-certificates} proves
\[
C_{\mathrm{strip}}<\min\{\Energy(L),\Energy(2\times3)\},
\qquad
E_4<\Energy(2\times2).
\]
The minimum of these three fixed gaps is a uniform lower bound for every nonstrip defect. Row and column strips have energy \(E_k\).
\end{proof}

\subsection{Defects concentrate in one corner}
Strict concavity of \(E_k\), together with \(E_0=0\), gives, for \(a,b\ge1\),
\begin{equation}
E_a+E_b-E_{a+b}\ge2E_1-E_2>0.
\label{eq:strip-concentration-gap}
\end{equation}

\begin{proposition}
\label{prop:concentration}
Let \(M_1,\ldots,M_b\) be Ferrers defects with total size \(k>0\).  Then
\begin{equation*}
\sum_i\Energy(M_i)\ge E_k.
\end{equation*}
Equality holds if and only if exactly one defect is nonempty and it is a row or column strip. Otherwise the gap is bounded below by an absolute positive constant.
\end{proposition}

\begin{proof}
Discard the empty defects and apply \cref{prop:strip-rigidity} to each remaining corner. If at least two corners remain, their total strip energy exceeds \(E_k\) by at least the fixed positive gap in \eqref{eq:strip-concentration-gap}. For a single corner, \cref{prop:strip-rigidity} gives both the equality case and the uniform gap.
\end{proof}

\subsection{Finite-domain comparison}
The difference \(\Energy_{r,s}(M)-\Energy(M)\) accounts for walks that leave the finite face rectangle \(U\) through its opposite sides. For a strip \(I\), write \(E_{r,s}(I)=\Energy_{r,s}(I)\). A span bound controls this correction in either orientation. For long-side strips, retaining absorption at both axes sharpens it to fourth order.

\begin{lemma}
\label{lem:finite-strip-error}
Let \(1\le r\le s\) be integers. The empty strip has \(E_{r,s}(I_0)=E_0=0\). If \(r\ge2\) and \(1\le k<r\), then
\begin{equation}
0\le E_{r,s}(I_k)-E_k
\le
\frac{4k}{\pi}
\left(
\frac1{(r-k)^2}+\frac1{(s-1)^2}
\right).
\label{eq:finite-strip-error}
\end{equation}
The same right side bounds the other strip orientation. For the long-side strip \(I_k^{\mathsf T}=\{(0,j):0\le j<k\}\), assume in addition that \(r\ge3\) and \(k+1<s\). With \(S_2(k)=k(k+1)(2k+1)/6\), the boundary-flux estimate is
\begin{equation}
\begin{aligned}
0\le E_{r,s}(I_k^{\mathsf T})-E_k
\le\frac{\pi^3}{60}\Bigg\{&2S_2(k)
\left[\frac1{(r^2-4)^2}+\frac1{(s^2-k^2)^2}\right]\\
&+k(k+1)\left[\frac1{(r^2-1)^2}+\frac1{(s^2-(k+1)^2)^2}\right]\Bigg\}.
\end{aligned}
\label{eq:finite-long-strip-error}
\end{equation}
\end{lemma}

\begin{proof}
A translated closed walk that contributes in the quadrant but not in the finite face rectangle must have horizontal span at least \(r-k\) or vertical span at least \(s-1\). There are at most \(k\) relevant translations for each word. For either span threshold \(q\in\{r-k,s-1\}\), reflect a planar bridge at its first visit to distance \(b=\lceil q/2\rceil\). The total loop weight of bridges with span at least \(q\) is bounded by \(-2\gamma_{2b}\), with the normalization
\[
\gamma_j=\frac1\pi\int_0^\pi
2\asinh\left(\sin\frac\theta2\right)\cos(j\theta)\,d\theta,
\qquad j\ge1.
\]
The integral representation in Appendix~\ref{subsec:hyperbolic-normalization} gives \(0\le-\gamma_j\le2/(\pi j^2)\), so the horizontal and vertical tails give \eqref{eq:finite-strip-error}. For the other orientation the tail denominators are \(r-1\) and \(s-k\). For \(k\ge1\), the function \(t\mapsto(t-k)^{-2}-(t-1)^{-2}\) is nonincreasing on \(t>k\); since \(r\le s\),
\[
\frac1{(r-1)^2}+\frac1{(s-k)^2}\le\frac1{(r-k)^2}+\frac1{(s-1)^2}.
\]
Thus the bound applies to both orientations. Appendix~\ref{subsec:boundary-flux} proves \eqref{eq:finite-long-strip-error} by a Green-function comparison retaining absorption at both axes.
\end{proof}

For general corner defects, each missing cell contributes a finite correction. Write \(c_{r,s}=\Energy_{r,s}(I_1)-E_1\) for the positive single-cell correction.

\begin{lemma}
\label{lem:finite-transfer}
For nested Ferrers ideals \(A\subseteq B\subseteq U\),
\begin{equation}
\Energy_{r,s}(B)-\Energy_{r,s}(A)
\ge\Energy(B)-\Energy(A)+\sum_{(a,b)\in B\setminus A}c_{r-a,s-b}.
\label{eq:finite-transfer}
\end{equation}
In particular,
\[
\Energy_{r,s}(M)\ge\Energy(M)+\sum_{(a,b)\in M}c_{r-a,s-b}
\ge\Energy(M)+|M|c_{r,s}.
\]
For fixed \(r,s\), the sequence \(E_{r,s}(I_j^{\mathsf T})-E_j\), \(0\le j\le s-1\), is discretely convex and vanishes at zero.
\end{lemma}

\begin{proof}
For a normalized closed word \(w\), let \(H_M(w)\) be the translation starts counted in \eqref{eq:defect-count}, and let \(F_w=\{t\in\mathcal Q:t+P_w\subset U\}\) be the starts for which the whole walk fits \(U\). Because \(P_w\) is normalized, \(F_w\) is empty or a rectangle containing zero. The determinant expansion, with \(\log4-h\) expressed as the total closed-word weight, gives
\[
\Energy_{r,s}(M)-\Energy(M)
=\sum_{\ell\ge1}\frac1{\ell4^\ell}
\sum_{|w|=\ell}|H_M(w)\setminus F_w|.
\]
For nested ideals \(A\subseteq B\), each coefficient increases by \(|(H_B(w)\setminus H_A(w))\setminus F_w|\ge0\).

When \(0\in P_w\), the lower-ideal property gives \(H_M(w)=M\). An added start \((a,b)\in B\setminus A\) is lost precisely when \(P_w\) fails to fit the smaller face rectangle with primal sides \(r-a,s-b\). By the singleton expansion, these words have total weight \(c_{r-a,s-b}\). Summing over the added cells and discarding the other nonnegative contributions proves \eqref{eq:finite-transfer}. Shrinking either side cannot decrease the lost-word set, so each correction is at least \(c_{r,s}\); a sufficiently long out-and-back word makes it positive.

For a vertical strip of length \(j\), the hitting starts form an initial interval of length \((j-a_w)_+\), where \(a_w\) is the least height of a visited point on the left edge of \(P_w\). The fitting starts on that axis form a fixed initial interval of length \(b_w\), with \(b_w=0\) when there is no fit. Hence the lost-start count is \((j-a_w-b_w)_+\). It is convex in \(j\), and summing the nonnegative word weights proves the final assertion.
\end{proof}

For a domain \(S\) obtained by deleting corner Ferrers defects \(M_i\) of total size \(k<r\) from \(\Rect r s\), with \(2\le r\le s\), let \(M\subseteq U\) be the union of their face placements. These placements are disjoint and have the same sizes as the vertex defects, so Jacobi's identity gives \(\tau(S)/\tau(\Rect r s)=\det G_U[M,M]\). Hadamard--Fischer bounds this minor by the product of the corner minors. The per-cell bound in \cref{lem:finite-transfer} applies at every reflected corner, giving
\begin{equation}
\log\frac{\tau(\Rect r s)}{\tau(S)}
\ge (h+c_{r,s})k+\sum_i\Energy(M_i).
\label{eq:finite-multicorner}
\end{equation}
Short defects also admit a lower bound from the full Green-matrix loss; \cref{lem:short-defect-corrections} gives the two estimates used below.

\subsection{Concentration and the exceptions}
For fixed \(N\), an \emph{admissible hull} has positive integer sides \(r\le s\) satisfying \(r+s=\lceil2\sqrt N\rceil\) and \(rs\ge N\); its deficit is \(k=rs-N\). It is as elongated as possible precisely when \(k\le s-r\): the next hull, with sides \(r-1,s+1\), would have too few vertices. We call this the canonical hull.

\begin{proposition}
\label{prop:canonical-concentration}
Let \(r\le s\) and \(0\le k<r\) be integers, put \(N=rs-k\ge18\), and assume
\[
r+s=\lceil2\sqrt N\rceil,\qquad k\le s-r.
\]
Among domains obtained from \(\Rect r s\) by deleting corner Ferrers defects of total size \(k\), the long-side corner strip uniquely maximizes the spanning-tree count up to lattice isometry.
\end{proposition}

The proof in \cref{subsec:terminal-hulls} uses the minimum-side-sum condition to obtain \(r>(k+1)^2/4\): the far boundaries recede quadratically in the defect size, making the strip's finite-volume error smaller than the nonstrip energy gap. Domain monotonicity of Green minors handles the short defects.

\begin{proposition}
\label{prop:analytic-cutoff}
If \(N\ge18\), a domain in the class supplied by \cref{thm:isoperimetric-rigidity} with several nonempty corners or a single nonstrip corner cannot maximize the number of spanning trees within that class.
\end{proposition}

\begin{proof}
On the canonical hull, apply \cref{prop:canonical-concentration}. On every other hull, \cref{lem:multicorner-uniform} excludes two or more nonempty corners by combining the two-corner fold with the finite corrections above and coupled rectangular errors.

For a single nonstrip corner, \cref{lem:single-corner-reduction} either gives a strict improvement or folds a hook on a square hull into two opposite strips without decreasing the tree count. The canonical strip strictly beats the latter domain by \cref{lem:multicorner-uniform}, so neither case can maximize.
\end{proof}

Below 18, only six admissible hulls can carry several nonempty corners or a single nonstrip corner. Their two- and three-vacancy comparisons give the exceptions.

\begin{lemma}
\label{lem:finite-classification}
Up to lattice isometry, the spanning-tree maximizers with \(1\le N\le17\) that are neither complete rectangles nor single-strip domains are precisely the two-corner domains at \(N=7,10,13,14,17\) listed in \cref{tab:exceptions}. They strictly beat the canonical strip at \(N=7,14,17\) and tie it at \(N=10,13\).
\end{lemma}

\begin{proof}
Corner separation reduces the two-vacancy hulls to opposite singleton defects. At three vacancies, it leaves three patterns on \(4\times4\) and four on \(4\times5\). Their face graphs are paths, four-cycles, ladders, or simple attachments of these graphs. The recurrences in \cref{subsec:finite-corner-comparisons} give all comparisons and equality cases directly.
\end{proof}

\begin{theorem}
\label{thm:strip-reduction}
Apart from the five configurations in \cref{lem:finite-classification}, every spanning-tree maximizer is a complete rectangle or a rectangle with one straight corner strip removed.  In particular, for every \(N\ge18\), every maximizer is a single-strip domain.
\end{theorem}

\begin{proof}
By \cref{thm:isoperimetric-rigidity}, every maximizer lies in the four-corner Ferrers class.  Apply \cref{prop:analytic-cutoff} above the cutoff and \cref{lem:finite-classification} below it.
\end{proof}

The energy comparisons control the shape of each defect and the cost of spreading vacancies among corners. The folds in the next section compare orientations and change the hull, supplying the geometric steps used in the reduction above. They also select the best remaining strip; \cref{sec:classification} combines these comparisons with the five two-corner cases.

\section{Polarization and the explicit extremizer}
\label{sec:polarization}

Polarization compares spanning-tree counts by folding face sets across lattice reflections. Axial folds separate adjacent defects; diagonal folds orient strips and move vacancies to more elongated hulls. Repeated hull changes then select the extremal strip explicitly.

\subsection{Discrete polarization}
Let \(\sigma\) be a reflection preserving \(\Z^2\), with open half-planes \(H^+,H^-\) and axis \(H^0\). The \emph{polarization} \(P(U)\) of a finite set \(U\subset\Z^2\) toward \(H^+\) moves the occupied point of each singly occupied pair \(\{v,\sigma v\}\) to \(H^+\), leaving empty or doubly occupied pairs and fixed points unchanged. This is the discrete analogue of classical polarization; see \cite{BrockSolynin2000} for the general symmetrization perspective.

\begin{lemma}
\label{lem:polarization}
For \(K_U=I+A_U\) and every \(\ell\ge0\),
\[
\tr(K_U^\ell)\le\tr(K_{P(U)}^\ell).
\]
For every real \(z\ge4\),
\begin{equation}
\det(zI-A_U)\ge\det(zI-A_{P(U)}).
\label{eq:polarization-det}
\end{equation}
For diagonal reflections, the ordinary closed-walk moments also satisfy
\begin{equation}
\tr(A_U^\ell)\le\tr(A_{P(U)}^\ell).
\label{eq:polarization-moments}
\end{equation}
\end{lemma}

\begin{proof}
The holding steps in \(K_U\) accommodate an axis running between adjacent rows. The ambient kernel \(K\) satisfies, for \(x,y\in H^+\),
\[
K(x,y)=K(\sigma x,\sigma y)
\ge K(x,\sigma y)=K(\sigma x,y).
\]
Indeed, an edge crossing the axis can only join a reflection pair, and its unit weight then matches the diagonal entry of \(K\).

Group the cyclic vertex sequences by their reflection orbits. For representatives in \(H^+\), write the chosen signs as \(s_i\in\{-1,1\}\). Each edge factor is \(a_i+b_i s_i s_{i+1}\), with \(a_i,b_i\ge0\), and each occupancy indicator is \(c_i+d_i s_i\), with \(c_i\ge0\). Polarization replaces \(d_i\) by \(|d_i|\). Expanding and summing over signs leaves only even powers of each sign, hence monomials in the \(d_i\) with nonnegative coefficients. Taking absolute values cannot decrease them. Fixed-axis positions carry no sign and contribute nonnegative factors. This proves the trace comparison for \(K_U\). Since \(\rho(K_U)<5\), the expansion of \(zI-A_U=(z+1)I-K_U\) converges for \(z\ge4\) and gives \eqref{eq:polarization-det}.

For a diagonal reflection, no edge joins \(H^+\) directly to \(H^-\), so ordinary walks admit a simpler argument. Fold the walk into \(H^+\cup H^0\). Each excursion between axis visits has two lifts if all its pairs are doubly occupied, one if its singly occupied pairs have the same sign, and none otherwise. When the root is off the axis, join the initial and terminal pieces cyclically; a walk avoiding the axis is one cyclic excursion. Polarization aligns the single pairs and cannot decrease the lift count, proving \eqref{eq:polarization-moments}.
\end{proof}

\begin{lemma}
\label{lem:polarization-strict}
Suppose that, inside \(P(U)\cap H^+\), a path joins the images of two singly occupied reflection pairs whose occupied points originally lie in opposite half-planes. Then \eqref{eq:polarization-det} is strict for every \(z\ge4\). For a diagonal reflection, \eqref{eq:polarization-moments} is strict for some even \(\ell\).
\end{lemma}

\begin{proof}
Traverse the path and return. In the expansion above, select the two endpoint factors \(d_i\), the factors \(b_i\) along the outward path, and the factors \(a_i\) on the return. All selected edge factors are positive: an edge between distinct vertices of \(H^+\) has same-side weight one and cross-side weight zero. The endpoint product changes from negative to positive under polarization, so one trace inequality for \(K_U\) is strict. The determinant expansion gives strictness for every \(z\ge4\). For a diagonal reflection, the same folded walk has one ordinary lift after polarization and none before, giving the final assertion.
\end{proof}

\subsection{Separating two corner defects}
An axial reflection moves adjacent defects to opposite corners; the intact row across from them makes the improvement strict.

\begin{proposition}
\label{prop:separate-corners}
Let \(S\) be obtained from \(\Rect r s\), with \(r\le s\), by deleting two nonempty Ferrers defects at adjacent corners, of total size \(k<r\). Keep one defect fixed and reflect the other across a hull midline so that they occupy opposite corners. The resulting domain has the same cardinality and minimal hull, and strictly more spanning trees.
\end{proposition}

\begin{proof}
After a lattice isometry, the face rectangle is \([0,u)\times[0,v)\), where \(k\le\min\{u,v\}\), and the defects are at the bottom-left and bottom-right corners. Their horizontal supports are disjoint, since their widths sum to at most \(k\le u\). Reflect the right defect across \(y=(v-1)/2\), and call the resulting face set \(V\).

Polarizing \(V\) upward gives the original face set: each affected source column is an upper interval, and both that interval and its reflection polarize to it. The disjoint horizontal supports let this argument apply independently at the two corners.

Both defects have height at most \(k-1\le v-1\), leaving the entire top face row occupied in the source. This row lies strictly above the axis and joins \((0,v-1)\) to \((u-1,v-1)\). In \(V\), the left endpoint is occupied and its reflection is missing; the right endpoint has the opposite status. These singly occupied pairs give strictness by \cref{lem:polarization-strict}, and rooted planar duality gives the strict increase in tree count.

Reflecting the primal defect realizes \(V\) with the same cardinality. Since \(k<r\le s\), every row and column remains nonempty, and an intact column meets all occupied row intervals. Thus the domain is connected with the same minimal hull.
\end{proof}

\subsection{The preferred strip orientation}
For \(1\le r\le s\) and \(0\le k<r\), the long- and short-side strip domains are
\[
\begin{aligned}
C_y(r,s,k)&=\Rect r s\setminus\{(0,j):0\le j<k\},\\
C_x(r,s,k)&=\Rect r s\setminus\{(i,0):0\le i<k\},
\end{aligned}
\]
with tree counts \(Y(r,s,k)=\tau(C_y(r,s,k))\) and \(X(r,s,k)=\tau(C_x(r,s,k))\).

\begin{proposition}
\label{prop:orientation}
For \(1\le r\le s\) and \(0\le k<r\),
\begin{equation}
Y(r,s,k)\ge X(r,s,k).
\label{eq:orientation}
\end{equation}
Equality holds if and only if \(r=s\) or \(k\le1\).
\end{proposition}

\begin{proof}
The bounded-face sets are
\[
\begin{aligned}
U_y&=[0,r-2]\times[0,s-2]\setminus\{(0,j):0\le j<k\},\\
U_x&=[0,r-2]\times[0,s-2]\setminus\{(i,0):0\le i<k\}.
\end{aligned}
\]
Polarization toward \(x<y\) sends \(U_y\) exactly to \(U_x\).  Rooted planar duality and \cref{lem:polarization} therefore give \eqref{eq:orientation}.  If \(r=s\), the two domains are reflected copies; if \(k\le1\), the deleted face sets coincide up to symmetry.  Otherwise the path
\[
(0,1),(0,2),\ldots,(0,r-1)
\]
inside the polarized face domain connects two singly occupied pairs whose occupied points originally lie in opposite half-planes. \Cref{lem:polarization-strict} gives strictness.
\end{proof}

\subsection{Neighboring hulls}
Moving one unit from the shorter hull side to the longer one decreases the strip deficit at fixed cardinality and cannot lower the best strip count.
\begin{proposition}
\label{prop:adjacent-hull}
Assume \(1\le r\le s\), \(0\le k<r\), and
\[
k'=k-(s-r+1)\ge0.
\]
Then
\begin{equation}
Y(r,s,k)
\le X(r-1,s+1,k')
\le Y(r-1,s+1,k').
\label{eq:adjacent-hull}
\end{equation}
Equality between the first and last terms holds if and only if
\begin{equation}
s-1-k\le1,
\qquad
k'\le1.
\label{eq:adjacent-equality}
\end{equation}
\end{proposition}

\begin{proof}
The second inequality is \cref{prop:orientation}. For the first, put \(a=r-2\), \(b=s-1\), and \(\ell=b-k\). Up to translation and reflection, the source \(C_y(r,s,k)\) and the candidate \(C_x(r-1,s+1,k')\) have face sets \(U_{\mathrm{left}}\) and \(U_{\mathrm{bottom}}\), respectively. They attach \(\ell\) faces to the left or bottom of the same core \(Q=[0,a-1]\times[0,b-1]\):
\[
U_{\mathrm{left}}=Q\cup\{(-1,j):0\le j<\ell\},
\qquad
U_{\mathrm{bottom}}=Q\cup\{(j,-1):0\le j<\ell\}.
\]
Since \(0\le\ell\le a\le b\), polarization toward \(x<y\) maps \(U_{\mathrm{bottom}}\) to \(U_{\mathrm{left}}\).  The determinant direction in \cref{lem:polarization} gives the first inequality in \eqref{eq:adjacent-hull}.

If \(\ell=0\), the face graphs coincide.  If \(\ell=1\), they are isomorphic. For \(\ell\ge2\), the path
\[
(-1,1),(0,1),(0,2),\ldots,(0,a)
\]
provides the strictness criterion of \cref{lem:polarization-strict}.  Combining this with the equality condition in \cref{prop:orientation} gives \eqref{eq:adjacent-equality}.
\end{proof}

Since \((s-1-k)+k'=r-2\), equality forces \(r\le4\). The complete list of adjacent-hull equality triples is
\begin{equation}
(2,2,1),\ (3,3,2),\ (3,3,1),\ (3,4,2),\ (4,4,2),\ (4,5,3),
\label{eq:adjacent-equality-list}
\end{equation}
corresponding to vertex counts \(3,7,8,10,14,17\).

\subsection{Folding two corner strips}
A change of hull can also concentrate two opposite strips: it absorbs part of one strip and collects the remaining vacancies at a single corner.

\begin{proposition}
\label{prop:two-corner-fold}
Let \(r\le s\) and \(b\ge1\) be integers with \(s-r+b+2<r\). If \(S\) is obtained from \(\Rect r s\) by deleting strips of lengths \(s-r+2\) and \(b\) at opposite corners, in either orientation, then
\[
\tau(S)\le Y(r-1,s+1,b+1).
\]
Equality holds if and only if \((r,s,b)=(4,4,1)\).
\end{proposition}

\begin{figure}[htbp]
\centering
\begin{tikzpicture}[x=.36cm,y=.36cm,font=\small]
  \begin{scope}[shift={(0,.5)}]
    \fill[gray!28] (0,0) rectangle (1,2);
    \fill[orange!40] (0,2) rectangle (1,3);
    \fill[teal!35] (4,0) rectangle (5,2);
    \draw[step=1,gray!55,thin] (0,0) grid (5,6);
    \draw[thick] (0,0) rectangle (5,6);
    \draw[dashed,gray!75] (-.15,3)--(5.15,3);
    \draw[green!40!black,very thick] (.5,5.5)--(4.5,5.5);
    \foreach \x in {.5,1.5,2.5,3.5,4.5}
      \fill[green!40!black] (\x,5.5) circle (1.1pt);
  \end{scope}
  \node[anchor=north] at (2.5,-.45) {adjacent strips};
  \draw[->,thick] (5.9,3.5)--(8.1,3.5);
  \node[anchor=south] at (7,3.8) {separate};
  \begin{scope}[shift={(9,.5)}]
    \fill[gray!28] (0,0) rectangle (1,2);
    \fill[orange!40] (0,2) rectangle (1,3);
    \fill[teal!35] (4,4) rectangle (5,6);
    \draw[step=1,gray!55,thin] (0,0) grid (5,6);
    \draw[thick] (0,0) rectangle (5,6);
    \draw[dashed,gray!75] (-.15,3)--(5.15,3);
  \end{scope}
  \node[anchor=north] at (11.5,-.45) {opposite strips};
  \draw[->,thick] (14.9,3.5)--(17.1,3.5);
  \node[anchor=south] at (16,3.8) {fold};
  \begin{scope}[shift={(18,0)}]
    \fill[teal!35] (3,4) rectangle (4,6);
    \fill[orange!40] (3,6) rectangle (4,7);
    \draw[step=1,gray!55,thin] (0,0) grid (4,7);
    \draw[thick] (0,0) rectangle (4,7);
    \draw[dashed,gray!75] (0,6)--(4,2);
  \end{scope}
  \node[anchor=north] at (20,-.45) {one strip};
\end{tikzpicture}
\caption{Two strict improvements at $N=37$. Each domain has 25 bounded faces; colored cells are missing. Separation preserves the $6\times7$ hull; the fold changes it to $5\times8$. The two gray cells are absorbed by the hull change, and the orange cell joins the teal pair in the new strip. The green row supplies the first strictness path. Both arrows increase the tree count by \cref{prop:separate-corners,prop:two-corner-fold}.}
\label{fig:corner-separation}
\end{figure}

\begin{proof}
First, the tree count is largest when both strips are vertical, along the longer side. In the face rectangle \([0,r-2]\times[0,s-2]\), polarization toward \(x<y\) turns a bottom-left vertical strip horizontal and fixes the opposite vertical strip: its cells lie either above the common square, with partners outside the rectangle, or in \(x\ge y\), where deletion stays fixed. The profiles remain disjoint because their lengths sum to less than \(r\). This compares two vertical strips with one mixed orientation; a half-turn gives the other mixed orientation.

For \(s>r\), an opposite horizontal strip is also fixed, since its row \(y=s-2\) lies above the common square. The same polarization therefore sends a mixed orientation to the two-horizontal configuration. Each polarization lowers or preserves the determinant by \cref{lem:polarization}.

For \(s=r\), two horizontal strips are isometric to two vertical strips. These comparisons establish the claimed preference for two long-side strips.

Consider the candidate face set
\[
V=\bigl([0,r-3]\times[0,s-1]\bigr)
\setminus\{(r-3,y):s-b-1\le y\le s-1\}.
\]
Polarize toward \(x+y<s-2\), using the reflection \((x,y)\mapsto(s-2-y,s-2-x)\). The full rectangle's top row moves to the new left column, leaving its bottom \(s-r+1\) cells absent. The top puncture \((r-3,s-1)\) reflects to \((-1,s-r+1)\), extending this missing interval by one. The other \(b\) punctures stay at the top-right: they lie in the unpreferred half-plane and their partners, with coordinates \((0,s-r+1),\ldots,(b-1,s-r+1)\), are occupied. Hence
\[
\begin{aligned}
P(V)={}&\bigl([-1,r-3]\times[0,s-2]\bigr)\\
&\setminus\bigl(\{-1\}\times[0,s-r+1]\bigr)\\
&\setminus\bigl(\{r-3\}\times[s-b-1,s-2]\bigr).
\end{aligned}
\]
After translation, this is the source with two long-side strips. The candidate has primal hull \((r-1,s+1)\), deficit \(b+1<r-1\), and the same cardinality:
\[
(r-1)(s+1)-(b+1)=rs-(s-r+2+b).
\]
In each domain the deficit is smaller than the short side, leaving nonempty row intervals and an intact column through them. Hence both domains are connected with the displayed minimal hulls; polarization and rooted planar duality give the comparison.

If \(r\ge5\), the lattice path through
\[
(0,0),\quad(0,s-3),\quad(-1,s-3),\quad(-1,s-2)
\]
lies in \(P(V)\cap\{x+y<s-2\}\). Its left-column points are present because \(s-3\ge s-r+2\). The first endpoint is occupied in \(V\) and its partner is outside; the last endpoint is outside \(V\) and its partner \((0,s-1)\) is occupied. The strictness criterion gives strict improvement. The only remaining parameters are \((r,s,b)=(4,4,1)\). All orientations are then related by symmetries of the square, and both face graphs are a four-cycle with a two-edge pendant path. Their determinants agree.
\end{proof}

\subsection{Folding a Ferrers corner}
For a general Ferrers corner, a diagonal fold absorbs a rectangular block and carries the remaining pieces to opposite corners of a more elongated hull.

\begin{proposition}
\label{prop:ferrers-corner-fold}
Let a domain be obtained from \(\Rect r s\), with \(r\le s\), by removing one nonstrip Ferrers defect \(M\) of size \(k<r\). Let \(q\) be the defect height along the longer hull side. If \(q>s-r\), there is a domain of the same cardinality and hull side sum, on a strictly more elongated hull and with at most two opposite Ferrers defects, whose spanning-tree count is at least as large. The comparison is strict unless \(r=s\) and \(M\) is a hook; in that case the new defects are two opposite strips.
\end{proposition}

\begin{figure}[htbp]
\centering
\begin{tikzpicture}[x=.34cm,y=.34cm,font=\small]
  \begin{scope}[shift={(0,1)}]
    \fill[orange!28] (6,0) rectangle (8,2);
    \fill[red!30] (5,0) rectangle (6,1);
    \fill[teal!35] (7,2) rectangle (8,3);
    \draw[step=1,gray!55,thin] (0,0) grid (8,8);
    \draw[thick] (0,0) rectangle (8,8);
    \draw[dashed,gray!80] (0,2)--(6,8);
    \node at (7,1) {$B$};
  \end{scope}
  \node[anchor=north,align=center] at (4,-.7) {$9\times9$ hull\\defect $(3,2,1)$};
  \draw[->,thick] (10,5)--(14,5);
  \node[anchor=south,align=center] at (12,5.3) {corner fold};
  \node[anchor=north,align=center] at (12,4.7) {larger tree count};
  \begin{scope}[shift={(16,0)}]
    \fill[red!30] (5,0) rectangle (6,1);
    \fill[teal!35] (0,9) rectangle (1,10);
    \draw[step=1,gray!55,thin] (0,0) grid (6,10);
    \draw[thick] (0,0) rectangle (6,10);
    \draw[dashed,gray!80] (0,2)--(6,8);
  \end{scope}
  \node[anchor=north,align=center] at (19,-.7) {$7\times11$ hull\\two opposite singletons};
\end{tikzpicture}
\caption{A Ferrers corner fold at $N=75$. The grids show bounded faces; colored cells are missing. The orange $2\times2$ block $B$ is absorbed by the change of hull, the red piece stays at the bottom-right corner, and the teal piece reflects across $x-y=-2$ to the top-left corner. Both domains have 58 bounded faces. Polarization sends the right-hand face set back to the left-hand one, so \cref{prop:ferrers-corner-fold} gives strict improvement in the direction of the arrow.}
\label{fig:corner-fold}
\end{figure}

\begin{proof}
Put \(u=r-1\) and \(v=s-1\), and place the missing faces at the bottom-right corner of \([0,u)\times[0,v)\). Choose the largest positive integer \(t\) for which the bottom-right \(t\times(s-r+t)\) block lies in \(M\), and put \(c=s-r+t\) and \(B=[u-t,u)\times[0,c)\). Such a block exists because \(q>s-r\).

No missing cell can lie both to the left of \(B\) and above it: the Ferrers property would then put the larger \((t+1)\times(c+1)\) block in \(M\). Thus the cells outside \(B\) split into a left piece in the bottom \(c\) rows and an upper piece in the rightmost \(t\) columns.

On the new face rectangle \([0,u-t)\times[0,v+t)\), retain the left piece as a bottom-right deletion and reflect the upper piece by
\[
(x,y)\longmapsto(y-c,x+c)
\]
into a top-left deletion. Figure~\ref{fig:corner-fold} shows the two remaining pieces. They are opposite Ferrers defects of total size \(k-tc\). The new primal hull is \(\Rect{r-t}{s+t}\), and
\[
(r-t)(s+t)-(k-tc)=rs-k,
\qquad 0\le k-tc<r-t.
\]
Since the total defect width is at most \(k-tc<r-t\), an intact column meets every occupied row. No whole row or column is lost, so the domain is connected with the displayed minimal hull.

Polarization toward \(x-y>-c\) sends the new full face rectangle to the old one minus \(B\). The left piece reflects into \(x<0\), since \(y<c\); the upper piece reflects into \(y\ge v\), since \(x\ge u-t\). Their reflected cells thus lie outside the old rectangle. Deleting their respective preimages makes the candidate polarize exactly to the source, so \cref{lem:polarization} gives the determinant comparison.

For strictness, let \(p\) be the defect width. The bottom row, rightmost column, and block \(B\) give
\[
u\ge k\ge p+q-1+(t-1)(c-1),
\qquad u-p-q+c\ge t(c-1).
\]
If \(u-p-q+c\ge2\), the path from \((u-p-1,c-1)\) vertically to \((u-p-1,q)\), then horizontally to \((u-t,q)\), lies in the strict preferred half-plane, avoids \(M\), and joins singly occupied reflection pairs whose occupied points originally lie in opposite half-planes. Hence \cref{lem:polarization-strict} gives strict improvement.

The only remaining pairs are \((t,c)=(1,1)\) and \((1,2)\). In the second case, failure of the same strictness test forces \(M\) to be a hook and \(u=k=p+q-1\). Use instead the path from \((u-p-1,1)\) horizontally to \((u-2,1)\), vertically to \((u-2,q)\), and then to \((u-1,q)\). This path also lies in the strict preferred half-plane and connects reflection pairs occupied on opposite sides before polarization.

In the first case, \(r=s\) and maximality of \(t\) excludes a \(2\times2\) corner block, so \(M\) is a hook. Removing its corner cell leaves two opposite strips of lengths \(p-1\) and \(q-1\), as claimed.
\end{proof}

\subsection{The explicit selector}
The neighboring-hull comparison selects the most elongated admissible hull. To index the successive hull changes from the balanced rectangle, fix \(N\ge1\) and set
\begin{gather}
L=\lceil2\sqrt N\rceil,\qquad
m=\lfloor L/2\rfloor,\qquad
e=L-2m\in\{0,1\},
\label{eq:selector-Lme}\\
d=m(m+e)-N,\qquad
T=\left\lfloor\frac{\sqrt{e^2+4d}-e}{2}\right\rfloor,
\label{eq:selector-dT}\\
r=m-T,\qquad
s=m+e+T,\qquad
k=d-T(T+e).
\label{eq:selector-rsk}
\end{gather}
An integer square root suffices to evaluate \(T\).

Every admissible hull in the strip family, with the shorter side first, is
\begin{equation}
r_t=m-t,
\qquad
s_t=m+e+t,
\qquad
k_t=d-t(t+e),
\qquad
0\le t\le T.
\label{eq:strip-family}
\end{equation}
The strip deficits satisfy
\begin{equation}
k_{t+1}=k_t-(s_t-r_t+1).
\label{eq:strip-step}
\end{equation}

\begin{theorem}
\label{thm:explicit-strip}
Within the entire strip family \eqref{eq:strip-family}, the maximum spanning-tree count is attained by the long-side strip on the hull \(\Rect r s\) from \eqref{eq:selector-rsk}.  Equivalently,
\begin{equation}
\max_{0\le t\le T}\max\{X(r_t,s_t,k_t),Y(r_t,s_t,k_t)\}
=Y(r,s,k).
\label{eq:strip-family-max}
\end{equation}
For \(N\ge18\), this strip construction is unique up to lattice isometry.
\end{theorem}

\begin{proof}
By \cref{prop:orientation}, it suffices to consider long-side strips. Equation \eqref{eq:strip-step} allows \cref{prop:adjacent-hull} at every step, so their counts are nondecreasing up to the largest feasible index \(T\). This proves \eqref{eq:strip-family-max}; the equality list \eqref{eq:adjacent-equality-list} contains no size \(N\ge18\).
\end{proof}

\section{Complete classification}
\label{sec:classification}
\suppressfloats[t]

The canonical strip and the five two-corner configurations in \cref{lem:finite-classification} now give the maximum count and all equality cases.

\begin{theorem}
\label{thm:all-N-classification}
For every \(N\ge1\), each \(N\)-vertex maximizer is obtained from its minimal rectangular hull \(\Rect r s\), with \(r+s=\lceil2\sqrt N\rceil\), by removing at most two straight strips at opposite corners. Its tree count is
\begin{equation*}
M_N=
\begin{cases}
16,&N=7,\\
8960,&N=14,\\
112456,&N=17,\\
\tau(C_N),&\text{otherwise}.
\end{cases}
\end{equation*}
The maximizers are \(C_N\) except at the cardinalities listed in \cref{tab:exceptions}, which gives all maximizing isometry classes. Nonuniqueness occurs exactly at \(N=3,8,10,13\). In particular, \(C_N\) is the unique maximizer up to lattice isometry for every \(N\ge18\).
\end{theorem}

\begin{table}[htbp]
\centering
\renewcommand{\arraystretch}{1.15}
\small
\begin{tabular}{c p{9.5cm} r}
\toprule
\(N\) & Maximizing domains, up to lattice isometry & \(M_N\)\\
\midrule
3 & \(R_{1,3}\); \(R_{2,2}\setminus\{(0,0)\}\) & 1\\
7 & \(R_{3,3}\setminus\{(0,0),(2,2)\}\) & 16\\
8 & \(R_{2,4}\); \(R_{3,3}\setminus\{(0,0)\}\) & 56\\
10 & \(R_{2,5}\); \(C_y(3,4,2)\); \(R_{3,4}\setminus\{(0,0),(2,3)\}\) & 209\\
13 & \(C_y(3,5,2)\); \(R_{4,4}\setminus\{(0,0),(0,1),(3,3)\}\) & 2656\\
14 & \(R_{4,4}\setminus\{(0,0),(3,3)\}\) & 8960\\
17 & \(R_{4,5}\setminus\{(0,0),(0,1),(3,4)\}\) & 112456\\
\bottomrule
\end{tabular}
\caption{Exceptional cardinalities: \(7,14,17\) change the canonical value, while \(3,8,10,13\) admit several maximizing isometry classes.}
\label{tab:exceptions}
\end{table}

\begin{proof}
By \cref{thm:strip-reduction,thm:explicit-strip}, every maximizer is \(C_N\), a strip tied with it, or one of the five two-corner configurations in \cref{lem:finite-classification}. The hull condition follows from \cref{thm:isoperimetric-rigidity}.

The two-corner counts beat the best strip at \(N=7,14,17\) and tie it at \(N=10,13\). The adjacent-hull equalities in \eqref{eq:adjacent-equality-list} add maximizing strips at \(N=3,8,10\). At their other sizes, \(7,14,17\), the two-corner domains have strictly larger counts. These comparisons give every value and isometry class in \cref{tab:exceptions}, including all failures of uniqueness.
\end{proof}

The threshold \(18\) is sharp: at \(N=17\), the two-corner domain in \cref{tab:exceptions} has \(112456\) spanning trees, whereas \(C_{17}\) has \(112272\).

\section{Stability and asymptotics}
\label{sec:stability}

The geometric estimates have two quantitative consequences. At square cardinalities, the hull comparison controls distance from the maximizing square, with a sharp constant. For general cardinalities, the rectangular expansion and the selected strip give the lower-order terms of the maximal count.

\subsection{Sharp stability for squares}
Recall the normalized translational asymmetry \(\An(S)\) from \eqref{eq:intro-asymmetry}, illustrated for an equal-area rectangle in \cref{fig:stability-distance}.
\begin{figure}[htbp]
\centering
\begin{tikzpicture}[x=.44cm,y=.44cm,font=\small]
\fill[gray!12] (-.5,-.5) rectangle (3.5,5.5);
\fill[orange!18] (-.5,5.5) rectangle (3.5,8.5);
\fill[blue!12] (3.5,-.5) rectangle (5.5,5.5);
\foreach \x in {0,...,3}{\foreach \y in {0,...,8}{\fill (\x,\y) circle (1.3pt);}}
\foreach \x in {4,5}{\foreach \y in {0,...,5}{\fill (\x,\y) circle (1.3pt);}}
\draw[thick] (-.5,-.5) rectangle (3.5,8.5);
\draw[thick,dashed] (-.5,-.5) rectangle (5.5,5.5);
\node[anchor=south] at (1.5,8.7) {$S=R_{4,9}$};
\node[anchor=south] at (4.5,5.7) {$Q_6$};
\node[anchor=west] at (7,7.4) {$|S|=|Q_6|=36$};
\node[anchor=west] at (7,6.0) {$|S\cap Q_6|=24$};
\node[anchor=west] at (7,4.6) {$|S\symdiff Q_6|=12+12=24$};
\node[anchor=west] at (7,3.2) {$\mathcal A_6(S)=2/3$};
\fill[orange!18] (7,.9) rectangle (7.7,1.6);
\node[anchor=west] at (8,1.25) {$S\setminus Q_6$: 12 vertices};
\fill[blue!12] (7,-.3) rectangle (7.7,.4);
\node[anchor=west] at (8,.05) {$Q_6\setminus S$: 12 vertices};
\end{tikzpicture}
\caption{The asymmetry compares sets of the same cardinality. Solid and dashed outlines show a $4\times9$ rectangle and a $6\times6$ square; each dot is a vertex. Their overlap has at most four columns and six rows, so the pictured alignment attains the minimum symmetric difference over translations. The two shaded excess regions give its 24 vertices.}
\label{fig:stability-distance}
\end{figure}

The hull inequality and rectangular comparison cancel the bulk terms \(h(rs-n^2)\), giving
\begin{equation}
\log\frac{\tau(\Square n)}{\tau(S)}
\ge\alpha(r+s-2n),
\label{eq:hull-stability}
\end{equation}
where \(r,s\) are the minimal hull side lengths; see \cref{cor:enhanced-hull,prop:rectangular-comparison}.

A geometric estimate converts this hull excess into distance from a translated square.

\begin{lemma}
\label{lem:geometric-asymmetry}
For every \(n^2\)-point set \(S\subset\Z^2\), with minimal rectangular hull side lengths \(r,s\),
\begin{equation}
\An(S)^2\le\frac{4(r+s-2n)}{n}.
\label{eq:geometric-asymmetry}
\end{equation}
\end{lemma}

\begin{proof}
Assume \(r\le s\) and put \(p=r+s-2n\). If \(p=0\), the area bound \(rs\ge n^2\) forces \(r=s=n\) and \(S\) to equal its hull. If \(p\ge n\), use \(\An(S)\le2\). It remains to consider \(0<p<n\); set \(x=p/n\).

If \(r\le n\le s\), set \(y=(n-r)/n\); the area bound gives \(y^2\le x(1-y)\). Choose an \(n\)-square covering all \(r\) hull columns and \(n\) consecutive hull rows. At most \(r(s-n)\) points of \(S\) lie outside it, and equal cardinalities give
\[
\frac{\An(S)}2\le(1-y)(x+y).
\]
With \(z=\sqrt x\), the admissible range is \(z\in[y/\sqrt{1-y},1]\). The quadratic
\[
(1-y)z^2-z+y(1-y)
\]
is convex and nonpositive at both endpoints, proving \(\An(S)\le2\sqrt x\).

If \(n\le r\le s\), both sides are below \(2n\), since \(p<n\). An \(n\)-square contained in the hull leaves at most \(rs-n^2\le np+p^2/4\) points of \(S\) outside. Alternatively, the four corner-anchored \(n\)-squares cover the hull, so one contains at least a quarter of \(S\). Combining these two choices gives
\[
\frac{\An(S)}2
\le\min\{x+x^2/4,3/4\}.
\]
If \(x\le9/16\), then \(x+x^2/4\le\sqrt x\); if \(x\ge9/16\), then \(3/4\le\sqrt x\). Again \(\An(S)\le2\sqrt x\), which is \eqref{eq:geometric-asymmetry}.
\end{proof}

\begin{theorem}
\label{thm:sharp-stability}
For every connected \(S\subset\Z^2\) with \(|S|=n^2\),
\begin{equation}
\log\frac{\tau(\Square n)}{\tau(S)}
\ge
\frac{\alpha}{4}\,n\,\An(S)^2.
\label{eq:sharp-stability}
\end{equation}
The constant \(\alpha/4\) is best possible uniformly in \(n\).
\end{theorem}

\begin{proof}
Combining \eqref{eq:hull-stability} with \eqref{eq:geometric-asymmetry} proves the inequality. To see sharpness, take \(r=m^2\), \(s=(m+1)^2\), and \(n=m(m+1)\). The resulting rectangles have area \(n^2\) and hull excess \(r+s-2n=1\), while their aspect ratio tends to one. The exact rectangular product and the remainder asymptotics in Appendix~\ref{app:rectangles} therefore give
\[
\log\frac{\tau(\Square n)}{\tau(\Rect r s)}\longrightarrow\alpha.
\]
Moreover
\[
\An(\Rect r s)=\frac2{m+1},
\qquad
n\An(\Rect r s)^2=\frac{4m}{m+1}\longrightarrow4.
\]
Thus the ratio of the two sides of \eqref{eq:sharp-stability} tends to one.
\end{proof}

In particular, a domain whose tree count is within a fixed multiplicative factor of optimal must differ from a translate of the square in only \(O(n^{3/2})\) vertices.

\subsection{Extremal asymptotics}
Write \(k_N\) for the selected strip length in \eqref{eq:selector-rsk}, and recall that \(E_k\) is the limiting strip energy from \cref{def:corner-energy}. The square expansion in Appendix~\ref{app:rectangles} defines the constant
\begin{equation}
C_{\square}
=\lim_{n\to\infty}
\left(
\log\tau(\Square n)-hn^2+2\alpha n+\frac12\log n
\right).
\label{eq:Csquare}
\end{equation}

\begin{proposition}
\label{prop:extremal-asymptotic}
As \(N\to\infty\),
\begin{equation}
\log M_N
=hN-\alpha\lceil2\sqrt N\rceil
-\frac14\log N
+C_{\square}-E_{k_N}
+O(N^{-1/2}).
\label{eq:extremal-asymptotic}
\end{equation}
\end{proposition}

\begin{proof}
For sufficiently large \(N\), \cref{prop:analytic-cutoff,thm:explicit-strip} give \(M_N=\tau(C_N)\). Both sides of its hull are \(\Theta(\sqrt N)\), so the uniform rectangular expansion in Appendix~\ref{app:rectangles} applies. The hull aspect ratio satisfies \(s/r-1=O(N^{-1/4})\); since the rectangular remainder is stationary at the square, this contributes only \(O(N^{-1/2})\).

Deleting the selected strip costs
\[
hk_N+E_{k_N}+O(N^{-1/2})
\]
by \cref{lem:finite-strip-error}, with \(k_N=O(N^{1/4})\). The hull area \(N+k_N\) contributes an extra bulk term \(hk_N\), which cancels the strip's bulk cost. The boundary and logarithmic terms give \eqref{eq:extremal-asymptotic}.
\end{proof}

The bulk, boundary, and logarithmic terms are common to all cardinalities. The bounded correction \(-E_{k_N}\) retains the dependence on the strip selected by the integer geometry.

\section{Exact counts for corner-strip domains}
\label{sec:counting}

For \(N\notin\{7,14,17\}\), the maximum \(M_N\) is the tree count of a rectangle with one corner strip removed. Planar duality separates the rectangle count from the effect of the deletion, expressed as a boundary Green minor. Layer elimination is standard for block-tridiagonal determinants \cite{Molinari2008}, and grid counts admit transfer recurrences \cite{Raff2008}. A single scalar resolvent represents all these minors as Hankel determinants; a signed Euclidean chain then gives one product formula for all strip lengths.

\subsection{Planar dual determinant and a boundary Green minor}
For the corner-strip domain \(C(r,s,k)\) defined in the introduction, fix integers \(r,s\ge2\), put \(R=r-1\) and \(w=s-1\), and write \(T_k=\tau(C(r,s,k))\). The endpoint \(k=w\) leaves an \((r-1)\)-by-\(s\) rectangle together with a pendant vertex, so the whole range \(0\le k\le w\) is connected. Let \(P_w\) be the path on \(w\) vertices and set \(B=4I-A_{P_w}\).

Introduce the monic continuants
\begin{equation}
f_{-1}(x)=0,
\qquad
f_0(x)=1,
\qquad
f_j(x)=x f_{j-1}(x)-f_{j-2}(x).
\label{eq:continuants}
\end{equation}
Eliminating the \(R\) layers of the dual face matrix gives the rectangle count and its boundary Green block:
\begin{equation*}
T_0=\det f_R(B),
\qquad
G=f_{R-1}(B)f_R(B)^{-1}.
\end{equation*}
Removing the corner strip deletes its first \(k\) boundary faces, so Jacobi's complementary-minor identity gives
\begin{equation}
T_k=T_0\det G[1:k,1:k].
\label{eq:green-minor}
\end{equation}

\subsection{A scalar resolvent and its Hankel minors}
To replace the matrix minors by scalar data, define
\begin{equation}
P(x)=f_w(x-4).
\label{eq:chi-psi}
\end{equation}
The polynomial \(P\) is the characteristic polynomial of \(B\), while \(f_{w-1}(z-4)/P(z)\) is the endpoint resolvent of \(B\). In \(\mathbb Q[x]/(P)\), set
\begin{align}
g(x)&\equiv f_{R-1}(x)f_R(x)^{-1}\pmod{P(x)},\label{eq:g-poly}\\
q(x)&\equiv g(x)f_{w-1}(x-4)\pmod{P(x)},\qquad \deg q<w.\label{eq:q-poly}
\end{align}
The modular inverse exists because every root of \(P\) lies in \((2,6)\), where \(f_R\) is positive.

\begin{proposition}
\label{prop:scalar-resolvent}
\label{prop:hankel-representation}
Let \(e_1\) be the first coordinate vector and \(m_j=e_1^{\mathsf T}GB^je_1\) for \(j\ge0\). The scalar resolvent has the rational representation and Laurent expansion
\[
\mathcal F(z):=e_1^{\mathsf T}G(zI-B)^{-1}e_1
=\frac{q(z)}{P(z)}
=\sum_{j\ge0}m_jz^{-j-1},
\]
where the rational identity holds off the spectrum of \(B\) and the series converges for \(|z|>6\). For every \(0\le k\le w\),
\begin{equation}
\det G[1:k,1:k]
=\det(m_{i+j})_{0\le i,j<k}.
\label{eq:hankel-representation}
\end{equation}
\end{proposition}

\begin{proof}
Since \(G=g(B)\), polynomial division of
\[
\frac{g(z)-g(B)}{zI-B}
\]
shows that the proper rational part of \(g(z)f_{w-1}(z-4)/P(z)\) is \(e_1^{\mathsf T}g(B)(zI-B)^{-1}e_1\). Reducing the numerator modulo \(P\) gives \(q\), and expansion at infinity gives the moments.

For the determinant identity, use the basis \(e_1,Be_1,\ldots,B^{k-1}e_1\) of the first \(k\) coordinate subspace. The vector \(B^je_1\) is supported on its first \(j+1\) coordinates and has \((j+1)\)-st coordinate \((-1)^j\), so the basis matrix is triangular with determinant \((-1)^{k(k-1)/2}\). Since \(B\) is symmetric and commutes with \(G\), the Gram matrix of this basis for the form \(u^{\mathsf T}Gv\) has entries
\[
(B^ie_1)^{\mathsf T}G(B^je_1)
=e_1^{\mathsf T}GB^{i+j}e_1
=m_{i+j}.
\]
The change of basis leaves the Gram determinant unchanged, proving \eqref{eq:hankel-representation}.
\end{proof}

Figure~\ref{fig:counting-reduction} illustrates the reduction to Hankel determinants.
\begin{figure}[htbp]
\centering
\begin{tikzpicture}[x=.45cm,y=.45cm,font=\small]
\node at (1.5,6.25) {(a) 22 vertices};
\draw[gray,dashed] (-.3,-.3) rectangle (3.3,5.3);
\foreach \x in {0,...,3} {
  \foreach \y in {0,...,5} {
    \pgfmathtruncatemacro{\present}{(\x>0)||(\y>1)}
    \ifnum\present=1
      \ifnum\x<3 \draw[gray] (\x,\y) -- ({\x+1},\y); \fi
      \ifnum\y<5 \draw[gray] (\x,\y) -- (\x,{\y+1}); \fi
      \fill (\x,\y) circle (2pt);
    \fi
  }
}
\foreach \y in {0,1} {
  \draw[blue!55!black,fill=white] (0,\y) circle (2pt);
}
\node[anchor=north] at (1.5,-.75) {$C(4,6,2)$};
\draw[->] (3.9,2.5) -- (6.6,2.5);
\node[above,font=\scriptsize] at (5.25,2.5) {dual faces};

\begin{scope}[shift={(8,.5)}]
  \node at (1,5.75) {(b) 13 bounded faces};
  \draw[gray,dashed] (-.4,-.4) rectangle (2.4,4.4);
  \foreach \x in {0,...,2} {
    \foreach \y in {0,...,4} {
      \pgfmathtruncatemacro{\present}{(\x>0)||(\y>1)}
      \ifnum\present=1
        \ifnum\x<2 \draw[gray] (\x,\y) -- ({\x+1},\y); \fi
        \ifnum\y<4 \draw[gray] (\x,\y) -- (\x,{\y+1}); \fi
        \fill (\x,\y) circle (2pt);
      \fi
    }
  }
  \foreach \y in {0,1} {
    \draw[blue!55!black,fill=white] (0,\y) circle (2pt);
  }
  \node[anchor=north] at (1,-1.25) {$13\times13$ face operator};
\end{scope}
\draw[<-] (11.3,2.5) -- (14.5,2.5);
\node[above,font=\scriptsize] at (12.9,2.5) {delete $1,2$};

\begin{scope}[shift={(16.5,.5)}]
  \node at (1,5.75) {(c) Full 15-face rectangle};
  \fill[gray!12] (-.3,-.3) rectangle (.3,4.3);
  \fill[blue!12] (-.3,-.3) rectangle (.3,1.3);
  \draw[gray,dashed] (-.4,-.4) rectangle (2.4,4.4);
  \foreach \x in {0,...,2} {
    \foreach \y in {0,...,4} {
      \ifnum\x<2 \draw[gray] (\x,\y) -- ({\x+1},\y); \fi
      \ifnum\y<4 \draw[gray] (\x,\y) -- (\x,{\y+1}); \fi
      \fill (\x,\y) circle (2pt);
    }
  }
  \foreach \y in {0,1} {
    \fill[blue!55!black] (0,\y) circle (2pt);
  }
  \foreach \y in {0,...,4} {
    \pgfmathtruncatemacro{\index}{\y+1}
    \node[anchor=east,font=\scriptsize] at (-.5,\y) {$\index$};
  }
  \node[anchor=north,align=center] at (1,-1.25)
    {$5\times5$ boundary\\Green block $G$};
\end{scope}

\node at (9.5,-4.0) {$\displaystyle
  \frac{T_2}{T_0}=\det G[1:2,1:2]
  =\begin{vmatrix}m_0&m_1\\m_1&m_2\end{vmatrix}$};
\node at (9.5,-6.0) {$\displaystyle
  \mathcal F(z)=\frac{q(z)}{P(z)}
  =\sum_{j\ge0}m_jz^{-j-1}$};
\node[anchor=north,font=\scriptsize] at (9.5,-7.3)
  {one scalar function for every strip length};
\end{tikzpicture}
\caption{The counting reduction for $C(4,6,2)$. Deleting two corner vertices removes two of the rectangle's 15 unit faces. Open circles mark these deletions; the shaded column selects the boundary Green block $G$ of the full rectangle. Jacobi's identity gives $T_2/T_0=\det G[1:2,1:2]$, and the triangular Krylov basis gives the equal Hankel determinant with moments from $q/P$. Dashed outlines show the original hulls; edges to the exterior root are omitted.}
\label{fig:counting-reduction}
\end{figure}
\FloatBarrier

\subsection{A product formula}
Form the signed Euclidean chain
\begin{equation}
p_0=P,
\qquad
p_1=q,
\qquad
p_{j+2}=-\remop(p_j,p_{j+1}).
\label{eq:sturm-chain}
\end{equation}
The positive Green form ensures that the chain drops one degree at each step. Thus, for \(0\le j<w\), the quotient \(p_j\operatorname{div}p_{j+1}\) is linear with positive leading coefficient \(b_j\), as proved below.

\begin{theorem}
\label{thm:scalar-product}
For every \(0\le k\le w\),
\begin{equation}
T_k
=T_0\prod_{j=0}^{k-1}b_j^{-(2k-2j-1)}.
\label{eq:scalar-product}
\end{equation}
Equivalently, if \(A_0=1\), then
\begin{equation*}
A_{j+1}=b_jA_j,
\qquad
T_{j+1}=\frac{T_j}{b_jA_j^2}.
\end{equation*}
The coefficients \(b_j\) are generated from the explicit polynomials \eqref{eq:chi-psi}--\eqref{eq:q-poly}.
\end{theorem}

\begin{proof}
The same positive form that gives the Hankel determinant also gives the product. On polynomials of degree less than \(w\), use the inner product
\[
\langle u,v\rangle=e_1^{\mathsf T}u(B)Gv(B)e_1.
\]
It is positive definite because \(G\) is positive definite and \(e_1,Be_1,\ldots,B^{w-1}e_1\) form a basis. Let \(Q_j\) be the monic orthogonal polynomial of degree \(j\), with squared norm \(h_j>0\). Gram--Schmidt is a unit triangular change from monomials, so the order-\(k\) Hankel determinant is \(\prod_{j<k}h_j\).

Multiplication by \(x\), taken modulo \(P\), is self-adjoint for this inner product. In the basis \(Q_0,\ldots,Q_{w-1}\), it is tridiagonal: monicity gives subdiagonal entries one, and self-adjointness gives superdiagonal entries \(\beta_j=h_j/h_{j-1}>0\). Write its diagonal entries as \(a_j\). Successive Schur complements give the endpoint resolvent
\[
\mathcal F(z)=\cfrac{h_0}{z-a_0-\cfrac{\beta_1}{z-a_1-
\cfrac{\beta_2}{\ddots-\cfrac{\beta_{w-1}}{z-a_{w-1}}}}}.
\]
For \(w=1\), this means \(h_0/(z-a_0)\). The trailing characteristic polynomials are monic and have consecutive degrees; their three-term recurrence, with \(\beta_j>0\), prevents cancellation. Polynomial division therefore identifies this fraction with the signed chain \eqref{eq:sturm-chain}, giving
\[
b_0=1/h_0,\qquad b_j=1/(\beta_jb_{j-1})>0\quad(1\le j<w).
\]
In particular, every degree drops by one. Substituting these relations into the product of squared norms yields
\[
\det(m_{i+j})_{i,j<k}
=h_0^k\prod_{j=1}^{k-1}\beta_j^{k-j}
=\prod_{j=0}^{k-1}b_j^{-(2k-2j-1)}.
\]
Together with \eqref{eq:green-minor}, this proves \eqref{eq:scalar-product}, including \(k=0\) by the empty-product convention. Dividing consecutive terms gives the recurrence. This is the classical Hankel/continued-fraction factorization associated with orthogonal polynomials; see \cite{BostanSalvySchost2008}.
\end{proof}

\section{Concluding Remarks}
\label{sec:discussion}

In the square lattice, determinant comparisons force minimum edge boundary and select the maximizing corner deletions. The exact formulas count the resulting single-strip family.

A natural next question is whether spanning-tree maximizers in higher-dimensional lattices are also edge-isoperimetric. The Laplacian determinant representation still applies, but the planar dual face operator has no direct counterpart. Extending the hull comparison requires a way to bound the loss from geometric defects without this two-dimensional duality.

For other planar lattices, duality survives but the face operator and its Green function change. The closed-walk approach then calls for new bounds on the entropy, boundary coefficients, and corner energies.

\section*{Acknowledgments}

The author thanks Ndiam\'e Ndiaye and Adrian Vetta for their guidance during his 2023 summer research project on the square-maximality conjecture. He thanks the participants in the SLMath workshop later that year, especially Jamie Tucker-Foltz and Anders Karlsson, for stimulating conversations. He is grateful to Sloan Nietert and Andr\'es Cristi for helpful discussions, and to Cristi for his support.

The rectangular balancing theorem \cite{ZhangRect2026} and the square-maximality result for \(N=n^2\) \cite{ZhangSquare2026} are preliminary stages of this work; the present paper extends the classification to all cardinalities. The author proposed using concavity to organize the extremal comparisons and led the development of uniform geometric arguments to replace comparisons at individual cardinalities.

Under the author's direction, GPT-5.6 Sol and GPT-6 Astra Pro assisted with mathematical exploration, proof development, checking, and writing. Codex GPT-6 Astra also assisted with the Lean formalization and contributed writing suggestions over successive exchanges. The GPT-6-assisted refinements reduced the uniform threshold from \(N\ge65\,281\) to \(N\ge18\) and replaced the computer-assisted proof steps with explicit arithmetic and graph recurrences. The author takes full responsibility for the entire content of the paper.

\appendix
\section{Rectangular comparison and uniform remainders}
\label{app:rectangles}

The rectangular product gives the area--boundary comparison and uniform remainder used in \cref{sec:isoperimetry,sec:corner,sec:stability}. The finite Green-function correction absorbs the difference of remainders, leaving an aspect cost and a strip correction to compare between hulls.

\subsection{Hyperbolic normalization}
\label{subsec:hyperbolic-normalization}
Set
\[
c(x)=2\asinh\left(\sin\frac{\pi x}{2}\right),
\qquad
C_r=\sum_{j=1}^{r-1}c(j/r),
\qquad
\epsilon_r=rh-\alpha-C_r.
\]
For integer \(r\ge1\) and real \(t>0\), define
\begin{equation}
\ell_r(t)=
\sum_{j=1}^{r-1}
\log\frac{\sinh(tc(j/r))}{\sinh c(j/r)},
\qquad
D_r(t)=hrt-\alpha(r+t)-\ell_r(t).
\label{eq:Fr-def}
\end{equation}
The hyperbolic product gives \(\ell_r(t)=\log\tau(\Rect r t)\) for positive integer \(t\), and the same formula interpolates between these values.

The correction to the trapezoidal sum has the Fourier integral representation
\begin{equation}
\epsilon_r
=-\frac1\pi\int_0^\pi
\log(1-e^{-2r\gamma(\phi)})\,d\phi,
\qquad
\gamma(\phi)=2\asinh\left(\sin\frac\phi2\right).
\label{eq:epsilon-integral}
\end{equation}
Here \(c(x)=\gamma(\pi x)\). The logarithmic endpoint singularity is integrable. To obtain the identity, use the cosine coefficients
\[
\gamma_j=-\frac1{j\pi}\int_0^\pi e^{-j\gamma(\phi)}\,d\phi\qquad(j\ge1)
\]
in the \(2r\)-point periodic trapezoidal sum \(\epsilon_r=-2r\sum_{a\ge1}\gamma_{2ra}\). Absolute summability justifies this step; nonnegativity then permits interchange of the integral and series.

\begin{lemma}
\label{lem:epsilon-bound}
For every \(r\ge1\),
\begin{equation*}
\epsilon_r>\frac1{4r}.
\end{equation*}
Moreover \(\epsilon_r/r\) is nonincreasing in integer \(r\).
\end{lemma}

\begin{proof}
Since \(\gamma(\phi)\le\phi\), substitute \(u=2r\phi\) in \eqref{eq:epsilon-integral} and use \(\int_0^\infty-\log(1-e^{-u})\,du=\pi^2/6\). The tail bound \(\int_L^\infty-\log(1-e^{-u})\,du\le e^{-L}/(1-e^{-L})\) yields
\[
\epsilon_r\ge
\frac1r\left(
\frac\pi{12}-\frac1{2\pi(e^{2\pi}-1)}
\right)>\frac1{4r}.
\]
The function \(-\log(1-e^{-2r\gamma(\phi)})\) is positive and decreases with \(r\), so both \(\epsilon_r\) and \(\epsilon_r/r\) decrease.
\end{proof}

The elementary bounds used below are
\begin{gather*}
\frac{29}{25}<h<\frac{1167}{1000},\qquad \frac{22}{25}<\alpha<\frac9{10},\qquad \frac{157}{50}<\pi<\frac{22}{7},\\
e^{4\alpha}>33,\qquad e^{3\alpha}>14,\qquad e^{2\pi}>500.
\end{gather*}
The rational series bounds in Appendix~\ref{app:corner-certificates} prove these inequalities. The exponential bounds use Taylor sums through degrees 15, 19, and 24 at \(88/25,66/25,157/25\), respectively; the lower bound for \(h\) uses the Catalan alternating sum through index 19, which exceeds \(183/200\).

\subsection{Monotonicity in the long side}
The derivative of the normalized remainder is
\begin{equation*}
D_r'(t)=
\epsilon_r-
\sum_{j=1}^{r-1}
\frac{2c(j/r)}{e^{2tc(j/r)}-1}.
\end{equation*}
Concavity of \(c\) implies \(c(j/r)\ge2\alpha j/r\), while \(z/(e^{2tz}-1)\) decreases in \(z>0\). Thus, for \(t\ge r\),
\[
\sum_{j=1}^{r-1}
\frac{2c(j/r)}{e^{2tc(j/r)}-1}
\le
\frac{4\alpha}{r}
\sum_{j\ge1}\frac{j}{e^{4\alpha j}-1}
<\frac1{8r}.
\]
Combining this with \cref{lem:epsilon-bound} gives
\begin{equation}
D_r'(t)>0
\qquad(t\ge r).
\label{eq:Fr-increasing}
\end{equation}

\subsection{The sequence of squares}
Set \(Q_r(t)=\ell_r(t+1)-\ell_r(t)-C_r\). Rectangle symmetry at integer side lengths yields
\begin{equation}
D_{u+1}(u+1)-D_u(u)
=\epsilon_u+\epsilon_{u+1}-Q_u(u)-Q_{u+1}(u).
\label{eq:square-increment}
\end{equation}
For \(u\ge3\), the argument for \eqref{eq:Fr-increasing} bounds \(Q_u(u)\). To bound \(Q_{u+1}(u)\), apply the same sum with \(r=u+1\) and \(t\ge u\), using \(u/(u+1)\ge3/4\) and \(e^{3\alpha}>14\). This gives
\[
Q_u(u)<\frac1{8u},
\qquad
Q_{u+1}(u)<\frac1{3(u+1)}.
\]
Thus \(\epsilon_v>1/(4v)\) makes the right side of \eqref{eq:square-increment} greater than \(1/(8u)-1/[12(u+1)]>0\). For \(u=1,2\), the values \(\tau(R_{1,1})=1\), \(\tau(R_{2,2})=4\), and \(\tau(R_{3,3})=192\), together with the rational bounds on \(h\) and \(\alpha\), give the same conclusion:
\begin{equation}
D_{u+1}(u+1)>D_u(u)
\qquad(u\ge1).
\label{eq:square-sequence}
\end{equation}

\subsection{Equal-area comparison}
To prove \cref{prop:rectangular-comparison}, assume \(r\le s\) and \(rs\ge n^2\). The two monotonicity results settle \(r\ge n\):
\[
D_r(s)\ge D_r(r)\ge D_n(n)
\]
by \eqref{eq:Fr-increasing} and \eqref{eq:square-sequence}.

For \(r<n\), long-side monotonicity first reduces the comparison to equal areas: putting \(t_0=n^2/r\), we have \(s\ge t_0\ge r\) and \(D_r(s)\ge D_r(t_0)\). Set \(v=n/r>1\) and
\[
G_v(z)=\log\frac{1-e^{-2vz}}{1-e^{-2z}}.
\]
Comparison through the integer rectangle \(R_{r,n}\) yields the identity
\begin{align}
\ell_n(n)-\ell_r(n^2/r)
&=(v-1)(rC_n-nC_r)+\Gamma,
\label{eq:equal-area-identity}\\
\Gamma
&=
\sum_{k=1}^{n-1}G_v(rc(k/n))
-
\sum_{j=1}^{r-1}G_v(nc(j/r)).\notag
\end{align}
Pair the \(j\)-th term of the second sum with \(k_j=\lfloor jn/r\rfloor\) in the first. These indices are distinct, and
\[
rc(k_j/n)\le rc(j/r)<nc(j/r).
\]
The function \(G_v\) is positive and strictly decreasing, so each paired difference is positive, as are the unpaired terms in the first sum. Hence \(\Gamma>0\). The monotonicity of the trapezoidal correction also gives
\[
rC_n-nC_r\ge\alpha(n-r).
\]
Substituting in \eqref{eq:equal-area-identity} gives
\[
\ell_n(n)-\ell_r(n^2/r)
>
\alpha(r+n^2/r-2n).
\]
Thus \(D_r(s)\ge D_r(t_0)>D_n(n)\). Returning to \eqref{eq:Fr-def} proves
\[
\log\tau(R_{r,s})-\log\tau(Q_n)
\le h(rs-n^2)-\alpha(r+s-2n),
\]
with equality only at \(r=s=n\).

\subsection{Uniform rectangular remainder}
For \(r\le s\), define the rectangular remainder by
\[
F(r,s)=\log\tau(R_{r,s})-hrs+\alpha(r+s)=-D_r(s).
\]
The hyperbolic product separates the remainder into a trapezoidal correction and an exponentially decaying sum:
\begin{align}
F(r,s)
={}&-\frac12\log r+\frac54\log2-s\epsilon_r
+\sum_{j=1}^{r-1}\log(1-e^{-2s\gamma(\pi j/r)})
\notag\\
&-\frac12\log(1-e^{-4r\alpha}).
\label{eq:phi-exact}
\end{align}
Put \(u=s/r\ge1\), and define
\[
B(u)=\frac54\log2-\frac{\pi u}{12}
+\sum_{j\ge1}\log(1-e^{-2\pi uj}),
\qquad
J(u)=B(u)+\frac14\log u.
\]
For admissible hulls, \(u\le3\): if \(s\ge3r+1\), then \((r+s-1)^2\ge4rs\), contradicting \(N>(r+s-1)^2/4\) and \(rs\ge N\). For these hulls, we prove the uniform expansion
\begin{equation}
F(r,s)
=-\frac14\log(rs)+J(u)+e_{r,s},
\qquad
|e_{r,s}|\le\frac1{2r^2}.
\label{eq:phi-uniform}
\end{equation}

The function \(\gamma\) is concave, with \(\gamma(0)=0\) and \(\gamma(\pi)=2\alpha\). Also \(\gamma''(\theta)=-\sin(\theta/2)/(1+\sin^2(\theta/2))^{3/2}\), so \( |\gamma''(\theta)|\le\theta/2\). The chord bound and two integrations give
\[
\frac{2\alpha}{\pi}\theta\le\gamma(\theta)\le\theta,
\qquad
0\le\theta-\gamma(\theta)\le\frac{\theta^3}{12}.
\]
These estimates control the logarithmic sum in \eqref{eq:phi-exact}. For the trapezoidal correction, differentiate twice more, writing \(v=\sin(\theta/2)\):
\[
\gamma'''(\theta)=-\frac{\sqrt{1-v^2}(1-2v^2)}{2(1+v^2)^{5/2}},\qquad
\gamma''''(\theta)=\frac{v(5-8v^2+2v^4)}{2(1+v^2)^{7/2}}.
\]
The numerator polynomial of \(\gamma''''\) has one zero in \(0<v^2<1\), so \(\gamma'''\) has total variation at most \(3/2\). Two further integrations by parts give \( |\gamma_j+1/(\pi j^2)|\le2/(\pi j^4)\), and the trapezoidal series yields
\begin{equation*}
\left|\epsilon_r-\frac{\pi}{12r}\right|
\le\frac{\pi^3}{360r^3}\le\frac1{2r^3}.
\end{equation*}
For the residual sum, the mean value theorem, the chord bound, and \(0\le\theta-\gamma(\theta)\le\theta^3/12\) give
\[
\sum_{j=1}^{r-1}\left|\log(1-e^{-2s\gamma(\pi j/r)})-\log(1-e^{-2\pi uj})\right|
\le\frac{\pi^3u}{6r^2}\sum_{j\ge1}\frac{j^3}{e^{4\alpha uj}-1}.
\]
For each \(j\), \(u/(e^{4\alpha uj}-1)\) decreases on \(u\ge1\). By \(e^{4\alpha}>33\), \(e^{2\pi}>500\), and the geometric sum for \(\sum j^3a^j\), the retained difference and the omitted terms with \(j\ge r\) contribute a total error of at most \(r^{-2}\) times
\[
\frac{(22/7)^3}{6}\frac{(1/33)(1+4/33+1/33^2)}{(1-1/33)^5}+\frac{500}{499^2}
<\frac5{24}.
\]
Here the tail uses \(r^2\le500^{r-1}\), proved by induction. Consequently
\begin{equation*}
\left|
\sum_{j=1}^{r-1}\log(1-e^{-2s\gamma(\pi j/r)})
-
\sum_{j\ge1}\log(1-e^{-2\pi uj})
\right|
\le\frac5{24r^2}.
\end{equation*}
The endpoint logarithm satisfies
\[
-\frac12\log(1-e^{-4r\alpha})\le\frac1{2(33^r-1)}\le\frac1{64r^2}.
\]
The three errors have total coefficient at most \((22/7)^3/120+5/24+1/64<1/2\), proving \eqref{eq:phi-uniform}.

Rectangle symmetry gives \(J(u)=J(1/u)\): apply the expansion in both orientations to \((r,s)=(bn,an)\) for fixed rational \(u=a/b>0\), then let \(n\to\infty\). The residual estimate has a finite constant for every fixed \(u>0\), including \(u<1\), and continuity extends the identity to all positive \(u\). Uniform exponential convergence on compact positive intervals justifies termwise differentiation, giving \(J'(1)=0\) and
\[
J''(u)=-\frac1{4u^2}-4\pi^2\sum_{j\ge1}\frac{j^2e^{-2\pi uj}}{(1-e^{-2\pi uj})^2}.
\]
For \(u\ge1\), use \(\pi^2<10\) and \(q=1/500\) to obtain
\[
0<-J''(u)\le\frac14+40\frac{q(1+q)}{(1-q)^5}
<\frac13.
\]

At fixed side sum \(L=r+s\), put
\[
K(u)=B(u)+\frac12\log(1+u),
\qquad F(r,s)=-\frac12\log L+K(u)+e_{r,s}.
\]
The symmetry of \(J\) gives \(K(u)=K(1/u)\), hence \(K'(1)=0\).

At \(r=s=n\), \eqref{eq:phi-uniform} gives the limit defining \(C_{\square}\) in \eqref{eq:Csquare}; over all admissible hulls, it gives the uniform rectangular expansion used in \cref{prop:extremal-asymptotic}.

\subsection{Coupling the rectangular corrections}
At fixed side sum, the first lattice correction depends only on aspect ratio. Its symmetry at the square bounds the difference of two errors by their area difference.

\begin{lemma}
\label{lem:coupled-rectangle-error}
Let \(R,r,s,S\) be integers with \(3\le R\le r\le s\le S\), \(r+s=R+S\), and \(S\le3R\). The errors in \eqref{eq:phi-uniform} satisfy
\[
e_{r,s}-e_{R,S}\le\frac{rs-RS}{R^4},
\]
with strict inequality when the two hulls differ. For distinct hulls, \(r>R\), and the more precise estimate is
\begin{equation}
e_{r,s}-e_{R,S}
<\frac{rs-RS}{2R^2r^2}+\frac1{4r^4}+
\begin{cases}
1/(3R^4),&R=3,\\
1/(4R^4),&R\ge4.
\end{cases}
\label{eq:refined-rectangle-error}
\end{equation}
\end{lemma}

\begin{proof}
Write \(u=s/r\), \(L=r+s\), and define the next aspect profile by
\[
K_2(u)=\frac{\pi^3u(1+u)^2}{1440}
\left(1+240\sum_{n\ge1}\frac{n^3}{e^{2\pi un}-1}\right).
\]
We use the following expansion and curvature bounds:
\begin{equation}
\begin{gathered}
\left|e_{r,s}+\frac{K_2(u)}{L^2}\right|<\frac1{4r^4}
\quad(r\ge4,\ 1\le u\le3),\\
K_2'(1)=0,\qquad 0<K_2''(u)<1\quad(1\le u\le3).
\end{gathered}
\label{eq:next-rectangle-term}
\end{equation}
At \(r=3\), the error bound is \(1/(3r^4)\). We first deduce the comparison from these bounds, then establish them by expanding the exact product and controlling the curvature.

At fixed \(L\), put \(q=(s-r)^2/4\). The area is \(L^2/4-q\), and
\[
\frac{du}{dq}=\frac{L}{r^3(u-1)},\qquad
0\le\frac{d}{dq}\frac{K_2(u)}{L^2}
<\frac1{Lr^3}\le\frac1{2r^4}.
\]
At the square, take the continuous right derivative. For distinct hulls, put \(g=rs-RS\); integration in the short side gives
\[
\begin{aligned}
\frac{K_2(S/R)-K_2(s/r)}{L^2}
&<\int_R^r\frac{L-2x}{Lx^3}\,dx\\
&=\frac{g}{2R^2r^2}
-\frac{(r-R)(L-2r)(L-2R)}{2LR^2r^2}
\le\frac{g}{2R^2r^2}.
\end{aligned}
\]
Adding the two fourth-order errors proves \eqref{eq:refined-rectangle-error}. For \(R\ge4\), it also gives
\[
e_{r,s}-e_{R,S}
<\frac{g}{2R^4}+\frac1{4r^4}+\frac1{4R^4}
\le\frac{g+1}{2R^4}\le\frac{g}{R^4}.
\]
If \(R=3\), distinct hulls have \(r\ge4\), and the same argument bounds the error difference by
\[
\frac{g}{3^4}\left(\frac12+\frac{(3/r)^4/4+1/3}{g}\right)
\le\frac{2803}{3072}\frac{g}{3^4}<\frac{g}{3^4}.
\]
Equal hulls give zero. It remains to prove \eqref{eq:next-rectangle-term}, treating the trapezoidal correction, the logarithmic sum, and the curvature in turn.

Put \(z=\sin^2(\theta/2)\). Direct differentiation gives
\begin{align*}
\gamma^{(5)}(\theta)
&=\frac{\sqrt{1-z}\,(5-54z+42z^2-4z^3)}{4(1+z)^{9/2}},\\
\gamma^{(6)}(\theta)
&=-\frac{\sqrt z\,(79-374z+384z^2-104z^3+4z^4)}{4(1+z)^{11/2}}.
\end{align*}
The quartic in the second numerator is strictly convex on \([0,1]\), with endpoint values \(79,-11\) and derivatives \(-374,98\). The quartic's unique zero makes \(\gamma^{(5)}\) first decrease and then increase, with endpoint values \(5/4\) and zero. To bound its negative part, note that
\[
10-34z+72z^2+16z^3+5z^4
=72(z-17/72)^2+431/72+16z^3+5z^4>0.
\]
This is \(5(1+z)^4\) plus the numerator polynomial of \(\gamma^{(5)}\). Hence \(|\gamma^{(5)}|\le5/4\) and its total variation is at most \(15/4\). Six integrations by parts in the cosine coefficients, followed by the trapezoidal identity, give the first bound below. Taylor's theorem at zero gives the second.
\begin{gather}
\left|\epsilon_r-\frac\pi{12r}-\frac{\pi^3}{1440r^3}\right|
\le\frac{\pi^5}{6048r^5},\notag\\
0\le\frac{\theta^3}{12}-(\theta-\gamma(\theta))\le\frac{\theta^5}{96}
\qquad(0\le\theta\le\pi).
\label{eq:next-trapezoidal-term}
\end{gather}
The coefficient in the first line uses the endpoint contribution \(5/4\), the variation bound \(15/4\), and \(\sum n^{-6}=\pi^6/945\).

For the logarithmic sum in \eqref{eq:phi-exact}, put \(f_s(x)=\log(1-e^{-2sx})\). Taylor's theorem between \(\gamma(\theta)\) and \(\theta\) yields
\begin{align*}
\left|f_s(\gamma(\theta))-f_s(\theta)
+\frac{s\theta^3}{6(e^{2s\theta}-1)}\right|
\le{}&\frac{s\theta^5}{48(e^{2s\theta}-1)}\\
&+\frac{s^2\theta^6}{72}
\frac{e^{-2s\gamma(\theta)}}{(1-e^{-2s\gamma(\theta)})^2}.
\end{align*}
The leading cubic sum combines with \eqref{eq:next-trapezoidal-term} to give \(-K_2(u)/L^2\). Substituting \(\theta=\pi n/r\) makes both displayed errors of order \(r^{-4}\); their coefficients are bounded as follows.

Concavity and \(\gamma(0)=0\) make \(r\gamma(\pi n/r)\) increasing in \(r\). A retained mode has \(r\ge n+1\), so evaluating the first four modes at \(r=\max\{r_0,n+1\}\) gives the uniform bounds \(e^{2r\gamma(\pi n/r)}>D_n\) below. For later modes, the chord bound gives \(e^{2r\gamma(\pi n/r)}>33^n\).
\begin{center}
\begin{tabular}{c r r r r c c}
\toprule
\(r_0\)&\(D_1\)&\(D_2\)&\(D_3\)&\(D_4\)&Tail bound&Coefficient bound\\
\midrule
4&396&37600&550000&22000000&\(1/1000\)&\(1/4\)\\
3&321&12000&550000&22000000&\(1/500\)&\(1/3\)\\
\bottomrule
\end{tabular}
\end{center}
The mode bounds follow from \(S_0(z)=2\), \(S_1(z)=z\), and \(S_{m+1}(z)=zS_m(z)-S_{m-1}(z)\), since \(S_m(x+x^{-1})=x^m+x^{-m}\). Here \(x=e^{\gamma(\theta)}\) and \(x+x^{-1}=4-2\cos\theta\). The needed values are
\begin{align*}
S_8(4\pm\sqrt2)&=276994\pm195584\sqrt2,& S_8(4)&=37634,\\
S_{10}((9+\sqrt5)/2)&=11260002+5004000\sqrt5,&
S_6(3)&=322,\quad S_6(5)=12098.
\end{align*}
The rational bounds \(7/5<\sqrt2<577/408\) and \(\sqrt5>11/5\), verified by squaring, give the four mode bounds in the table. In the tight first entry, \(S_8(4-\sqrt2)>20198/51>396+1/396\). The functions \(u/(e^{au}-1)\) and \(u^2e^{-au}/(1-e^{-au})^2\) decrease for \(u\ge1\) and \(a>0\); for the latter, use that \(\sinh x/x\) increases. Thus the Taylor-error coefficients are bounded at \(u=1\).

For the remaining sums, write
\[
Q_j(x)=\sum_{n\ge1}n^jx^n
=\left(x\frac{d}{dx}\right)^j\frac{x}{1-x}.
\]
Put \(p=22/7\), \(q=1/500\), and \(a=1/33\). For either row of the table, the total fourth-order coefficient is at most
\begin{align*}
&\frac{3p^5}{6048}+\frac{p^5Q_5(q)}{48(1-q)}\\
&\quad+\frac{p^6}{72}\left[
\sum_{n=1}^4\frac{n^6D_n}{(D_n-1)^2}
+\frac{Q_6(a)-\sum_{n=1}^4n^6a^n}{(1-a)^2}\right]
+\text{tail bound}.
\end{align*}
The tail bound covers the endpoint logarithm and omitted linear and cubic terms. Multiplied by \(r^4\), their bounds are
\[
\frac{r^4}{2(33^r-1)},\qquad
\frac{r^4q^r}{(1-q)^2},\qquad
\frac{p^3r^5q^r(1+4q+q^2)}{6(1-q)^5}.
\]
Each decreases for integer \(r\ge3\); their sum is below \(1/1000\) at \(r=4\) and \(1/500\) at \(r=3\). Substitution proves the table's coefficient bounds and hence the expansion in \eqref{eq:next-rectangle-term}, including its \(r=3\) version. The geometric sums converge for every fixed \(u>0\), so the same argument gives an \(O_u(r^{-4})\) remainder there.

The symmetry argument for \(J\) also applies to \(K_2\). Apply the expansion in both orientations along integer dilates of a rectangle of rational aspect ratio, cancel the symmetric principal profile, multiply by \(L^2\), and pass to the limit. This gives \(K_2(u)=K_2(1/u)\); continuity extends the identity to all \(u>0\), and hence \(K_2'(1)=0\).

To control the curvature, expand \(1/(e^{2\pi un}-1)\) as a geometric series and put \(f(u)=u(1+u)^2\). Apart from its positive factor, each series term in \(K_2\) is \(f(u)e^{-bu}\), with \(b=2\pi n\ell\ge6\). For \(j=2\), the coefficients of \((\partial_u-b)^j f(u)\) in powers of \(u-1\) are
\[
4b^2-16b+10,\quad8b^2-20b+6,\quad5b^2-6b,\quad b^2.
\]
For \(j=4\), they are
\[
4b(b^3-8b^2+15b-6),\quad4b^2(2b^2-10b+9),\quad b^3(5b-12),\quad b^4.
\]
All are positive for \(b\ge6\); the first cubic factor is \((b-6)^3+10(b-6)^2+27(b-6)+12\). Hence \(K_2''>0\) and \(K_2^{(4)}>0\), making \(K_2''\) convex. Its maximum on \([1,3]\) is at an endpoint.

The exponential Taylor sum at \(157/25\) gives \(e^{2\pi}>530\). Set \(q=1/530\). Grouping the series by \(n\ell\), use \(\sum_{d\mid n}d^3\le(5/4)n^3\), with the exact value one at \(n=1\), and \(\pi^2<10\), \(\pi>25/8\). This yields
\begin{align*}
K_2''(1)&\le\frac{p^3}{1440}\left[10-4200q
+300\bigl(160Q_5(q)-100Q_4(q)+10Q_3(q)\bigr)\right]<1,\\
K_2''(3)&\le\frac{p^3}{1440}
\left[22+300\bigl(1920Q_5(q^3)+22Q_3(q^3)\bigr)\right]<\frac12.
\end{align*}
The divisor bound follows from \(\sum_{n\ge1}n^{-3}<5/4\). These rational endpoint bounds prove the remaining assertions in \eqref{eq:next-rectangle-term}, and hence the lemma.
\end{proof}

\subsection{Boundary flux in the Green comparison}
\label{subsec:boundary-flux}
The fourth-order estimate \eqref{eq:finite-long-strip-error} follows by summing the harmonic equation along the defect strip. Use positive face coordinates in \(\mathcal Q_+=\Z_{>0}^2\), with
\[
U=\{1,\ldots,r-1\}\times\{1,\ldots,s-1\},\qquad
M=\{(1,a):1\le a\le k\}.
\]
Let \(G_U=(4I-A_U)^{-1}\). The quadrant Green kernel \(G_{\mathcal Q_+}\) is the entrywise increasing limit of \((4I-A_V)^{-1}\) over expanding finite boxes \(V\subset\mathcal Q_+\), with killing on the coordinate axes. Write \(W=G_{\mathcal Q_+}-G_U\), extending \(G_U\) by zero outside \(U\). Put \(H=G_U[M,M]\). The matrix \(W[M,M]\) is positive semidefinite by finite-volume approximation, and its entries are nonnegative by the killed-walk expansion. Jacobi's identity and the Schur complement give
\[
E_{r,s}(I_k^{\mathsf T})-E_k
\le\tr(H^{-1}W[M,M])
\le\tr((4I-A_{P_k})W[M,M]).
\]
The two Green functions satisfy the same point-source equation inside \(U\), so \(W\) is harmonic there. Summing its equation along \(M\) cancels the internal vertical neighbors and leaves the boundary flux
\begin{equation}
\tr((4I-A_{P_k})W[M,M])
=\sum_{a=1}^k W((2,a),(1,a))+W((1,k+1),(1,k)).
\label{eq:green-flux}
\end{equation}
The lower endpoint vanishes on the absorbing axis. Figure~\ref{fig:green-boundary} marks the remaining pairs, whose endpoints lie inside \(U\) because \(r\ge3\) and \(k+1<s\).

\begin{figure}[htbp]
\centering
\begin{tikzpicture}[x=.48cm,y=.48cm,font=\small]
  \fill[gray!7] (0,0) rectangle (6,7);
  \foreach \x in {1,...,5} {
    \draw[gray!28] (\x,1) -- (\x,6);
  }
  \foreach \y in {1,...,6} {
    \draw[gray!28] (1,\y) -- (5,\y);
  }
  \draw[->,thick] (0,0) -- (7.1,0);
  \draw[->,thick] (0,0) -- (0,8.0);
  \node[anchor=west] at (7.25,0) {$y=0$};
  \node[anchor=south] at (0,8.08) {$x=0$};
  \node[anchor=north east,font=\scriptsize] at (-.08,-.08) {$0$};
  \draw[orange!65!black,thick,dashed] (6,0) -- (6,7.85);
  \draw[orange!65!black,thick,dashed] (0,7) -- (7.1,7);
  \node[anchor=south,text=orange!65!black] at (6,7.98) {$x=r$};
  \node[anchor=west,text=orange!65!black] at (7.25,7) {$y=s$};

  \fill[blue!10,rounded corners=3pt] (.65,.65) rectangle (1.35,3.35);
  \foreach \x in {1,...,5} {
    \foreach \y in {1,...,6} {
      \fill[black!70] (\x,\y) circle (1.45pt);
    }
  }
  \foreach \y in {1,2,3} {
    \fill[blue!55!black] (1,\y) circle (2.2pt);
  }
  \foreach \y in {1,2,3} {
    \draw[green!45!black,line width=1pt] (1,\y) -- (2,\y);
  }
  \draw[green!45!black,line width=1pt] (1,3) -- (1,4);
  \draw[blue!55!black,fill=white,line width=.7pt] (1,2) circle (3.15pt);
  \fill[blue!55!black] (1,2) circle (1.7pt);
  \draw[blue!55!black,thin] (1.16,2.08) -- (2.10,2.45);
  \node[anchor=west,fill=white,inner sep=1pt,text=blue!55!black]
    at (2.13,2.45) {$(1,j)$};
  \node[fill=white,inner sep=2pt] at (3.65,5.45) {$U$};

  \draw[gray!65,densely dotted] (1,0) -- (1,-.72);
  \draw[gray!65,densely dotted] (6,0) -- (6,-.72);
  \draw[<->,gray!80] (1,-.55) -- (6,-.55);
  \node[anchor=north,font=\scriptsize] at (3.5,-.64) {$r-1$};
  \draw[gray!65,densely dotted] (0,2) -- (-.80,2);
  \draw[gray!65,densely dotted] (0,7) -- (-.80,7);
  \draw[<->,gray!80] (-.62,2) -- (-.62,7);
  \node[anchor=east,font=\scriptsize] at (-.74,4.5) {$s-j$};

  \node[anchor=north west,align=left,text width=6.35cm]
    at (9.3,7.85) {Positive face coordinates\\[.3em]
      $U=\{1,\ldots,r-1\}\times\{1,\ldots,s-1\}$};
  \fill[blue!55!black] (9.48,5.4) circle (2.2pt);
  \node[anchor=west,align=left] at (10.0,5.4)
    {$M=\{(1,a):1\le a\le k\}$};
  \node[anchor=north west,align=left,text width=6.35cm]
    at (9.3,4.5) {Green segments mark the neighbor pairs left by
      harmonic cancellation along $M$.};
  \node[anchor=north west,align=left,text width=6.35cm]
    at (9.3,2.6) {Quadrant walks lost on restriction to $U$ meet
      $x=r$ or $y=s$, possibly both.};
  \node[anchor=north west,font=\scriptsize,text=black!65]
    at (9.3,.30) {Shown: $r=6$, $s=7$, $k=3$, $j=2$.};
\end{tikzpicture}
\caption{The finite Green comparison in positive face coordinates. Solid axes remain absorbing for both kernels; dashed lines are the additional absorbing boundaries of $U$. The blue prefix $M$ consists of marked boundary sites of the full Green matrix, before Jacobi deletion. Internal vertical pairs cancel in \eqref{eq:green-flux}, leaving the $k$ horizontal pairs and one upper endpoint pair marked in green. Double arrows give the separations $r-1$ and $s-j$ from the ringed site $(1,j)$ to the remote sides.}
\label{fig:green-boundary}
\end{figure}

Each entry of \(W\) counts walks lost on restriction to \(U\). Every lost walk reaches \(x=r\) or \(y=s\), so we bound the two boundary losses separately. For sites \((a,b),(a',b')\), imposing only \(x=r\) gives, with \(\gamma=\gamma(\theta)\),
\begin{equation}
\frac4\pi\int_0^\pi
\frac{\sin(b\theta)\sin(b'\theta)\sinh(a\gamma)\sinh(a'\gamma)}
{\sinh\gamma\,(e^{2r\gamma}-1)}\,d\theta.
\label{eq:boundary-loss-kernel}
\end{equation}
Use \(|\sin(b\theta)\sin(b'\theta)|\le bb'\sin^2\theta\), substitute \(t=\gamma(\theta)\), and put \(c=\max(a,a')\). Since \(\sinh(xt)/x\) increases in \(x>0\), this loss is at most
\[
\frac{4aa'bb'}{\pi c^2}\int_0^\infty
\frac{t\sinh^2(ct)}{e^{2rt}-1}\,dt.
\]
Expansion of the positive integrand gives
\[
\begin{aligned}
&\frac4\pi\int_0^\infty\frac{t\sinh^2(ct)}{e^{2rt}-1}\,dt\\
&\quad=\frac1{4\pi}\sum_{n\ge1}
\left[\frac1{(rn-c)^2}+\frac1{(rn+c)^2}-\frac2{(rn)^2}\right]\\
&\quad\le\frac{\pi^3c^2}{60(r^2-c^2)^2}.
\end{aligned}
\]
Indeed, for \(b>c\), the bracket equals \(2c^2(3b^2-c^2)/[b^2(b^2-c^2)^2]\), which is at most \(6c^2/(b^2-c^2)^2\); now use \((rn)^2-c^2\ge n^2(r^2-c^2)\) and \(\sum n^{-4}=\pi^4/90\). Coordinate exchange handles the other boundary. The entrywise lost-walk bound is therefore
\[
W((a,b),(a',b'))
\le\frac{\pi^3aa'bb'}{60}
\left[\frac1{(r^2-\max(a,a')^2)^2}
+\frac1{(s^2-\max(b,b')^2)^2}\right].
\]
Apply it to the pairs in \eqref{eq:green-flux}, sum \(a^2\), and replace \(a\) by \(k\) in the second denominator. This proves \eqref{eq:finite-long-strip-error}.

\subsection{Finite singleton corrections}
Every nonempty corner incurs a finite boundary correction. The singleton correction is the logarithmic ratio
\[
c_{r,s}=\log\frac{g_{\mathcal Q}}{g_{r,s}},
\qquad
g_{\mathcal Q}=2-\frac{16}{3\pi},
\qquad
g_{r,s}=G_U((1,1),(1,1)).
\]
Here \(U=\{1,\ldots,r-1\}\times\{1,\ldots,s-1\}\), as in the Green comparison above. The quadrant value \(g_{\mathcal Q}=G_{\mathcal Q_+}((1,1),(1,1))\) is computed in \cref{subsec:small-diagram-energies}.

\begin{lemma}
\label{lem:singleton-lower}
For integers \(r,s\ge2\),
\[
c_{r,s}>\frac5{3\min\{r,s\}^4},
\qquad
c_{r,r}>\frac{12}{5r^4}.
\]
\end{lemma}

\begin{proof}
By symmetry, it suffices to take \(r\le s\). Assume first that \(r\ge4\), and impose only the boundary \(x=r\), enlarging \(U\) to an infinite strip. The resulting loss from the quadrant is
\[
W_r=\frac4\pi\int_0^\pi
\frac{\sin^2\theta\sinh\gamma(\theta)}{e^{2r\gamma(\theta)}-1}\,d\theta,
\qquad
c_{r,s}\ge\frac{W_r}{g_{\mathcal Q}}.
\]
Under \(t=\gamma(\theta)\), put \(z=\sinh^2(t/2)\). The numerator, including the change of variable, becomes
\[
\sin\theta(t)\sinh^2t\,dt
=8z^{3/2}(1+z)\sqrt{1-z}\,dt.
\]
For \(0\le t\le5/4\), we have \(z\le1/2\) because \(\cosh(5/4)<2\). Then \((1+z)^2(1-z)=1+z-z^2-z^3\ge1\) and \(2\sqrt z=2\sinh(t/2)\ge t\), so the numerator is at least \(t^3\,dt\). Integrating over this interval and using \(\int_0^\infty t^3/(e^{2rt}-1)\,dt=\pi^4/(240r^4)\) gives
\[
r^4W_r\ge\frac{\pi^3}{60}
-\frac4\pi r^4\int_{5/4}^\infty\frac{t^3}{e^{2rt}-1}\,dt.
\]
Bounding the denominator by its geometric tail shows that the scaled integral is at most
\[
\frac{125r^3/128+75r^2/64+15r/16+3/8}{e^{5r/2}-1}.
\]
This decreases for integer \(r\ge4\): the numerator grows by at most \((5/4)^3\) at each step, while the denominator grows by more than \(e^{5/2}\). Taylor's series gives \(e^{10}>20000\), so evaluation at \(r=4\) yields
\[
r^4W_r>
\frac{(157/50)^3}{60}
-\frac43\frac{683/8}{19999}
>\frac{50}{99}.
\]
Together with \(g_{\mathcal Q}<10/33\), this proves the first bound.

For the square, the second boundary supplies an additional loss. Let \(g_{\mathrm{strip}}\) be the corner Green entry on the infinite strip of width \(r-1\). The half-line and finite-path endpoint resolvents are \(e^{-\gamma}\) and \(\sinh((r-1)\gamma)/\sinh(r\gamma)\), so separation of variables gives
\[
g_{\mathrm{strip}}-g_{r,r}
=\frac4r\sum_{j=1}^{r-1}
\frac{\sin^2(\pi j/r)\sinh\gamma(\pi j/r)}
{e^{2r\gamma(\pi j/r)}-1}.
\]
It suffices to retain the first mode. Put \(\theta=\pi/r\le\pi/4\) and \(x=\theta^2<5/8\). Since \(\gamma^{(5)}\le5/4\) on \([0,\pi/4]\), Taylor's theorem yields
\[
\gamma(\theta)\le\theta-\frac{\theta^3}{12}+\frac{\theta^5}{96}
\le\theta-\frac{\theta^3}{15}.
\]
Also \(\sin^2\theta\ge\theta^2(1-x/3)\) and \(\sinh\gamma(\theta)\ge\theta\). For the latter inequality, the derivative of \(v\sqrt{1+v^2}-\arcsin v\) is nonnegative on \(0\le v\le1/\sqrt2\); take \(v=\sin(\theta/2)\). Since
\[
e^{-2x/5}\le1-\frac{2x}{5}+\frac{2x^2}{25}
\le1-\frac{7x}{20}<1-\frac x3,
\]
and \(2r\gamma(\theta)<2\pi-2x/5\), the first summand gives
\[
g_{\mathrm{strip}}-g_{r,r}
>\frac{4\pi^3}{r^4(e^{2\pi}-1)}
>\frac{3g_{\mathcal Q}}{4r^4}.
\]
The last step uses \(\pi>157/50\), \(g_{\mathcal Q}<10/33\), and \(e^{2\pi}<540\). For the exponential upper bound, sum the series at \(44/7\) through degree 24 and bound its tail geometrically. Adding this second boundary loss to \(W_r\) and using \(-\log(1-y)>y\) gives
\[
c_{r,r}>\frac1{r^4}\left(\frac53+\frac34\right)
>\frac{12}{5r^4}.
\]

For \(r=2,3\), the infinite-strip endpoint resolvents give
\[
g_{2,\infty}=2-\sqrt3<\frac{67}{250},
\qquad
g_{3,\infty}=2-\frac{\sqrt5+\sqrt{21}}4<\frac{591}{2000}.
\]
Squaring \(433/250\), \(559/250\), and \(2291/500\) proves the radical bounds. Since \(g_{\mathcal Q}>151/500\), domain monotonicity and \(\log(g_{\mathcal Q}/g)>1-g/g_{\mathcal Q}\) give the first line below. For the squares, use \(g_{2,2}=1/4\) and \(g_{3,3}=7/24\):
\begin{align*}
c_{2,s}&>\frac{17}{151}>\frac5{48},&
c_{3,s}&>\frac{13}{604}>\frac5{243},\\
c_{2,2}&>\frac{26}{151}>\frac{12}{80},&
c_{3,3}&>\frac{31}{906}>\frac{12}{405}.
\end{align*}
These finish both estimates for every side length.
\end{proof}

\subsection{Finite corrections for short corner defects}
\label{subsec:short-defect-corrections}
For a two-cell strip and a corner square, the full Green matrix gives the following lower bounds on the finite correction.

\begin{lemma}
\label{lem:short-defect-corrections}
For integers \(5\le r\le s\),
\[
E_{r,s}(I_2^{\mathsf T})-E_2>\frac{21}{4r^4},
\qquad
\Energy_{r,s}(2\times2)-\Energy(2\times2)>\frac{20}{r^4}.
\]
\end{lemma}

\begin{proof}
The logarithmic loss is bounded below by a trace, which reduces to an integral with an explicit polynomial factor. For either defect \(M\), put \(M^+=M+(1,1)\) and \(U=\{1,\ldots,r-1\}\times\{1,\ldots,s-1\}\). With \(Q=G_{\mathcal Q_+}[M^+,M^+]\) and \(H=G_U[M^+,M^+]\), the inequality \(0<H\preceq Q\) gives
\[
\Energy_{r,s}(M)-\Energy(M)
=\log\frac{\det Q}{\det H}
\ge\tr\bigl(Q^{-1}(Q-H)\bigr).
\]
This follows from \(-\log(1-u)\ge u\) on the eigenvalues of \(Q^{-1/2}(Q-H)Q^{-1/2}\). Enlarging the rectangle to the infinite strip \(\{1,\ldots,r-1\}\times\Z_{>0}\) can only lower the trace.

For the infinite strip, apply \eqref{eq:boundary-loss-kernel} and set \(t=\gamma(\theta)\), \(x=\cosh t-1\), and \(w(t)=\sin(\theta(t))\sinh^2t\). Factoring \(\sinh t\sin\theta\) from the boundary modes expresses the trace as
\[
\frac4\pi\int_0^{2\alpha}
\frac{w(t)P_M(x)}{e^{2rt}-1}\,dt,
\qquad P_M(x)=v_M(x)^{\mathsf T}Q^{-1}v_M(x).
\]
The two defect profiles give the following bounds for \(x\ge0\):
\[
\begin{array}{c|c|c}
M&v_M(x)^{\mathsf T}&\text{lower bound for }P_M(x)\\ \hline
I_2^{\mathsf T}&(1,\,2-2x)&73/6-(43/2)x+12x^2\\
2\times2&(1,\,2+2x,\,2-2x,\,4-4x^2)&40-36x^2
\end{array}
\]
To prove them, put \(z=1/\pi\) and
\[
a=2-16z/3,\quad b=7/2-32z/3,\quad c=-4+64z/5,\quad d=40-4352z/35.
\]
The quadrant kernel in Appendix~\ref{app:corner-certificates} gives
\[
Q_{2\times2}=\begin{pmatrix}a&b&b&c\\b&4b-c&c&2c\\b&c&4b-c&2c\\c&2c&2c&d\end{pmatrix},
\qquad 4a=1+2b.
\]
Write \(p=a(4b-c)-b^2\), \(e=2b-c\), and \(D=bd-2c^2\). Direct elimination yields
\[
P_{I_2^{\mathsf T}}(x)=\frac{4a-c-2x+4ax^2}{p},\qquad
P_{2\times2}(x)=\frac2b+\frac{4x^2}{e}+\frac{4(e-2bx^2)^2}{bD}.
\]
For the square, the sum of the second and third row equations minus four times the first forces the first component of \(Q^{-1}v_M\) to vanish. The remaining two-variable system gives the displayed expression, which is even with positive quartic coefficient. On \(.3183098<z<.3183100\), the denominators \(p,b,e,D\) are positive. Clearing them reduces the five coefficient comparisons to positivity of
\[
6(4a-c)-73p,\quad43p-4,\quad a-3p,\quad
2D+4e^2-40bD,\quad4D-16e^2+36eD.
\]
Expansion of these polynomials at \(.3183099\), with radius \(10^{-7}\), bounds them below by \(.008\), \(.006\), \(.022\), \(.0008\), and \(.0003\), respectively. Their degrees are at most three, so this proves all the coefficient bounds.

It remains to bound the common weight \(w(t)\). On \(0\le t\le1\), the cosh series gives
\[
\frac{t^2}{2}+\frac{t^4}{24}\le x\le\frac{11t^2}{20}.
\]
The upper bound uses the monotonicity of \((\cosh t-1)/t^2\) and \(\cosh1<31/20\). Moreover,
\[
\begin{aligned}
w(t)^2
&=x^3(8+4x-2x^2-x^3)\\
&\ge8x^3(1+5x/16)\\
&\ge t^6(1+t^2/4)(1+5t^2/32)
\ge t^6(1+t^2/5)^2.
\end{aligned}
\]
The first inequality holds for \(x\le11/20\); the last difference, after division by \(t^6\), is \(t^2(20-3t^2)/3200\). Thus \(w(t)\ge t^3(1+t^2/5)\). Combining this with the two positive polynomial bounds gives, for \(0<t\le1\),
\[
\begin{aligned}
w(t)P_{I_2^{\mathsf T}}(x)
&>t^3\left(\frac{73}{6}-\frac{1127}{120}t^2\right),\\
w(t)P_{2\times2}(x)
&>t^3\left(40+\frac{29}{5}t^2-11t^4\right).
\end{aligned}
\]
For the strip, the omitted coefficients are \(127/200\) and \(3/5\), both positive. For the square, use \(36x^2<11t^4\) and \(t^6\le t^2\).

Integrate these bounds over \([0,1]\), using the complete moments of degrees three, five, and seven: \(\pi^4/(240r^4)\), \(\pi^6/(504r^6)\), and \(\pi^8/(480r^8)\). Extending the negative terms to the half-line and subtracting the tails \(T_j(r)=\int_1^\infty t^j/(e^{2rt}-1)\,dt\), \(j=3,5\), from the positive terms gives
\[
\begin{aligned}
r^4\bigl(E_{r,s}(I_2^{\mathsf T})-E_2\bigr)
&>\frac{73\pi^3}{360}
-\frac{1127\pi^5}{15120r^2}
-\frac{146}{3\pi}r^4T_3(r),\\
r^4\bigl(\Energy_{r,s}(2\times2)-\Energy(2\times2)\bigr)
&>\frac{2\pi^3}{3}
-\frac{160}{\pi}r^4T_3(r)
-\frac{116}{5\pi}r^4T_5(r).
\end{aligned}
\]
In the square bound, the complete quintic term exceeds the seventh-degree term because
\[
\frac{29r^2}{630}-\frac{11\pi^2}{120}
>\frac{59}{252}>0
\qquad(r\ge5).
\]

Integration by parts bounds each scaled tail by
\[
r^4T_j(r)
\le\frac1{e^{2r}-1}
\sum_{\ell=0}^{j}\frac{j!\,r^{3-\ell}}{(j-\ell)!\,2^{\ell+1}}.
\]
At each integer step \(r\ge5\), the numerator grows by at most \((6/5)^3\) and the denominator by more than \(e^2>7\). Thus the tails decrease; using \(e^{10}>22000\) from the exponential series gives
\[
r^4T_3(r)<\frac{683}{8\cdot21999},
\qquad
r^4T_5(r)<\frac{4433}{40\cdot21999}.
\]
Together with \(157/50<\pi<22/7\), these give lower bounds \(5.3006>21/4\) and \(20.4044>20\), respectively, proving the lemma.
\end{proof}

\subsection{Bounds for corner-strip comparisons}
The finite singleton correction absorbs the difference of rectangular remainders. The next three estimates reduce comparisons between hulls to the aspect profile and strip correction, and control both as the sides vary.

\begin{lemma}
\label{lem:normalized-rectangle-error}
Let \((R,S)\) be the canonical hull for \(N\), and let \((r,s)\) be another minimum-side-sum hull, with the shorter sides first. If \(R\ge4\), then
\[
e_{r,s}-e_{R,S}\le(rs-RS)c_{r,s},
\]
with strict inequality unless the hulls coincide.
\end{lemma}

\begin{proof}
Use the selector parameters \(m,e,T\) from \eqref{eq:selector-Lme}--\eqref{eq:selector-dT}. For distinct hulls, the source has \(r=m-t\), \(s=m+e+t\), with \(0\le t<T\), while \(R=m-T\), \(S=m+e+T\). Put \(g=rs-RS\) and \(a=r/R\). Admissibility gives \(R\ge T^2-T+1\); together with \(R\ge4\), this implies \(2T+e\le2R\), and hence \(S\le3R\). The refined estimate \eqref{eq:refined-rectangle-error}, divided by \(g/r^4\), gives the upper bound
\[
\frac{a^2}{2}+\frac{1+a^4}{4g}.
\]
If \(T=1\), then \(a\le5/4\), \(g\ge1\), and the displayed bound is at most \(1681/1024<5/3\).

If \(T\ge2\), then \(g=T(T+e)-t(t+e)\ge3\) and \(a\le3/2\). For \(T=2\), the latter follows from \(R\ge4\); for \(T\ge3\), use \(R\ge T^2-T+1\ge2T\). The bound is then at most \(313/192<5/3\). In both cases \cref{lem:singleton-lower} proves the strict comparison. Identical hulls give zero.
\end{proof}

For the aspect cost, the elementary part of the rectangle profile suffices. Define
\[
D(w)=\frac{\pi w}{12}-\frac12\log\left(1+\frac w2\right),\qquad w\ge0.
\]
\begin{lemma}
\label{lem:aspect-cost}
The function \(D\) is increasing and convex, and
\begin{equation}
K(1+a)-K(1+b)\le D(b)-D(a)
\qquad(0\le a\le b).
\label{eq:aspect-difference}
\end{equation}
If \(A\ge B\ge0\) and \(a\le b\), then
\[
D\left(\frac A{x+a}\right)-D\left(\frac B{x+b}\right)
\]
is nonincreasing on \(x+a>0\).
\end{lemma}

\begin{proof}
The terms \(\log(1-e^{-2\pi un})\) in \(K(u)\) increase with \(u\). Dropping their nonpositive difference gives \eqref{eq:aspect-difference}. Moreover,
\[
D'(w)=\frac\pi{12}-\frac1{2(w+2)}>0,
\qquad
D''(w)=\frac1{2(w+2)^2}>0.
\]
In the final difference, the first argument is no smaller than the second and has at least as large a derivative magnitude. Since \(D'\) is positive and increasing, the difference is nonincreasing.
\end{proof}

The strip correction has complementary monotonicity properties. Let \(B_{r,s}(k)\) denote the right side of \eqref{eq:finite-long-strip-error}, with \(B_{r,s}(0)=0\). The same expression defines it for real arguments with \(r>2\), \(k\ge0\), and \(s>k+1\).

\begin{lemma}
\label{lem:flux-monotonicity}
For fixed \(k\), \(B_{r,s}(k)\) is nonincreasing in either side; for fixed sides it increases with \(k\) in its domain. If \(a\ge5\) and \(b\ge a\), then \(B_{a+x,b+x}(x)\) increases on \(0\le x\le a-1\).
\end{lemma}

\begin{proof}
The first assertions follow from the numerators and denominators. For the simultaneous variation, consider the four summands in the flux formula. The logarithmic derivative of the short-side cubic term is at least
\[
\frac3{x+1}-\frac4{x+a-2}>0
\qquad(0<x\le a-1).
\]
For the short-side quadratic term, the derivative has the sign of
\[
(a-1-x)\bigl(2x^2+(2a+1)x+a+1\bigr),
\]
which is nonnegative. The long-side denominators reduce to squares of positive linear functions. After positive factors are removed, the two derivative numerators are
\[
4x^3+6bx^2+(6b-2)x+b,
\qquad 2bx+b+1,
\]
and both are positive. Continuity handles \(x=0\).
\end{proof}

For a canonical hull with parameters \(m,e,T\), write \(R=m-T\), \(S=m+e+T\), and \(j=RS-N\). If \(T\ge4\), then
\begin{equation}
B_{R,S}(j)<\frac{8\pi^3}{45T^4(T+1)}.
\label{eq:flux-imbalance}
\end{equation}
Indeed, \(j\le2T+e\) and \(S-j\ge R\), so
\[
B_{R,S}(j)\le\frac{\pi^3j(j+1)(j+2)}{45(R^2-4)^2}.
\]
Put \(x=j+1\) and \(a=T^2-T-2\). The canonical size relation gives \(R-2\ge a+x\), while \(x\le2T+2\le3a\). Since \(x^3/(a+x)^4\) increases up to \(3a\), its endpoint bound yields
\[
B_{R,S}(j)<\frac{\pi^3(2T+2)^3}{45(T^2+T)^4}
=\frac{8\pi^3}{45T^4(T+1)},
\]
as claimed.

\section{Corner-energy bounds and finite comparisons}
\label{app:corner-certificates}

Corner energies and rectangular estimates together prove concentration at one corner. This appendix supplies their explicit bounds and reduces the remaining small hulls to graph recurrences.\footnote{Supplementary Python programs reproduce the numerical bounds and finite comparisons below. All comparisons follow from the identities, recurrences, and rational error bounds stated in the paper; the proof does not require running the programs. See the accompanying README for commands and a table matching the displayed decimal bounds.}

\subsection{The local electrical calculation}
Let $F_5$ be the free nearest-neighbor grid on $\{0,\ldots,4\}^2$, with Laplacian $L_5$ using its ordinary vertex degrees. Set $b_E=\mathbf e_{(2,2)}-\mathbf e_{(3,2)}$ and $b_N=\mathbf e_{(2,2)}-\mathbf e_{(2,3)}$. Grounding $(4,4)$ removes the additive ambiguity in the two potentials.

The following integer potential $V_E$ has rows indexed by $y=4,3,2,1,0$ and columns by $x=0,1,2,3,4$:
\[
V_E=\begin{pmatrix}
350&322&241&72&0\\
378&375&329&-25&-72\\
409&471&725&-429&-191\\
378&375&329&-25&-72\\
350&322&241&72&0
\end{pmatrix}.
\]
Substitution at the 25 vertices gives $L_5V_E=2200b_E$ and $V_E(4,4)=0$. Exchanging coordinates gives $V_N(x,y)=V_E(y,x)$, which satisfies $L_5V_N=2200b_N$ with the same ground. The two current potentials are therefore $\phi_E=V_E/2200$ and $\phi_N=V_N/2200$.

The transfer-current entries are $K_{ab}=b_a^{\mathsf T}\phi_b$, for $a,b\in\{E,N\}$. In particular,
\[
K_{EE}=\frac{725-(-429)}{2200}=\frac{577}{1100},\qquad
K_{EN}=\frac{725-329}{2200}=\frac9{50}.
\]
Symmetry gives the other entries, and hence
\[
K_5=\begin{pmatrix}577/1100&9/50\\9/50&577/1100\end{pmatrix},\qquad
\det(I-K_5)=\frac{9373}{48400}>\frac3{16}.
\]
The last inequality is the local electrical estimate in \cref{lem:local-estimates}.

\subsection{Rational bounds for the strict constants}
The common constant bounds \eqref{eq:elementary-constants}, proved below, give \(h<7/6\) and \(\alpha>7/8\). Hence
\[
2(h-\alpha)<\frac7{12}<\frac35<\frac78<\alpha,
\]
which proves \eqref{eq:strict-constants}. For the interior vacancy cost, note that \((13/16)^2/4<1/6\). With \(v=5/7\), four terms of the positive logarithm series give
\[
\log6>2\sum_{j=0}^{3}\frac{v^{2j+1}}{2j+1}
>\frac76+\frac35>h+\delta.
\]
Thus \(\log6-h>\delta\), as required in \cref{prop:interior-surcharge}.

\subsection{An upper bound for strip energy}
The boundary Green matrix separates the bulk cost of a strip from its limiting energy. For \(z>4\), set
\[
\gamma_z(\theta)=\operatorname{arcosh}(z/2-\cos\theta),
\qquad f_z(\theta)=e^{-\gamma_z(\theta)},
\]
and distinguish the coefficients of the symbol from those of its logarithm:
\begin{align*}
a_j(z)&=\frac1\pi\int_0^\pi f_z(\theta)\cos(j\theta)\,d\theta,\\
v_j(z)&=-\frac1\pi\int_0^\pi\gamma_z(\theta)\cos(j\theta)\,d\theta.
\end{align*}
The sine transform in the quadrant gives the Green matrix on the first \(k\) boundary sites:
\begin{align*}
G_{k,z}(i,j)&=\frac2\pi\int_0^\pi f_z(\theta)\sin((i+1)\theta)\sin((j+1)\theta)\,d\theta\\
&=a_{|i-j|}(z)-a_{i+j+2}(z),
\end{align*}
where \(0\le i,j<k\). Integrating the logarithm of the lattice symbol \(z-2\cos\theta-2\cos\phi\) in one variable gives the bulk entropy \(h_z=\pi^{-1}\int_0^\pi\gamma_z(\theta)\,d\theta=-v_0(z)\). Jacobi's identity and the closed-walk expansion therefore identify the regularized strip energy as
\[
E_{k,z}=-\log\det G_{k,z}-kh_z
=-\log\det\bigl(e^{-v_0(z)}G_{k,z}\bigr).
\]

Both \(f_z\) and its real logarithm are even analytic functions. The shifted Toeplitz--Hankel limit of \cite[Propositions 3.1(III), 3.3(III) and Theorem 4.3]{BasorEhrhardt2008} thus applies to this matrix and gives
\[
\lim_{k\to\infty}\bigl(\log\det G_{k,z}-kv_0(z)\bigr)
=-\sum_{j\ge1}v_{2j}(z)+\frac12\sum_{j\ge1}jv_j(z)^2.
\]
Since \(E_{k,z}\) increases with \(k\), the negative of this limit is \(\sup_kE_{k,z}\).

To reach \(z=4\), observe that \(\gamma'_z\) increases to a single maximum of at most one and then decreases for \(z>4\), while \(\gamma'_4\) decreases from one to zero. Two integrations by parts give \(|v_j(z)|\le2/(\pi j^2)\) for \(j\ge1\), uniformly in \(z\ge4\), so both series converge by domination. Nonnegative loop coefficients permit exchanging the suprema over \(k\) and \(z>4\). At the endpoints, \(-\gamma_4(0)=0\) and \(-\gamma_4(\pi)=-2\alpha\); adding their Fourier expansions gives \(\sum_{j\ge1}v_{2j}(4)=(h-\alpha)/2\). Thus, writing \(v_j=v_j(4)\),
\begin{equation*}
C_{\mathrm{strip}}
=\frac{h-\alpha}{2}
-\frac12\sum_{j\ge1}jv_j^2.
\end{equation*}
The coefficient integral in Appendix~\ref{app:rectangles} gives the first five values:
\begin{align*}
v_1&=1-2/\pi,&
 v_2&=2-6/\pi,\\
v_3&=19/3-178/(9\pi),&
 v_4&=24-226/(3\pi),\\
v_5&=501/5-23606/(75\pi).
\end{align*}
Discarding the remaining nonnegative square terms gives
\[
C_{\mathrm{strip}}
\le
\frac{h-\alpha}{2}-\frac12\sum_{j=1}^5jv_j^2
<0.064859<\frac{13}{200}.
\]

\subsection{Exact energies of small diagrams}
\label{subsec:small-diagram-energies}
For the Green-function calculation, use positive face coordinates in \(\mathcal Q_+=\Z_{>0}^2\), with killing on the coordinate axes. The method of images gives
\begin{align*}
G_{\mathcal Q_+}((x,y),(u,v))
={}&\frac14\bigl[-a(|x-u|,|y-v|)+a(x+u,|y-v|)\bigr]\\
&+\frac14\bigl[a(|x-u|,y+v)-a(x+u,y+v)\bigr].
\end{align*}
Here \(a\) is the square-lattice potential kernel normalized by \(a(1,0)=1\). Its diagonal values are
\[
a(j,j)=\frac4\pi\sum_{i=1}^j\frac1{2i-1}.
\]
Together with symmetry and the initial values \(a(0,0)=0\), \(a(1,0)=1\), harmonicity gives the recurrences
\begin{align*}
a(n+1,n)&=2a(n,n)-a(n,n-1),\\
a(m+1,0)&=4a(m,0)-a(m-1,0)-2a(m,1),\\
a(m+1,n)&=4a(m,n)-a(m-1,n)-a(m,n+1)-a(m,n-1).
\end{align*}
Their respective ranges are \(n\ge1\), \(m\ge1\), and \(m>n\ge1\). They fill the adjacent diagonal and then successive diagonals in increasing coordinate order. Every entry is affine in \(t=1/\pi\); the calculations below require only coordinate sums at most fourteen.

For a Ferrers diagram \(M\), put \(M^+=M+(1,1)\). Jacobi's identity on finite quadrant boxes, followed by the nonnegative closed-walk limit, gives
\[
\Energy(M)=-\log\det G_{\mathcal Q_+}[M^+]-h|M|.
\]
The determinant reduces to the northeast boundary of \(M^+\). Write
\[
I=\{v\in M^+:v+(1,0),v+(0,1)\in M^+\},
\qquad B=M^+\setminus I.
\]
No edge joins \(I\) to the rest of the quadrant without meeting \(B\). In a finite quadrant box, block elimination therefore gives
\[
G[I,I]-G[I,B]G[B,B]^{-1}G[B,I]=(4I-A_I)^{-1}.
\]
Taking determinants and then the increasing-box limit gives the polynomial
\[
P_M(t):=\det G_{\mathcal Q_+}[M^+]
=\frac{\det G_{\mathcal Q_+}[B]}{\det(4I-A_I)}.
\]
For the diagrams below, the boundary matrix has order at most six and the interior matrix has order at most four. The interior determinants are \(4,15,56,209\) for paths on one through four vertices, and \(192\) for the \(2\times2\) square.

For a strip, write \(P_j=P_{I_j}\); for example,
\[
P_1(t)=2-\frac{16}{3}t,
\qquad
P_2(t)=\frac{95}{4}-\frac{1984}{15}t+\frac{8192}{45}t^2.
\]
All energy bounds use the same rational intervals:
\begin{equation}
\begin{gathered}
3.141592653589793<\pi<3.141592653589794,\\
1.1662436<h<1.1662437,
\qquad .8813735<\alpha<.8813736.
\end{gathered}
\label{eq:elementary-constants}
\end{equation}
Machin's formula with twelve alternating terms proves the interval for \(\pi\). For \(a_j=(2j+1)^{-2}\) and \(\Delta a_j=a_j-a_{j+1}\), Euler's positive transform gives
\[
\sum_{m=0}^{32}\frac{\Delta^m a_0}{2^{m+1}}
\le\Cat\le
\sum_{m=0}^{32}\frac{\Delta^m a_0}{2^{m+1}}+2^{-33}.
\]
Here \(\Delta^m a_0=\int_0^1(1-x^2)^m(-\log x)\,dx\in[0,1]\), which proves positivity and the tail bound. Together with the interval for \(\pi\), this gives the stated bounds on \(h\). Squaring verifies
\[
1.414213562373095<\sqrt2<1.414213562373096;
\]
the logarithm bound below then gives the interval for \(\alpha\).

For rational \(1\le x\le2\), put \(u=(x-1)/(x+1)\le1/3\). The positive logarithm series gives
\[
0\le\log x-2\sum_{j=0}^{17}\frac{u^{2j+1}}{2j+1}
\le\frac{2u^{37}}{37(1-u^2)}.
\]
Dyadic rescaling treats every positive argument. To bound a determinant polynomial, let \(t_0\) and \(\varepsilon\) be the midpoint and half-width of the reciprocal \(\pi\)-interval. If \(P_M(t_0+v)=\sum_jq_jv^j\), then
\[
q_0-\sum_{j\ge1}|q_j|\varepsilon^j
\le P_M(t)\le
q_0+\sum_{j\ge1}|q_j|\varepsilon^j.
\]
Laplace expansion of the boundary matrices, followed by these polynomial and logarithm bounds, gives the following enclosures. All decimal endpoints denote exact rational numbers; partition notation lists nonincreasing row lengths.
{\small
\begin{longtable}{lccc}
\toprule
Diagram \(M\)&\(|B|\)&\(\det(4I-A_I)\)&Energy enclosure\\
\midrule
\endfirsthead
\toprule
Diagram \(M\)&\(|B|\)&\(\det(4I-A_I)\)&Energy enclosure\\
\midrule
\endhead
\bottomrule
\endfoot
\(L=(2,1)\)&2&4&\([.0671434,.0671438]\)\\
\(2\times2\)&3&4&\([.0567450,.0567455]\)\\
\(2\times3\)&4&15&\([.0652221,.0652228]\)\\
\(2\times4\)&5&56&\([.0705432,.0705441]\)\\
\(2\times5\)&6&209&\([.0742190,.0742201]\)\\
\(3\times3\)&5&192&\([.0755745,.0755755]\)\\
\((3,2)\)&3&15&\([.0818758,.0818764]\)\\
\((4,1)\)&4&4&\([.0793645,.0793651]\)\\
\((3,1,1)\)&4&4&\([.0851385,.0851391]\)\\
\addlinespace
\(I_1\)&1&1&\([.0299353,.0299355]\)\\*
\(I_2\)&2&1&\([.0408316,.0408319]\)\\*
\(I_3\)&3&1&\([.0463698,.0463702]\)\\*
\(I_4\)&4&1&\([.0497336,.0497341]\)\\*
\(I_5\)&5&1&\([.0520012,.0520018]\)\\*
\(I_6\)&6&1&\([.0536372,.0536379]\)\\
\end{longtable}
}
The five-coefficient expression also gives \(C_{\mathrm{strip}}<.064859\). Together with the energy table, this yields
\begin{gather*}
\Energy(L)-C_{\mathrm{strip}}>\frac1{4000},
\qquad
\Energy(2\times3)-C_{\mathrm{strip}}>\frac1{4000},\\
\Energy(2\times2)-E_4>\frac1{4000},
\qquad 2E_1-E_2>\frac{19}{1000}.
\end{gather*}
These are the strict shape and concentration gaps used in \cref{sec:corner}.

\subsection{Bounds in terms of defect size}
The small-diagram energies give lower bounds for every nonstrip configuration. Discrete concavity bounds the competing strips from above.
\begin{lemma}
\label{lem:refined-envelope}
For \(k\ge2\), define \(b(k)\) by
\begin{center}
\begin{tabular}{c@{\qquad}cccccc}
\toprule
\(k\) & 2 & 3 & 4 & 6 & 8 & other \(k\ge2\)\\
\midrule
\(b(k)\) & .05987 & .06714 & .05674 & .06522 & .07054 & .07421\\
\bottomrule
\end{tabular}
\end{center}
Every four-corner configuration of total size \(k\) which is not a single strip satisfies \(\sum_i\Energy(M_i)>b(k)\).
\end{lemma}

\begin{proof}
Every nonempty defect has energy at least \(E_1\), and a defect with at least three maximal cells has energy at least \(3E_1>.07421\) by \cref{lem:maximal-cells}.

With exactly two maximal cells, a nonrectangular defect has two tiers \((b+c)^p,b^q\), where \(b,c,p,q\ge1\). Increasing any parameter duplicates a row or column. Such a two-tier defect of size at least five therefore repeats \((3,2)\), \((4,1)\), \((3,1,1)\), or a transpose: use \(b\ge2\), \(p\ge2\), \(c\ge3\), or \(q\ge3\); the remaining case has \(b=p=1,c=q=2\). The energies of these three diagrams exceed \(.07421\). Smaller nonrectangles dominate \(L\), whose energy exceeds \(.06714\).

Finally, a nonstrip rectangle is \(2\times2,2\times3,2\times4\), or repeats \(2\times5\) or \(3\times3\), up to transpose. Their energy bounds give the listed thresholds for a single defect.

For at least three nonempty corners, the bound \(3E_1\) suffices. Two corners with a nonstrip defect cost at least \(\Energy(2\times2)+E_1>.07421\). For two strips of total size \(k\), concavity minimizes their combined energy at sizes \(1,k-1\); for \(k\ge5\), this is at least \(E_1+E_4>.07421\). At \(k=2,3,4\), use respectively \(2E_1\), \(E_1+E_2\), and \(2E_1\).
\end{proof}

\begin{lemma}
\label{lem:strip-supporting-lines}
For every integer \(k\ge0\),
\[
E_k<\frac{k+27}{600}.
\]
\end{lemma}

\begin{proof}
The six strip-energy intervals give
\[
E_5<\frac4{75},\qquad E_5-E_4>\frac1{600}>E_6-E_5.
\]
Discrete concavity makes the line through \(E_5\) with slope \(1/600\) an upper bound for every \(E_k\). Since \(E_5<4/75=32/600\), this line lies below \((k+27)/600\).
\end{proof}

\subsection{Finite boundary minors}
\label{subsec:finite-boundary-minors}
The short-defect comparisons require eight boundary minors. One continued-fraction calculation bounds them all, using at most seven one-dimensional modes and matrices of order at most four.

For \(r,s\ge2\) and \(0\le j\le s-1\), write
\[
D_j(r,s)=e^{-hj-E_{r,s}(I_j^{\mathsf T})},
\qquad d_j=e^{-hj-E_j},
\]
for the finite and quadrant strip minors, and put \(d_L=e^{-3h-\Energy(L)}\) and \(d_{\square}=e^{-4h-\Energy(2\times2)}\). In positive face coordinates, \(D_j(r,s)\) is the Green minor on \(\{(1,a):1\le a\le j\}\) in \(\{1,\ldots,r-1\}\times\{1,\ldots,s-1\}\). Domain inclusion increases the Green matrix in Loewner order, so \(D_j(r,s)\) is nondecreasing in each side.

\begin{lemma}
\label{lem:finite-boundary-minors}
The following strict bounds hold:
\begin{equation}
\begin{aligned}
D_1(4,6)&>\frac{30013}{100000},&
D_2(4,6)&>\frac{9117}{100000},\\
D_2(4,7)&>\frac{73}{800},&
D_2(5,7)&>\frac{2307}{25000},\\
D_3(5,8)&>\frac{14173}{500000},&
D_3(6,10)&>\frac{5719}{200000},\\
D_4(7,11)&>\frac{111}{12500},&
D_4(8,12)&>\frac{891}{100000}.
\end{aligned}
\label{eq:finite-boundary-floors}
\end{equation}
\end{lemma}

\begin{proof}
Use the continuants \(f_n\) from \eqref{eq:continuants}. Let \(z\) range over the roots of \(f_{r-1}\), and put
\[
\begin{gathered}
x=4-z,\qquad w(z)=\frac{4-z^2}{2r},\qquad N=s-1,\\
h_0(x)=0,\qquad h_n(x)=\frac1{x-h_{n-1}(x)}.
\end{gathered}
\]
All roots satisfy \(-2<z<2\), so \(x>2\) and induction gives \(0<h_n(x)<1\) for \(n\ge1\). If \(1\le a\le b\le k\le N\), elimination of the two tails of the length-\(N\) path gives
\[
R_{ab}(x)=
\frac{\displaystyle\prod_{t=a}^{b-1}h_{N-t}(x)}
{x-h_{a-1}(x)-h_{N-a}(x)}.
\]
The empty product is one. The endpoint spectral weight of \(z=2\cos(\ell\pi/r)\) on the short path is \(w(z)\). Thus the boundary Green block \(H\) satisfies
\[
H_{ab}=\sum_{f_{r-1}(z)=0}w(z)R_{ab}(4-z),
\qquad H_{ba}=H_{ab},\qquad D_k(r,s)=\det H.
\]
Every denominator is positive, so rational root brackets suffice to bound the Green entries. In the following table, an integer \(A\) denotes \([A/10^8,(A+1)/10^8]\); include the negatives and, for even \(r\), the exact root zero.
\[
\begin{array}{c|l}
r&\text{positive root brackets}\\ \hline
4&141421356\\
5&161803398,\ 61803398\\
6&173205080,\ \text{exact }1\\
7&180193773,\ 124697960,\ 44504186\\
8&184775906,\ 141421356,\ 76536686
\end{array}
\]
These brackets follow by substitution and sign change in the factors
\[
\begin{aligned}
f_3(z)&=z(z^2-2),\\
f_4(z)&=(z^2+z-1)(z^2-z-1),\\
f_5(z)&=z(z^2-1)(z^2-3),\\
f_6(z)&=(z^3-z^2-2z+1)(z^3+z^2-2z-1),\\
f_7(z)&=z(z^2-2)(z^4-4z^2+2).
\end{aligned}
\]
The brackets are disjoint and account for the full degree. For \(x\in[l,u]\) and \(h_{n-1}\in[a,b]\), propagate the bounds by
\[
h_n\in\left[\frac1{u-a},\frac1{l-b}\right].
\]
Evaluating \(H\) by positive interval arithmetic, with endpoints rounded outwards to multiples of \(10^{-12}\), gives
\[
|10^7H_{ab}-M_{ab}|<1
\]
for the symmetric integer matrices \(M\) below; each entry error is in fact below \(6\cdot10^{-8}\). The last column lists lower bounds, in units of \(10^{-6}\), for the successive \(LDL^{\mathsf T}\) pivots of \(A=(M-kI)/10^7\).

{
\begin{longtable}{@{}ccc@{}}
\toprule
\((r,s,k)\)&\(M\)&Lower pivot integers\\
\midrule
\endfirsthead
\toprule
\((r,s,k)\)&\(M\)&Lower pivot integers\\
\midrule
\endhead
\bottomrule
\endfoot
\((4,6,2)\)&
\(\begin{pmatrix}3001394&1009547\\1009547&3377371\end{pmatrix}\)&
\(\begin{gathered}300139\\303779\end{gathered}\)\\
\addlinespace
\((4,7,2)\)&
\(\begin{pmatrix}3001799&1010598\\1010598&3380101\end{pmatrix}\)&
\(\begin{gathered}300179\\303986\end{gathered}\)\\
\addlinespace
\((5,7,2)\)&
\(\begin{pmatrix}3014469&1030566\\1030566&3413720\end{pmatrix}\)&
\(\begin{gathered}301446\\306139\end{gathered}\)\\
\addlinespace
\((5,8,3)\)&
\(\begin{pmatrix}
3014722&1031169&400439\\
1031169&3415161&1204253\\
400439&1204253&3495462
\end{pmatrix}\)&
\(\begin{gathered}301471\\306245\\307031\end{gathered}\)\\
\addlinespace
\((6,10,3)\)&
\(\begin{pmatrix}
3019325&1039213&410266\\
1039213&3429591&1222487\\
410266&1222487&3519698
\end{pmatrix}\)&
\(\begin{gathered}301932\\307190\\308334\end{gathered}\)\\
\addlinespace
\((7,11,4)\)&
\(\begin{pmatrix}
3021240&1042694&414753&188173\\
1042694&3435993&1230868&509688\\
414753&1230868&3530928&1282518\\
188173&509688&1282518&3560325
\end{pmatrix}\)&
\(\begin{gathered}302123\\307613\\308936\\309304\end{gathered}\)\\
\addlinespace
\((8,12,4)\)&
\(\begin{pmatrix}
3022162&1044417&417067&190836\\
1044417&3439229&1235253&514789\\
417067&1235253&3536951&1289644\\
190836&514789&1289644&3568952
\end{pmatrix}\)&
\(\begin{gathered}302215\\307828\\309263\\309719\end{gathered}\)\\
\end{longtable}
}

The pivot bounds follow from \(a_{ij}\mapsto a_{ij}-a_{i1}a_{1j}/a_{11}\), with at most three elimination steps per row. Thus \(A\) is positive definite, while the entry errors give \(\|H-M/10^7\|_{\mathrm{op}}<k/10^7\). It follows that \(H\succ A\) and
\[
D_k(r,s)>\det A
\ge\prod_i\frac{p_i}{10^6},
\]
where the \(p_i\) are the listed pivot integers. Their products prove all seven bounds with \(k\ge2\) in \eqref{eq:finite-boundary-floors}. In the first row, the first pivot also gives \(D_1(4,6)>300139/10^6>30013/100000\), completing the eighth bound.
\end{proof}

\subsection{Concentration on canonical hulls}
\label{subsec:terminal-hulls}
On a canonical hull, \(k\le s-r\). The minimum-side-sum condition \((r+s-1)^2<4(rs-k)\le(r+s)^2\) therefore gives
\[
r-k>\frac{(s-r-1)^2}{4},
\qquad
r>\frac{(k+1)^2}{4},\quad s\ge r+k.
\]
The quadratic growth of the hull makes the flux estimate effective for long defects; the fixed Green minors handle the shorter ones.

Combining \cref{lem:finite-boundary-minors} with the quadrant bounds gives
\begin{equation}
\begin{aligned}
D_2(4,6)&>\frac{256}{259}d_1^2,&
D_2(5,7)&>d_1^2,\\
D_3(5,8)&>d_L,\\
D_4(7,11)&>\frac{2401}{2407}d_{\square},&
D_4(8,12)&>d_{\square}.
\end{aligned}
\label{eq:terminal-green-minors}
\end{equation}
These bounds also supply the short-strip energies for the noncanonical comparisons.

\begin{proof}[Proof of \cref{prop:canonical-concentration}]
At \(k=0,1\), every configuration is a long-side strip up to lattice isometry. For \(k\ge5\), the quadrant-energy gap exceeds the flux bound \(B_{r,s}(k)\). Indeed, at \(k=5,6\), \cref{lem:refined-envelope,lem:strip-supporting-lines} give \(b(k)-E_k>1/100\), while the canonical geometry gives \((r,s)\ge(10,15)\) and \((r,s)\ge(13,19)\), respectively. At these least side lengths,
\[
B_{10,15}(5)<\frac1{100},\qquad B_{13,19}(6)<\frac1{100}.
\]
For \(k\ge7\), put \(x=k+1\). Since \(s\ge r+k\), the flux formula gives
\[
B_{r,s}(k)
\le\frac{\pi^3k(k+1)(k+2)}{45(r^2-4)^2}
<\frac{256\pi^3(x^3-x)}{45(x^4-64)^2}.
\]
The last expression decreases for \(x\ge8\): its logarithmic derivative is less than \(4/x-8/x\). At eight, \(\pi<22/7\) bounds it by \(5324/972405<277/50000\). The envelope satisfies \(b(k)\ge.07054\), whereas \(E_k<13/200=.065\), so their difference exceeds \(277/50000\). Thus \eqref{eq:finite-multicorner} gives strict improvement for every \(k\ge5\).

For \(k=2,3,4\), we compare the competing finite Green minors with the strip minor \(D_k(r,s)\). The least nonstrip quadrant losses are respectively \(2E_1\), \(\Energy(L)\), and \(\Energy(2\times2)\). At size three, the other possibility costs at least \(E_1+E_2>\Energy(L)\). At size four, nonrectangles dominate \(L\), while two nonempty corners cost at least \(2E_1>\Energy(2\times2)\). The fixed energy bounds prove each of these comparisons.

At \(k=2\), the canonical condition and \(N\ge18\) force \(r\ge4\): \(r=3\) would give \((r,s,k)=(3,5,2)\) and \(N=13\). For \(r\ge5\), domain monotonicity and \eqref{eq:terminal-green-minors} give \(D_2(r,s)\ge D_2(5,7)>d_1^2\). For \(r=4\), the weaker consequence \(c_{r,s}>3/(2r^4)\) of \cref{lem:singleton-lower} bounds the two-corner minor by
\[
e^{-2c_{4,s}}d_1^2
<e^{-3/256}d_1^2
<\frac{256}{259}d_1^2<D_2(4,6).
\]
Here \(s\ge6\), so the finite strip again wins.

At \(k=3\), the canonical geometry gives \(r\ge5,s\ge8\); the bound \(D_3(r,s)\ge D_3(5,8)>d_L\) excludes every nonstrip. At \(k=4\), it gives \(r\ge7,s\ge11\). If \(r\ge8\), use \(D_4(r,s)\ge D_4(8,12)>d_{\square}\). If \(r=7\), the finite-transfer bound and \cref{lem:singleton-lower} give
\[
e^{-4c_{7,s}}d_{\square}
<e^{-6/2401}d_{\square}
<\frac{2401}{2407}d_{\square}<D_4(7,11).
\]
The exponential bounds use \(e^u>1+u\). Every nonstrip configuration is therefore strictly worse. Since \(k\ge2\) and \(k\le s-r\) imply \(r<s\), \cref{prop:orientation} excludes the short-side strip as well. All long-side placements are lattice-isometric.
\end{proof}

\subsection{Noncanonical hulls}
\label{subsec:noncanonical-hulls}
The uniform comparison below first excludes \(T\ge4\); the remaining families with \(T\le3\) are treated according to their corner defects. Use the selector parameters \(m,e,T\), and write the source and canonical hulls as
\[
(r,s)=(m-t,m+e+t),
\qquad (R,S)=(m-T,m+e+T),
\qquad t<T.
\]
If their deficits are \(k\) and \(j\), respectively, then
\[
k-j=T(T+e)-t(t+e),\qquad
m\ge T(T+e)+j+1,\qquad 0\le j\le2T+e.
\]
For \(R\ge4\), \cref{lem:normalized-rectangle-error} absorbs the rectangular remainder into the finite corner correction. Thus the canonical strip beats a source with corner defects \(M_i\) whenever
\begin{equation}
E_j+D\left(\frac{2T+e}{m-T}\right)
-D\left(\frac{2t+e}{m-t}\right)
+B_{m-T,m+e+T}(j)
<\sum_i\Energy(M_i).
\label{eq:normalized-corner-test}
\end{equation}
Here \(D\) is the aspect bound from \cref{lem:aspect-cost}, and the nonpositive term \(-jc_{r,s}\) has been dropped. At fixed \(T,e,t,j\), the left side decreases with \(m\): the aspect difference decreases by \cref{lem:aspect-cost}, and the flux term decreases as both sides grow.

\begin{lemma}
\label{lem:large-imbalance-exclusion}
If a noncanonical minimum-side-sum hull has \(T\ge4\), every four-corner configuration on it with several nonempty corners or a single nonstrip defect has fewer spanning trees than the canonical strip of the same cardinality.
\end{lemma}

\begin{proof}
Drop the nonnegative source term \(D((2t+e)/(m-t))\) from \eqref{eq:normalized-corner-test}. Admissibility gives
\[
R\ge T^2-T+j+1,\qquad S\ge T^2+T+j+1,
\qquad
\frac{2T+e}{R}\le\frac{2T}{T^2-T+j+1},
\]
and \(k\ge2T+e-1+j\). The aspect inequality for \(e=1\) follows from \(j+1\le T^2+T\).

For \(T=4\), one has \(0\le j\le9\). By \cref{lem:flux-monotonicity},
\[
B_{R,S}(j)\le B_{13+j,21+j}(j)\le B_{22,30}(9).
\]
For \(j\le1\), the cost is at most \(E_1+D(8/13)+B_{22,30}(9)<.058968<b(8)\); here \(k\ge7\) and \(b(k)\ge b(8)\). For \(2\le j\le3\), use \(E_3+D(8/15)+B_{22,30}(9)<.069851<b(10)\), while \(k\ge9\). For \(4\le j\le9\), the affine bound in \cref{lem:strip-supporting-lines} gives the convex majorant
\[
\frac{j+27}{600}+D\left(\frac8{13+j}\right)+B_{22,30}(9).
\]
Its endpoint values are less than \(.07121\) and \(.07367\), both below \(.07421=b(k)\) for \(k\ge11\).

For \(T\ge5\), one has \(k\ge9\) and hence \(b(k)\ge b(10)=.07421\). Write \(B_T=8\pi^3/[45T^4(T+1)]\) for \eqref{eq:flux-imbalance}. At \(j\le11\), both bounds decrease with \(T\), giving \((2T+e)/R\le10/(21+j)\) and \(B_{R,S}(j)<B_5\). For \(j\le3\), use
\[
E_3+D(10/21)+B_5<.065772.
\]
On \(4\le j\le11\), the convex expression \((j+27)/600+D(10/(21+j))+B_5\) has endpoint values below \(.06670\) and \(.07403\). Finally, \(j\ge12\) forces \(T\ge6\); then \((2T+e)/R\le12/43\), \(B_{R,S}(j)<B_6\), and
\[
C_{\mathrm{strip}}+D(12/43)+B_6<.073388.
\]
Every bound is below \(.07421\). To verify the eight scalar inequalities, combine the displayed energy and flux bounds with six terms of the positive logarithm series. Each aspect argument is at most \(4/3\), so the series parameter is at most \(1/4\) and the error is below \(2.5\cdot10^{-9}\). Hence \eqref{eq:normalized-corner-test} excludes every nonstrip for \(T\ge4\).
\end{proof}

\begin{lemma}
\label{lem:multicorner-uniform}
Every four-corner domain with at least two nonempty defects on a minimum-side-sum hull, and with \(N\ge18\) vertices, has fewer spanning trees than the canonical strip of the same cardinality.
\end{lemma}

\begin{proof}
By \cref{prop:canonical-concentration,lem:large-imbalance-exclusion}, only noncanonical hulls with \(T\le3\) remain. The two-corner fold removes several strip pairs. For the remaining patterns, a monotonicity argument reduces the corner-energy comparison to the least admissible side length.

For a source \(R_{r,r+a}\) with \(k\) vacancies, write its target as \(R_{r-p,r+a+p}\), where \(p\ge1\) is the integer change in the shorter side and \(j\) is the target strip length. Put \(g=k-j=p(a+p)\). The refined rectangle estimate gives the comparison cost
\begin{equation}
\mathcal C_{a,k}(r)
=E_{r-p,r+a+p}(I_j^{\mathsf T})
+D\left(\frac{a+2p}{r-p}\right)-D\left(\frac a r\right)
+P(r-p,r;g),
\label{eq:multicorner-family-cost}
\end{equation}
where
\[
P(R,r;g)=\frac{g}{2R^2r^2}+\frac1{4r^4}
+\begin{cases}
1/(3R^4),&R=3,\\
1/(4R^4),&R\ge4.
\end{cases}
\]
Here the bulk and boundary terms have cancelled. By Hadamard--Fischer, the target strictly wins whenever the sum of the finite corner energies exceeds \(\mathcal C_{a,k}(r)\). The rectangle estimate applies to every target below: each has \(R\ge3\) and aspect ratio at most three, with \(r\ge4\) when \(R=3\).

First consider all \(T=1\) families, together with \(T=2,e=0\) and \(j\le2\). Their least side lengths, including the constraint \(N\ge18\), are
\begin{center}
\begin{tabular}{@{}cccc@{}}
\toprule
Side gap \(a\)&Deficit \(k\)&Least side \(r_0\)&Target data \((p,j)\)\\
\midrule
0&2,3&5&\((1,k-1)\)\\
0&4,5,6&\(k+1\)&\((2,k-4)\)\\
1&2,3&\(k+2\)&\((1,k-2)\)\\
1&4,5&\(k+1\)&\((1,k-2)\)\\
2&3,4,5&\(k+1\)&\((1,k-3)\)\\
\bottomrule
\end{tabular}
\end{center}
If one of two strips has length \(a+2\), \cref{prop:two-corner-fold} combines them on a more elongated hull; first separate them by \cref{prop:separate-corners} if their corners are adjacent. The improvement is strict, and the resulting strip is no better than the canonical one by \cref{thm:explicit-strip}. This removes \(1+2,2+2,2+3,2+4\) on square hulls, \(1+3,2+3\) at gap one, and \(1+4\) at gap two.

Each remaining pattern has finite corner energy at least \(q+\ell(r)\), with the two terms listed below. For each row, it suffices to show \(\mathcal C_{a,k}(r)-\ell(r)<q\). Here \(b+c\) denotes two strips of those lengths, ``other'' excludes the listed and folded patterns, and \(w_i=(r-i)^{-4}\).
\begin{longtable}{@{}ccll@{}}
\toprule
\((a,k)\)&Pattern&\(q\)&\(\ell(r)\)\\
\midrule
\endfirsthead
\toprule
\((a,k)\)&Pattern&\(q\)&\(\ell(r)\)\\
\midrule
\endhead
\bottomrule
\endfoot
\((0,2)\)&all&\(2E_1\)&\(24w_0/5\)\\
\((1,2)\)&all&\(2E_1\)&\(10w_0/3\)\\
\((0,3)\)&three corners&\(3E_1\)&0\\
\((1,3)\)&\(1+2\)&\(E_1+E_2\)&\(83w_0/12\)\\*
&three corners&\(3E_1\)&\(5w_0\)\\
\((2,3)\)&all&\(E_1+E_2\)&\(5w_0\)\\
\addlinespace
\((0,4)\)&\(1+3\)&\(E_1+E_3\)&\(24w_0/5+5(w_1+w_2)/3\)\\*
&other&\(E_1+\Energy(L)\)&\(48w_0/5\)\\
\((1,4)\)&\(2+2\)&\(2E_2\)&\(21w_0/2\)\\*
&other&\(3E_1\)&\(20w_0/3\)\\
\((2,4)\)&all&\(E_1+E_3\)&\(559w_0/60\)\\
\((0,5)\)&\(1+4\)&\(E_1+E_4\)&\(153w_0/20+5w_2/3+3w_3/4\)\\*
&\((2\times2)+1\)&\(E_1+\Energy(2\times2)\)&\(112w_0/5\)\\*
&other&\(E_1+\Energy(L)\)&\(12w_0\)\\
\((1,5)\)&\(1+4\)&\(E_1+E_4\)&\(73w_0/6\)\\*
&other&\(E_1+\Energy(2\times2)\)&\(25w_0/3\)\\
\((2,5)\)&all&\(E_1+\Energy(2\times2)\)&0\\
\addlinespace
\((0,6)\)&\(1+5\)&\(E_1+E_5\)&\(36w_0/5+5(w_1+w_2+w_3)/3\)\\*
&\(3+3\)&\(2E_3\)&\(72w_0/5\)\\*
&other&\(E_2+\Energy(2\times2)\)&0\\
\end{longtable}
The individual strip terms use the long-side energy by \cref{prop:orientation}. The singleton bounds come from \cref{lem:singleton-lower}; \cref{lem:short-defect-corrections} gives \(21/(4r^4)\) for a two-cell strip and \(20/r^4\) for a corner square.

Finite transfer extends these estimates to longer strips. For example, a square-hull \(1+4\) split has correction at least
\[
\frac{12}{5r^4}+\frac{21}{4r^4}
+c_{r,r-2}+c_{r,r-3}
>\frac{153}{20}w_0+\frac53w_2+\frac34w_3.
\]
For a gap-one \(1+4\) split, convexity of the finite strip correction in \cref{lem:finite-transfer} gives
\[
E_{r,r+1}(I_4^{\mathsf T})-E_4
\ge2\bigl(E_{r,r+1}(I_2^{\mathsf T})-E_2\bigr).
\]
Adding the singleton gives \((21/2+5/3)w_0=73w_0/6\). For the square-hull \(1+5\) row, retain the first four cells of the larger strip, and apply the square singleton bound to its last cell and the other corner. This gives \(36w_0/5+5(w_1+w_2+w_3)/3\).

At gap two and size four, the \(1+3\) correction is at least
\[
\left(\frac53+\frac{21}4+\frac{12}5\right)w_0
=\frac{559}{60}w_0.
\]
Concavity and two strip corrections give a larger bound for \(2+2\). The other patterns satisfy the same lower bound because \(\Energy(L)-E_3>559/(60\cdot5^4)\).

The remaining quadrant bounds follow from strip concavity, domination by \(L\) or \(2\times2\), and the refined envelope. With at least three corners, use \(2E_1+E_{k-2}\). The fixed comparisons
\begin{align*}
E_1+E_2&>\Energy(L),&2E_1&>\Energy(2\times2),\\
E_2+E_3&>E_1+\Energy(2\times2),&
E_1+b(5)&>E_2+\Energy(2\times2)
\end{align*}
give the common bounds for several patterns. At gap two and size five, for example, folding removes \(1+4\); each remaining case---\(2+3\), a nonstrip corner, or at least three corners---has quadrant loss at least \(E_1+\Energy(2\times2)\).

To reduce the comparison to \(r=r_0\), we show that \(\mathcal C_{a,k}(r)-\ell(r)\) is nonincreasing. Write the aspect term in \eqref{eq:multicorner-family-cost} as
\[
f_{a,p}(r)=D\left(\frac{a+2p}{r-p}\right)-D(a/r).
\]
Differentiation gives
\[
-f_{a,p}'(r)\ge
\frac{2rg-ap^2}{4r^2(r-p)^2}
\ge\frac{g}{2r^3}.
\]
Thus \(f_{a,p}(r)-c_0/r^4\) is nonincreasing once \(r^2\ge8c_0/g\). The error terms and target energy also decrease. For families beginning at target side three, keep the valid coefficient \(1/3\) in \(P\) at all larger sides.

This criterion applies directly to the gap-two rows and to gap one with two vacancies. For the square two-vacancy row, use the sharper bound
\[
-f_{0,1}'(r)>\frac1{2r(r-1)^2}>\frac{96}{5r^5}
\qquad(r\ge5),
\]
where the second inequality follows at five and then from the increase of \(r^4/(r-1)^2\). For the other gap-one rows,
\[
r^5\bigl(-f_{1,1}'(r)\bigr)
=\frac{r^3(4r-1)}{4(r-1)^2}
+\frac{\pi-3}{12}r^3
\left(2+\frac6{r-1}+\frac3{(r-1)^2}\right)
\]
is increasing for \(r\ge4\). The inequalities
\[
-f_{1,1}'(5)>\frac{42}{5^5},\qquad
-f_{1,1}'(6)>\frac{146}{3\cdot6^5}
\]
therefore cover all their correction coefficients.

For the square rows with \(p=2\),
\[
-f_{0,2}'(r)=\frac{(\pi-3)r+6}{3r(r-2)^2}.
\]
Every term \(r(r-2)^2/(r-i)^5\), \(0\le i\le3\), decreases in the ranges of the tables: its logarithmic derivative is
\[
\frac1r+\frac2{r-2}-\frac5{r-i}
<\frac3{r-2}-\frac5r\le0\qquad(r\ge5).
\]
It therefore suffices to compare derivatives at the left endpoints, where
\[
-\ell'(5)<\frac2{45},\qquad
-\ell'(6)<\frac{19}{800},\qquad
-\ell'(7)<\frac2{175},
\]
for the applicable corrections. Each right side is below \(-f_{0,2}'\) at that endpoint; the middle case uses \(\pi>157/50\). Hence every \(\mathcal C_{a,k}(r)-\ell(r)\) is nonincreasing throughout its stated range.

It remains to verify the comparison at \(r=r_0\) for every row:
\[
\mathcal C_{a,k}(r_0)<q+\ell(r_0).
\]
For the one-cell targets in the two-vacancy square family and the gap-one three-vacancy family, use \(E_1+B_{r-p,r+a+p}(1)\) as the target-energy bound; its flux term also decreases. The remaining endpoint energies use \(D_2(4,6)\) for \((a,k)=(0,3)\), \(D_1(4,6)\) for \((0,5),(2,4)\), \(D_2(5,7)\) for \((0,6),(2,5)\), \(D_3(5,8)\) for \((1,5)\), and \(D_2(4,7)\) for \((1,4)\). The two families using \(D_1(4,6)\) have target \(4\times8\), and those using \(D_2(5,7)\) have target \(5\times9\); domain inclusion gives the stated energy bounds. The minor bounds and rational energy enclosures prove the endpoint inequalities; monotonicity completes these families.

The remaining ranges are \(T=2,e=0\) with \(j=3,4\), \(T=2,e=1\), and \(T=3\). For \(T=2,e=0,j=3\), one has \(m\ge8\). The square singleton bound at \(t=0\) and the general bound at \(t=1\) give the endpoint comparisons
\[
\begin{aligned}
E_{6,10}(I_3^{\mathsf T})+D(2/3)+P(6,8;4)
-\frac{84}{5\cdot8^4}&<E_1+E_6,\\
E_{6,10}(I_3^{\mathsf T})+D(2/3)-D(2/7)+P(6,7;3)
-\frac{10}{7^4}&<E_1+E_5.
\end{aligned}
\]
The derivative conditions are \(8^2>168/5\) and \(7^2>80/3\). For \(j=4\), admissibility gives \(m\ge9\); discarding the source aspect, using the smaller source side in \(P\), and retaining \(kc_{r,s}\ge35/(3m^4)\) gives
\[
E_{7,11}(I_4^{\mathsf T})+D(4/7)+P(7,8;4)
-\frac{35}{3\cdot9^4}<E_1+E_6.
\]
Its derivative condition is \(m^2\ge70/3\), so it covers both source imbalances throughout \(m\ge9\).

For \(T=2,e=1\), one has \(4\le g\le6\), \(j\le5\), and \(m\ge j+7\). Both source imbalances are bounded by
\[
E_j+D\left(\frac5{m-2}\right)-D(1/m)
+\frac6{(m-2)^4}+B_{m-2,m+3}(j)
-\frac{5(j+4)}{3m^4}.
\]
At \(m=j+7\), this expression is below \(E_1+E_3\) for \(j=0\) and below \(E_1+E_4\) for \(1\le j\le5\). The derivative condition \(m^2\ge20(j+4)/9\) holds at each endpoint, so the six comparisons cover every larger \(m\). Since \(k-1\ge j+3\), the source loss is larger throughout these families.

For \(T=3\), use \eqref{eq:normalized-corner-test} and discard the source aspect. For \(e=0\), \(R\ge j+7\), \(S\ge j+13\), \(j\le6\), and \(g\ge5\). For \(e=1\), \(R\ge j+10\), \(S\ge j+17\), \(j\le7\), and \(g\ge6\). For \(j\le3\), the aspect and flux bounds for \(e=1\) are no larger than those for \(e=0\). Four comparisons suffice:
\[
\begin{aligned}
D(6/7)&<E_1+E_4,\\
D(3/4)+B_{8,14}(1)&<E_5,\\
E_3+D(2/3)+B_{10,16}(3)&<E_1+E_6,\\
D(6/11)+B_{13,19}(6)&<E_1.
\end{aligned}
\]
The first two apply at \(j=0,1\); the third covers \(j=2,3\) by flux monotonicity. At \(j\ge4\), the aspect is at most \(6/11\), and the flux endpoint \(B_{17,24}(7)\) for \(e=1\) is below \(B_{13,19}(6)\). The last bound therefore puts the cost below \(E_j+E_1\le E_{k-1}+E_1\).

Every remaining configuration with \(T\le3\) admits either a strict energy comparison or a strict fold. Together with canonical concentration and \cref{lem:large-imbalance-exclusion}, this proves the claim for all minimum-side-sum hulls.
\end{proof}

\subsection{Rectangular corner notches}
\label{subsec:rectangular-notch-family}
For the two rectangular notches needed below, the hull's side gap equals the longer defect dimension. Their finite corrections offset the cost of moving to a more elongated hull.

\begin{lemma}
\label{lem:rectangular-notch-family}
Let \(b\in\{2,3\}\) and \(r\ge3b-1\) be integers. Deleting a \(2\times b\) corner rectangle from \(R_{r,r+b}\), in either orientation, gives fewer spanning trees than \(C(r-1,r+b+1,b-1)\).
\end{lemma}

\begin{proof}
Both domains have \(r(r+b)-2b\) vertices and hull side sum \(2r+b\). The aspect and rectangle bounds in \cref{lem:aspect-cost,lem:coupled-rectangle-error} are
\[
A_b(r)=\frac{\pi(2r+b)}{12r(r-1)}-\frac12\log\frac r{r-1},
\qquad
P_b(r)=\frac{b+1}{2r^2(r-1)^2}+\frac1{4r^4}+\frac1{4(r-1)^4}.
\]
Both the target strip energy and \(P_b(r)\) decrease with \(r\). To reduce the comparison to the least side length, it remains to show that the aspect cost minus the source correction also decreases.

For \(b=2\), \cref{lem:short-defect-corrections} gives source excess loss at least \(\Energy(2\times2)+20/r^4\). The target singleton energy is at most \(E_{4,6}(I_1)\). Since \(\pi>25/8\),
\[
-A_2'(r)>\frac{r^2+74r-25}{48r^2(r-1)^2}.
\]
At six, this lower bound is \(455/43200>80/6^5\). The ratio of \(80/r^5\) to the displayed bound is \(3840(r-1)^2/[r^3(r^2+74r-25)]\), which decreases for \(r\ge3\). Thus \(A_2(r)-20/r^4\) decreases for \(r\ge6\). At the two integer endpoints \(r=5,6\),
\[
E_{4,6}(I_1)+A_2(r)+P_2(r)-\frac{20}{r^4}
<.055963<\Energy(2\times2).
\]
This proves the square-notch comparison for every \(r\ge5\).

For \(b=3\), the defect contains a \(2\times2\) square. Applying \cref{lem:finite-transfer} to the two additional cells and using the same trace bound gives source excess loss at least \(\Energy(2\times3)+70/(3r^4)\). The target strip energy is at most \(E_2+B_{r-1,r+4}(2)\). Now \(\pi>3\) gives
\[
-A_3'(r)>\frac{2r-3/4}{r^2(r-1)^2}>\frac{280}{3r^5}
\qquad(r\ge8),
\]
so \(A_3(r)-70/(3r^4)\) decreases on the required range. Its endpoint comparison is
\[
E_2+B_{7,12}(2)+A_3(8)+P_3(8)-\frac{70}{3\cdot8^4}
<.062375<\Energy(2\times3).
\]
The flux bound also decreases with \(r\), so the comparison holds for every \(r\ge8\).
\end{proof}

\subsection{Single-corner domains}
\label{subsec:single-corner-families}
\begin{lemma}
\label{lem:single-corner-reduction}
Suppose a minimum-side-sum hull \(R_{r,s}\), with \(r\le s\), carries one nonstrip Ferrers defect of size \(k<r\), and let \(N=rs-k\ge18\). If the hull is noncanonical, the domain either admits a strict tree-count improvement or admits a domain of the same cardinality with two nonempty opposite corner strips on a more elongated hull and at least as many spanning trees.
\end{lemma}

\begin{proof}
If the defect height exceeds \(s-r\), \cref{prop:ferrers-corner-fold} gives a strict improvement unless the defect is a hook on a square hull, which folds to two opposite strips. By \cref{lem:large-imbalance-exclusion}, it remains to exclude shallow defects, whose height is at most \(s-r\), on hulls with \(T\le3\).

For a two-row Ferrers defect, \(M\setminus M^{\mathsf T}\subset\{x>y\}\), so the strip-orientation polarization favors the transpose. If the width exceeds the hull side gap, that transpose folds strictly. This handles every shallow source of gap two except \(L\) and the \(2\times2\) square.

For \(T\le2\), sources with \(t=0\) have side gap zero or one and already fold. At \((T,e,t)=(2,0,1)\), \cref{lem:rectangular-notch-family} handles the square notch. The remaining \(L\) family lies on \((r,r+2)\), with \(r\ge4\), and has a complete canonical target. The cellwise bound in \cref{lem:finite-transfer} gives source correction at least \((5/3)(2/r^4+1/(r-1)^4)\), leaving comparison cost at most
\[
\begin{aligned}
&D\left(\frac4{r-1}\right)-D\left(\frac2r\right)
+\frac3{2r^2(r-1)^2}+\frac1{4r^4}+\frac1{3(r-1)^4}\\
&\qquad-\frac53\left(\frac2{r^4}+\frac1{(r-1)^4}\right).
\end{aligned}
\]
The coefficient \(1/3\) covers target side three in \cref{lem:coupled-rectangle-error}. The negative derivative of the aspect difference is greater than \((r^2+74r-25)/[48r^2(r-1)^2]\). At four, this exceeds the magnitude of the correction derivative, since
\[
\frac{287}{6912}>\frac{20}{3}\left(\frac2{4^5}+\frac1{3^5}\right).
\]
Each correction-derivative term divided by that aspect bound decreases for \(r\ge4\); the two ratios are positive multiples of \((r-1)^2/[r^3(r^2+74r-25)]\) and \(r^2/[(r-1)^3(r^2+74r-25)]\). Thus the whole cost decreases. At four it is less than \(.056330<\Energy(L)\), which excludes the entire family.

At \((T,e,t)=(2,1,1)\), the side gap is three, \(k=j+4\), \(0\le j\le5\), and \(r\ge j+6\). At these least side lengths, the respective upper costs from \eqref{eq:normalized-corner-test} are
\[
.039792,\quad .060872,\quad .067264,\quad
.070288,\quad .072150,\quad .073488.
\]
The first is below \(\Energy(2\times2)\), and the other five are below \(.07421\), covering all nonrectangles and the \(3\times3\) rectangle. The \(2\times4\) rectangle transposes to height four and folds strictly. The remaining \(2\times3\) rectangle is covered by \cref{lem:rectangular-notch-family}. Monotonicity in \(r\) proves each comparison throughout its family.

It remains to treat \(T=3\). For \(e=0\), \(t=0\) folds, and a shallow source at \(t=1\) has two rows and size at least eight, so its transpose folds. At \(t=2\), one has \(k=5+j\), \(0\le j\le6\), and \(m\ge10+j\). Over a strip-length range \(a\le j\le b\), monotonicity in \(m\), the aspect difference, and \cref{lem:flux-monotonicity} give the common upper cost
\[
E_b+D\left(\frac6{a+7}\right)-D\left(\frac4{a+8}\right)
+B_{b+7,b+13}(b).
\]
For \(e=1\), \(t=0\) folds. For \(t=1,2\), retain the larger cost at \(t=1\) and use \(m\ge13+j\); the corresponding upper bound is
\[
E_b+D\left(\frac7{a+10}\right)-D\left(\frac3{a+12}\right)
+B_{b+10,b+17}(b).
\]
In the last range with \(e=1\), concavity bounds \(E_7\le2E_6-E_5\), so the estimates use only strip energies through length six. The six comparisons are
\begin{center}
\begin{tabular}{cccc}
\toprule
\(e\) & Strip length & Upper cost & Corner lower bound\\
\midrule
\(0\) & \(0\le j\le1\) & .0574 & .06522\\
\(0\) & \(2\le j\le3\) & .0662 & .07054\\
\(0\) & \(4\le j\le6\) & .0711 & .07421\\
\(1\) & \(0\le j\le1\) & .0569 & .06522\\
\(1\) & \(2\le j\le3\) & .0669 & .07054\\
\(1\) & \(4\le j\le7\) & .0732 & .07421\\
\bottomrule
\end{tabular}
\end{center}
The rightmost entries are the least values of \(b(k)\) over the corresponding deficits: \(k=5+j\) for \((e,t)=(0,2)\), \(k=10+j\) for \((e,t)=(1,1)\), and \(k=6+j\) for \((e,t)=(1,2)\). Monotonicity extends these strict comparisons to every admissible \(m\). All shallow defects are therefore excluded. A hook on a square hull folds to two opposite strips with at least as many spanning trees, giving the second alternative of the lemma.
\end{proof}

\subsection{The exceptional hulls}
\label{subsec:finite-corner-comparisons}
The remaining small cases admit a direct calculation with paths and four-cycles. If \(N\le17\) and the deficit satisfies \(k\ge2\), then \(k<r\le s\) and \(r+s=\lceil2\sqrt N\rceil\le9\). Thus \(r=3\) or \(4\). Applying \((r+s-1)^2<4(rs-k)\) leaves exactly
\[
(r,s,k)=(3,3,2),\ (3,4,2),\ (3,5,2),\ (4,4,2),\ (4,4,3),\ (4,5,3).
\]
Deficits zero and one already give strips. With two vacancies, a nonstrip deletes two corner singletons; \cref{prop:separate-corners} forces them to opposite corners. With three vacancies, the nonstrip types are an \(L\), a two-cell strip and a singleton, or three corner singletons. The same proposition strictly excludes adjacent placements of the strip and singleton. Up to isometry, this leaves three cases on \(4\times4\) and four on \(4\times5\), where the opposite strip has two orientations.

For a face graph \(G\), put \(p_G(t)=\det(tI-A_G)\); its primal tree count is \(p_G(4)\). The continuants \(f_j\) from \eqref{eq:continuants} satisfy \(p_{P_j}=f_j\). If \(v\) is a leaf adjacent to \(u\), determinant expansion gives \(p_G=t p_{G-v}-p_{G-\{u,v\}}\). Identifying one vertex of \(G\) and \(J\), denoted by \(v\) in both graphs, gives
\[
p_{G\vee J}=p_Gp_{J-v}+p_{G-v}p_J-tp_{G-v}p_{J-v}.
\]
This identity follows by taking Schur complements at the common vertex. For the four-cycle and the \(2\times j\) ladder, respectively, write
\[
c(t)=t^4-4t^2,\qquad l_j(t)=f_j(t-1)f_j(t+1).
\]
Symmetric and antisymmetric combinations of the two ladder rows give the second identity. If \(q_j\) is its polynomial after deleting one corner, leaf deletion gives \(q_1=t\) and \(q_j=t l_{j-1}-q_{j-1}\).

The opposite-singleton configurations on \(3\times3\) and \(3\times4\) have face graphs consisting of two isolated vertices and \(P_4\), respectively, giving counts 16 and 209. The canonical strips have face graphs \(P_2\) and \(P_4\), giving 15 and 209.

On the \(4\times4\) primal hull, the three remaining face graphs are a four-cycle with a two-vertex pendant path, with two leaves at opposite vertices, or with two leaves at the same vertex. Leaf deletion gives
\begin{center}
\begin{tabular}{lcr}
\toprule
Defects & \(p_G(t)\) & \(p_G(4)\)\\
\midrule
opposite \(2+1\) & \(f_2c-tf_3\) & 2656\\
\(L\) & \(t^2c-2tf_3+t^2\) & 2640\\
three singletons & \(t^2c-2tf_3\) & 2624\\
\bottomrule
\end{tabular}
\end{center}
Successive rows differ by \(t^2\). The canonical 13-vertex strip has the first face graph, so the opposite \(2+1\) configuration ties it. The opposite-singleton configuration on \(3\times5\) has the second face graph and is strictly worse.

For two opposite singletons on \(4\times4\), the face graph is two four-cycles joined at one vertex. The canonical 14-vertex strip has face polynomial \(q_4\). The gluing identity and the ladder recurrence give
\[
2cf_3-tf_3^2=t^7-8t^5+12t^3,
\qquad q_4=t^7-8t^5+12t^3-3t.
\]
The difference is \(3t\); at \(t=4\), the opposite configuration has count 8960, exceeding the canonical count 8948 by twelve.

For \(N=17\), the opposite long-side \(2+1\) configuration joins a four-cycle to an end corner of the \(2\times3\) ladder, as shown in \cref{fig:small-face-comparisons}. Its polynomial is
\[
w(t)=l_3f_3+q_3(c-tf_3).
\]
The other cases share the eight-vertex core \(H=\{0,1,2\}^2\setminus\{(0,0)\}\). After translation and reflection, the \(L\) and opposite short-side configurations each attach a leaf at \(a=(2,0)\), and the three-singleton configuration attaches one at \(b=(1,2)\).

\begin{figure}[htbp]
\centering
\begin{tikzpicture}[x=.65cm,y=.65cm,font=\small]
  \begin{scope}[shift={(0,.5)}]
    \draw[thick] (1,0)--(2,0)--(2,2)--(0,2)--(0,1)--(2,1);
    \draw[thick] (1,0)--(1,2);
    \foreach \x/\y in {1/0,2/0,0/1,1/1,2/1,0/2,1/2,2/2}
      \fill (\x,\y) circle (1.6pt);
    \draw[gray] (-.08,-.08)--(.08,.08) (-.08,.08)--(.08,-.08);
    \fill[teal!70!black] (2,0) circle (2.5pt);
    \node[below,teal!70!black] at (2,0) {\(a\)};
    \fill[orange!85!black] (1,2) circle (2.5pt);
    \node[above,orange!85!black] at (1,2) {\(b\)};
  \end{scope}
  \node[below] at (1,-.2) {common core \(H\)};
  \begin{scope}[shift={(5,0)}]
    \draw[teal!70!black,thick] (1,0)--(2,0)--(2,2)--(1,2)--(1,0);
    \draw[teal!70!black,thick] (1,1)--(2,1);
    \draw[orange!85!black,thick] (1,2)--(0,2)--(0,3)--(1,3)--(1,2);
    \foreach \x/\y in {1/0,2/0,1/1,2/1,2/2}
      \fill[teal!70!black] (\x,\y) circle (1.6pt);
    \foreach \x/\y in {0/2,0/3,1/3}
      \fill[orange!85!black] (\x,\y) circle (1.6pt);
    \fill (1,2) circle (2.1pt);
  \end{scope}
  \node[below] at (6,-.2) {maximizing face graph};
\end{tikzpicture}
\caption{The remaining comparisons at \(N=17\). Attaching a leaf at \(a\) or \(b\) of \(H\) gives the two competing face graphs. In the maximizer, a \(2\times3\) ladder and a four-cycle share one vertex; their colors distinguish the two pieces.}
\label{fig:small-face-comparisons}
\end{figure}

Reflection across \(x=y\) splits \(A_H\) into an antisymmetric copy of \(A_{P_3}\) and a symmetric bipartite block. Using normalized sums on the reflected pairs, the latter has off-diagonal block
\[
B=\begin{pmatrix}1&\sqrt2&0\\1&\sqrt2&\sqrt2\end{pmatrix},
\qquad
p_H(t)=tf_3(t)\det\begin{pmatrix}t^2-3&-3\\-3&t^2-5\end{pmatrix}.
\]
Deleting \(a\) from \(H\) leaves a \(2\times3\) ladder with a leaf at the middle of one long side; deleting that attachment vertex from the ladder leaves \(P_5\). Deleting \(b\) from \(H\) leaves a four-cycle with a two-vertex pendant path and a leaf at adjacent cycle vertices. Hence the two alternatives have polynomials
\begin{align*}
u(t)&=tp_H-(tl_3-f_5),\\
v(t)&=tp_H-\bigl(tf_2c-(t^2+f_2)f_3+tf_2\bigr).
\end{align*}
The canonical 17-vertex strip has face polynomial \(q_5\). The recurrences reduce the comparisons to
\[
w-u=t(t^4+4t^2-2),\qquad
u-v=t(t^4-t^2+1),\qquad
w-q_5=3t^3-2t.
\]
All three differences are positive at four. To obtain the exact counts, use \(f_3(4)=56\), \(f_5(4)=780\), \(c(4)=192\), \(l_3(4)=21\cdot115=2415\), and \(l_4(4)=55\cdot551=30305\). These give \(q_3(4)=712\), \(q_4(4)=8948\), \(q_5(4)=112272\), and \(p_H(4)=4\cdot56(13\cdot11-9)=30016\), hence
\[
w(4)=112456,\qquad u(4)=111184,\qquad v(4)=110220.
\]
The long-side opposite configuration is therefore the unique maximizing isometry class at 17, exceeding the canonical count by 184.

These comparisons prove the five nonstrip cases in \cref{lem:finite-classification}. The single-strip equalities in \eqref{eq:adjacent-equality-list} add maximizers at 3, 8, and 10; their face graphs are empty, \(P_3\), and \(P_4\), with counts 1, 56, and 209. The strip ties at 7, 14, and 17 are strictly beaten above. Thus nonuniqueness occurs exactly at \(3,8,10,13\).

\end{document}